\documentclass{amsart}
\usepackage{mathrsfs}
\usepackage{graphicx}
\usepackage{amsmath}
\usepackage{amscd}
\usepackage{amssymb}

\input xy
\xyoption{all}

\newcommand{\pd}{\partial}
\newcommand{\bC}{{\mathbb C}}
 \newcommand{\bH}{{\mathbb H}}

\newcommand{\bP}{{\mathbb P}}
\newcommand{\bQ}{{\mathbb Q}}
\newcommand{\bR}{{\mathbb R}} \newcommand{\bV}{{\mathbb V}}
\newcommand{\bZ}{{\mathbb Z}}
\newcommand{\cA}{{\mathcal A}}
\newcommand{\cB}{{\mathcal B}}

\newcommand{\cF}{{\mathcal F}}

\newcommand{\cO}{{\mathcal O}}
\newcommand{\cP}{{\mathcal P}}
 
\newcommand{\half}{\frac{1}{2}}
\newcommand{\cV}{{\mathcal V}}

\newcommand{\mB}{\mathfrak{B}}
\newcommand{\vac}{|0\rangle}
\newcommand{\lvac}{\langle 0|}

 \DeclareMathOperator{\End}{End}
 
\DeclareMathOperator{\ch}{ch} \DeclareMathOperator{\id}{id}
 \DeclareMathOperator{\tr}{tr}
 \DeclareMathOperator{\Res}{Res}

\newtheorem{theorem}{Theorem}[section]
\newtheorem{theorem/definition}{Theorem/Definition}[section]

\newtheorem{prop}{Proposition}[section]

\theoremstyle{remark}

\theoremstyle{definition}
 
\newtheorem{definition}{Definition}[section]

\newcommand{\be}{\begin{equation}}
\newcommand{\ee}{\end{equation}}
\newcommand{\bea}{\begin{eqnarray}}
\newcommand{\ben}{\begin{eqnarray*}}
\newcommand{\een}{\end{eqnarray*}}
\newcommand{\eea}{\end{eqnarray}}
\newcommand{\bet}{\begin{equation}
\begin{split}}
\newcommand{\eet}{\end{split}
\end{equation}}

\usepackage{color}

\definecolor{yellow}{rgb}{1,1,0}
\definecolor{orange}{rgb}{1,.7,0}
\definecolor{red}{rgb}{1,0,0} \definecolor{green}{rgb}{0,1,1}
\definecolor{white}{rgb}{1,1,1}

\definecolor{A}{rgb}{.75,1,.75}

\newcommand{\corr}[1]{\langle {#1} \rangle}

\newcommand{\Corr}[1]{\{ {#1} \}}
\newcommand{\COrr}[1]{\biggl\{ {#1} \biggr\}}
\newcommand{\Cors}[1]{[ {#1} ]}
\newcommand{\COrs}[1]{\biggl[ {#1} \biggr]}

\begin{document}

\title
{K-Theory of Hilbert Schemes as a Formal Quantum Field Theory}

\author{Jian Zhou}
\address{Department of Mathematical Sciences\\
Tsinghua University\\Beijing, 100084, China}
\email{jzhou@math.tsinghua.edu.cn}

\begin{abstract}
We define a notion of formal quantum field theory and associate a formal quantum field theory
to K-theoretical intersection theories on Hilbert schemes of points on algebraic surfaces.
This enables us to find an effective way to compute  K-theoretical intersection theories on Hilbert schemes
via a connection to Macdonald polynomials and vertex operators.
\end{abstract}
\maketitle

\section{Introduction}

There are some striking similarities between quantum field theory and algebraic topology.
In both of these theories one starts with some geometric objects
and constructs some algebraic structures associated with them.
In a quantum field theory,
one constructs a space of observables,
defines and computes\emph{} their correlation functions.
In an algebraic topological theory,
one constructs a space of cohomology groups or K-groups, etc.,
and defines and computes their intersection numbers.
Both admit some axiomatic characterizations,
and both seek for identifications of different constructions.
In mathematics,
there are various isomorphism theorems that identifies different cohomology theories,
e.g. the de Rham theorem.
In physics,
the identifications of different quantum field theories
are often referred to as duality by physicists.
This aspect of quantum field theory has produced many striking
results relating different branches of mathematics and thus has attracted
the attentions of many mathematicians.

Such similarities make some mathematicians
to regard quantum field theory as physicists' algebraic topology.
It might be fair also to regard algebraic topology as mathematicians' quantum field theory.
It has been very fruitful to borrow terminologies
and formalisms from quantum field theory.
See for example,
Atiyah's axioms for topological quantum field theories \cite{Atiyah}
and Kontsevich-Manin's axioms for Gromov-Witten theory \cite{Kon-Man}.

The construction of a quantum  field theory involves integrations
over infinite-dimensional spaces that often requires mathematical
justifications to become mathematically rigourous.
There are some quantum field theories that admit supersymmetries,
and physicists apply supersymmetric localizations
to reduce to integrals over finite-dimensional spaces.
The topological quantum field theories
are those supersymmetric quantum field theories
which reduce to intersection theories
on finite-dimensional moduli spaces in differential
geometry and algebraic geometry.
For example,
Donaldson theory \cite{Wit1},
intersection theory on moduli spaces of holomorphic vector bundles on Riemann surfaces \cite{Wit2},
and Gromov-Witten theory \cite{Wit3}.

If one reverses the direction of the above reduction from topological quantum field theory
to intersection theories on moduli spaces,
one can start with some intersection theories on some  spaces,
and seek for suitable topological   quantum field theory that leads to it.
Once such a geometric problem has been reformulated as a quantum field theory,
one can apply techniques developed in physics literature
to its study.
In particular,
one may find some method to calculate the intersection numbers
in a formalism familiar to physicists,
but mathematically difficult to find.

Concrete examples of this beautiful idea have been
provided by the mathematical work on Donaldson theory,
intersection theory on moduli spaces of holomorphic vector bundles on Riemann surfaces,
and Gromov-Witten theory.
In this paper we will prsent another example.
We will apply this idea to the K-theory of Hilbert schemes of points
on algebraic surfaces.
We will not actually construct a quantum field theory in its original sense
from a Lagrangian formulation as normally done by physicists,
but instead we will introduce a notion of {\em formal quantum field}
that preserves only the main features of a quantum field theory such as correlators, partition
functions, etc.
This formalism avoids the difficult problem of justifications
of the mathematical rigor of a physical quantum field theory,
but nevertheless it enables us to apply techniques from conformal field theory,
such as vertex operators on bosonic Fock space to  compute
the K-theoretical intersection numbers on Hilbert schemes.

Our starting point is that by a result of Ellingsrud-G\"ottsche-Lehn \cite{EGL},
one can reduce to the case of Hilbert schemes of points on $\bC^2$.
In this case one can apply localization formula to compute the equivariant
K-theoretical intersection numbers
by a summation over partitions,
hence we are facing a problem in the same spirit of a problem in statistical physics.
To reformulate it as a formal quantum field theory,
one needs to combine two different connections to Macdonald polynomials
discovered by Haiman \cite{Haiman} and Wang-Zhou \cite{Wang-Zhou1} for different purposes.
Furthermore,
one needs to reformulate these connections in terms of operators on
the bosonic Fock space.
For this purpose,
we introduce a notion of {\em vertex realizable operators},
and develop a formalism for calculating their correlation.
So this work is a sequel to Wang-Zhou \cite{Wang-Zhou1} with
a broader perspective.
In particular, the main result there is used and generalized in this paper.

We arrange the rest of this paper in the following way.
In Section 2 we define the notions of formal quantum field theory and operatorial
formal quantum field theory.
In Section 3 we recall some relevant background results on Hilbert schemes
that serve as motivations to our work.
We introduce the main object of our study, the equivariant K-theoretical
intersection numbers of tautological sheaves on $(\bC^2)^{[n]}$, in Section 4,
and establish their relationship with Macdanald functions and vertex operators
in Section 5.
Next, in Section 6,
we introduce a notion of vertex realizable operators
and develop a formalism for their concrete computations.
Applications are made in Section 7 where we associate a formal quantum
field theory to K-theory of Hilbert schemes and some correlators are computed.
We present our concluding remarks on further questions in Section 8.

\section{Formal Quantum Field Theories}

In this Section we introduce a notion of a formal quantum field theory.
Using the bosonic Fock space
we present some examples of special formal quantum field theories.

\subsection{Formal quantum field theories}

By an {\em observable algebra} we mean a commutative algebra $\cO$ with identity $1$.
We assume $\cO$ is
generated by $\{\cO_i\}_{i\geq 0}$,
where $\cO_0 = 1$,
and
$$\cO_m \cO_n = \cO_n \cO_m,$$
for all $m, n\geq 0$.
Then the algebra $\cO$ consists of linear combinations of elements
of the form $\cO_{m_1} \cdots \cO_{m_n}$.
Elements in $\cO$ will be called {\em observables}.
Let $R$ be another commutative ring.
By a {\em formal quantum field theory} with observable algebra $\cO$ and  with values in $R$,
we mean a sequence of elements in $S^n(\cO^*) \otimes R$, one for each $n\geq 1$.
It is specified by the {\em correlators} $\corr{\cO_{m_1}, \dots, \cO_{m_n}} \in R$.
The {\em normalized correlators} $\corr{\cO_{m_1}, \dots, \cO_{m_n}}'$
are defined by
\be
\corr{\cO_{m_1}, \dots, \cO_{m_n}}':= \frac{\corr{\cO_{m_1}, \dots, \cO_{m_n}}}{\corr{\cO_0}}.
\ee
A convenient way to encode all the normalized correlators is to
introduce formal variables $\{t_0, t_1, \dots, t_n, \dots\}$,
and consider the generating series:
\be
Z:= 1+\sum_{n \geq 1} \corr{\cO_{m_1}, \dots, \cO_{m_n}}' \frac{t_{m_1} \cdots t_{m_n}}{n!}.
\ee
It is called the {\em partition function}. Note:
\be \label{eqn:corr-Z}
\corr{\cO_{m_1}, \dots, \cO_{m_n}}'
= \frac{\pd^n}{\pd t_{m_1} \cdots \pd t_{m_n}} Z|_{t_i=0, i \geq 0}.
\ee
The {\em free energy} $F$ is defined by:
\be \label{eqn:F-Z}
F:= \log Z.
\ee
The {\em connected correlators} are defined by:
\be \label{eqn:corr-F}
\corr{\cO_{m_1}, \dots, \cO_{m_n}}'_c
= \frac{\pd^n}{\pd t_{m_1} \cdots \pd t_{m_n}} F|_{t_i=0, i \geq 0}.
\ee
By \eqref{eqn:corr-Z} to \eqref{eqn:corr-F},
one can get:
\ben
&& \corr{\cO_{m_1}}'=\corr{\cO_{m_1}}_c', \\
&& \corr{\cO_{m_1}\cO_{m_2}}'=\corr{\cO_{m_1}\cO_{m_2}}_c'
+ \corr{\cO_{m_1}}_c'\corr{\cO_{m_2}}_c', \\
&& \corr{\cO_{m_1}\cO_{m_2}\cO_{m_3}}'=\corr{\cO_{m_1}\cO_{m_2}\cO_{m_2}}_c'
+ \corr{\cO_{m_1}\cO_{m_2}}_c'\corr{\cO_{m_3}}_c' \\
&& \;\;\;\;\;\;\;\;\; +\corr{\cO_{m_1}\cO_{m_3}}_c'\corr{\cO_{m_2}}_c'
+ \corr{\cO_{m_2}\cO_{m_3}}_c'\corr{\cO_{m_1}}_c'\\
&& \;\;\;\;\;\;\;\;\;  +\corr{\cO_{m_1}}_c'\corr{\cO_{m_2}}_c'\corr{\cO_{m_3}}_c',
\een
and conversely,
\ben
&& \corr{\cO_{m_1}}'_c=\corr{\cO_{m_1}}', \\
&& \corr{\cO_{m_1}\cO_{m_2}}'_c=\corr{\cO_{m_1}\cO_{m_2}}'
- \corr{\cO_{m_1}}'\corr{\cO_{m_2}}', \\
&& \corr{\cO_{m_1}\cO_{m_2}\cO_{m_3}}'_c=\corr{\cO_{m_1}\cO_{m_2}\cO_{m_2}}'
- \corr{\cO_{m_1}\cO_{m_2}}_c'\corr{\cO_{m_3}}' \\
&& \;\;\;\;\;\;\;\;\; -\corr{\cO_{m_1}\cO_{m_3}}'\corr{\cO_{m_2}}'
-\corr{\cO_{m_2}\cO_{m_3}}'\corr{\cO_{m_1}}'\\
&& \;\;\;\;\;\;\;\;\;  +2\corr{\cO_{m_1}}'\corr{\cO_{m_2}}'\corr{\cO_{m_3}}',
\een
etc.
These notations borrowed from quantum field theory
have been used in mathematical literature,
see e.g. Okounkov \cite{Okounkov}.
One can also define the {\em entropy} by
\be
G = \sum_{n\geq 0} t_n \frac{\pd F}{\pd t_n} - F.
\ee

\subsection{Operatorial formal quantum field theory}

Let $H$ be a space with a scalar product with values in the ring $R$,
following physicists,
a vector in $H$ will be denoted by $|v\rangle$.
For a linear operator $A: H \to H$,
the scalar product of $A|v\rangle$ with $|w\rangle$ will be denoted by:
$\langle w|A|v\rangle$.
Let $\cO_0, \cO_1, \dots, \cO_n, \dots$ be a sequence of operators on $H$,
such that $\cO_0 = \id_H$,
and
$$\cO_m \cO_n = \cO_n \cO_m,$$
for all $m, n\geq 0$.
By the {\em operator algebra} generated by $\{\cO_n\}_{n\geq 0}$
we mean the commutative algebra generated by these operators.
Fix two nonzero vectors $|v_0\rangle$ and $|w_0\rangle$ in $H$,
we define the correlators $\corr{\cO_{m_1}, \dots, \cO_{m_n}}$ is defined by:
\be
\corr{\cO_{m_1}, \dots, \cO_{m_n}}
:= \langle w_0 | \cO_{m_1} \cdots \cO_{m_n} |v_0\rangle.
\ee
This gives rise to a formal quantum field theory in the sense of last Subsection.
To summarize,
given a space $H$ with an inner product, an operator algebra $\cO$ generated by a commuting family of
operators  $\{\cO_n\}_{n\geq 0}$ on $H$ with $\cO_0 = \id_H$,
and two fixed vector $|v_0\rangle$ and $|w_0\rangle$,
one can associate  a formal quantum field theory
by defining correlators  as above.
We will refer to such a formal quantum field theory
as an {\em operatorial formal quantum field theory}.

\subsection{Special operatorial formal quantum field theories}

We now consider a special class of formal quantum field theory.
Denote by $\Lambda$ the space of symmetric functions.
Let $\{\alpha_n, \alpha_{-n}\}_{n \geq 1}$ be operators on $\Lambda$ defined by:
\begin{align}
\alpha_n &:= n \frac{\pd}{\pd p_n}, & \alpha_{-n} & : = p_n \cdot,
\end{align}
then the following commutation relations hold:
\be
[\alpha_m, \alpha_n]= m \delta_{m, -n}.
\ee
Denote by $\vac$ the function $1$ in $\Lambda$.
Then $\Lambda$ is spanned by $\alpha_{-\lambda} \vac= \alpha_{-\lambda_1}
\cdots \alpha_{-\lambda_l}\vac$,
where $\lambda=(\lambda_1, \dots, \lambda_l)$ runs through all partitions.
A scalar product on $\Lambda$ can be defined such that
\bea
&& \langle0|0\rangle = 1, \\
&& \alpha_n^*=\alpha_{-n}.
\eea
The space $\Lambda$ is called the {\em bosonic Fock space}.
A basic computation tool on $\Lambda$ is the following Weyl commutation relation:
\be \label{eqn:Weyl1}
\begin{split}
& \exp \sum_{m \geq 0} \frac{1}{m} a_m \alpha_m \cdot
\exp \sum_{n \geq 0} \frac{1}{n} b_n \alpha_{-n} \\
= & \exp \sum_{m \geq 0} \frac{1}{m} a_m b_m \cdot
\exp \sum_{n \geq 0} \frac{1}{n} b_n \alpha_{-n}  \cdot
\exp \sum_{m \geq 0} \frac{1}{m} a_m \alpha_m.
\end{split}
\ee

We will actually work with $\Lambda$ tensored with some coefficient field $F$,
denoted by $\Lambda_F$.
Fixing two vectors $|v_0\rangle = A\vac$, $|w_0\rangle= B\vac$ in $\Lambda_F$,
where $A$ and $B$ are two operators on $\Lambda_F$,
and a family of commuting family of operators $\{\cO_n\}_{n \geq 0}$ on $\Lambda_F$,
one can define a formal quantum field theory.
We will refer to it as a {\em special formal quantum field theory}.
In this case,
the partition function can be written  defined by vacuum expectation values:
\be
Z = \frac{\lvac B^* \cdot \exp \sum_{n \geq 0} t_n \cO_n \cdot A \vac}{\lvac B^*\cdot A \vac}.
\ee
Introduce an operator $L$ defined by:
\be
L: = B^* \cdot \exp \sum_{n \geq 0} t_n \cO_n \cdot A.
\ee
Then its evolution with respect to $t_n$ is given by:
\be
\frac{\pd}{\pd t_n} L = B^* \cdot \cO_n \cdot \exp \sum_{m \geq 0} t_m \cO_m \cdot A.
\ee
It can be regarded as the {\em Lax operator} in this setting if one can find a sequence
of operators $P_n$ such that
\be
[P_n, L] = B^* \cdot \cO_n \cdot \exp \sum_{m \geq 0} t_m \cO_m \cdot A.
\ee

Later in this paper,
we will consider special operatorial formal quantum field theories
in which $\cO_n$ can all be given using vertex operators.
They will be called {\em vertex operatorial formal quantum field theories}.

\section{Backgrounds on Hilbert Schemes and Motivations}

In this Section we recall some earlier results on Hilbert schemes
that motivate this work.

By a well-known theorem of Fogarty \cite{For} in 1968,
when $S$ is a nonsingular projective or quasi-projective
algebraic surface,
the Hilbert scheme $S^{[n]}$
is a nonsingular variety of dimension $2n$.
This result makes it possible to apply various methods to study
the topology of the Hilbert schemes.
Developments since the 1980s have revealed
very rich and deep connections with many other branches of mathematics,
e.g., modular forms, representation of infinite-dimensional algebras,
vertex operators, Boson-Fermion correspondence, symmetric functions,
and integrable hierarchies.
For standard reference to these materials,
see the book by Nakajima \cite{Nak-Lecture} and his more recent lecture notes \cite{Nak-More}.
Based on these earlier developments,
we will present in this paper a method to compute K-theoretical intersection
numbers and cohomological intersection numbers
on Hilbert schemes
based on some connections with the theory of Macdonald polynomials,
by applying the vertex operator realizations
of Macdonald operators.
To better motive the problem we want to solve,
it is instructive to review some of the earlier developments.
In 1987, Ellingsrud and Str{\o}mme \cite{Ell-Str1} calculated the Betti numbers of
the Hilbert schemes of points on the projective
plane, the affine plane, and rational ruled surfaces.
In 1990, G\"ottsche \cite{Got} computed the generating series of
Betti numbers hence also Euler numbers of $S^{[n]}$ by Weil Conjecture.
By G\"ottsche's formula,
the generating function of the Poincar\'e polynomials of $S^{[n]}$ is given by
\be
\sum_{n=0}^\infty P_t(S^{[n]}) q^n
= \prod_{m=1}^\infty
\frac{(1 + t^{2m-1}q^m)^{b_1(S)}(1 + t^{2m+1}q^m)^{b_3(S)}}
{(1 - t^{2m-2}q^m)^{b_0(S)}(1 - t^{2m}q^m)^{b_2(S)}(1 - t^{2m+2}q^m)^{b_4(S)}},
\ee
where $b_i(S)$ is the Betti number of $S$.
By taking $t=-1$,
one gets the generating series of Euler numbers:
\be
\sum_{n \geq 0} \chi(X^{[n]}) q^n = \prod_{n=1}^\infty
\frac{1}{(1-q^n)^{\chi(S)}},
\ee
and a manifest  connection with modular forms.
In 1993, G\"ottsche and Soergel \cite{Got-Soe} computed the generating series of
Hodge numbers,
and there is a similar formula for the generating series of Hodge polynomials of
$S^{[n]}$.
In 1994, Vafa and Witten \cite{Vafa-Witten} put the above formulas
in the context of $N = 4$ topologically twisted supersymmetric Yang-Mills on four-manifolds,
and conjectured  a relation with $2$-dimensional conformal field theory
and string theory.
In particular,
by seeing the orbifold cohomology of symmetric products $\oplus_{n\geq 0}
H^*_{orb}(X^n/S_n)$ as the bosonic Fock space,
an action of infinite-dimensional Heisenberg algebra on $\bH(X):=\oplus_{n\geq 0} H^*(X^{[n]})$
was conjectured.
Furthermore,
an action of Virasoro algebra was conjectured.

In 1995,
Nakajima \cite{Nak1} and Grojnowski \cite{Gro} discovered
geometric realizations of the Heisenberg algebra
action on $\bH(S)$.
In 1998, Lehn \cite{Leh} constructed
geometric realization of Virasoro action on $\bH(S)$
and initiated a program that studies the  ring structure on $H^*(S^{[n]})$
via the deep interaction with Heisenberg algebra structure.
Both in Lehn's work and an earlier work of Eillingsrud-Str{\o}mme \cite{ES2}
in 1993,
the Chern classes of the tautological sheaves have played an important role.
By a combination and generalization of ideas from \cite{Leh} and \cite{ES2},
Li, Qin and Wang \cite{LQW1} found a set of ring generators of $H^*(S^{[n]})$
for an arbitrary $S$ in 2000.
Later these authors \cite{LQW2} found a different set of ring generators
which uses only Nakajima's Heisenberg operators.
In subsequent works \cite{LQW3, LQW5} they also established
the universality and stability of Hilbert schemes
concerning about the relations among the cohomology rings $H^*(S^{[n]})$ when $S$ or $n$ varies.
In \cite{LQW4}, a $W$-algebra was constructed geometrically acting on $\bH(S)$
by these authors.
Qin and Wang \cite{Qin-Wang} axiomatized the results in \cite{LQW1}-\cite{LQW4}
and used the axiomatic approach to establish similar results for the
orbifold chomology of the symmetry products
of an even-dimensional closed complex manifold.
Through all these works,
various algebraic structures  in the theory of vertex operator algebras
have been geometrically  constructed via Hilbert schemes $S^{[n]}$

Another important aspect of the study of Hilbert schemes
is various universal formulas in the study of Hilbert schemes.
The first kind of universal formulas
give universal expressions for characteristic classes
of the tangent bundles or tautological bundles on Hilbert schemes.
The first example of this type was due to Lehn  \cite{Leh}:
\be
\sum_{n\geq 0} c(L^{[n]})z^n = \exp \biggl( \sum_{m \geq 1} \frac{(-1)^{m-1}}{m} q_m(c(L))z^m\biggr),
\ee
where $L \to S$ is a line bundle, $L^{[n]} \to S^{[n]}$ is the tautological bundle
associated with $L$.
See Boissi\`ere and Nieper-Wi{\ss}kirchen \cite{Boi-Nie} for more examples.
One of their results that concerns us is as follows:
There are unique rational constants $\alpha_\lambda$, $\beta_\lambda$,
$\gamma_\lambda$, $\delta_\lambda$ such that for each surface $S$ and each vector bundle
$F$ on $S$, the generating series of the Chern characters of the tautological bundles of
$F$ is given by:
\be
\begin{split}
\sum_{n \geq 0} \ch(F^{[n]}) & =
\sum_{\lambda\in \cP} \biggl(
\alpha_\lambda q_\lambda(\ch(F )) + \beta_\lambda q_\lambda(e_S \ch(F )) \\
& + \gamma_\lambda q_\lambda(K_S \ch(F )) + \delta_\lambda q_\lambda(K_S^2 \ch(F )) \biggr)\vac.
\end{split}
\ee
The coefficients $\alpha_\lambda$ and $\beta_\lambda$ are known from the
work of Li-Qin-Wang \cite{LQW5},
but unfortunately the closed expressions for $\gamma_\lambda$ and $\delta_\lambda$ are not known.

Another kind of universality results are due to
Ellingrud-G\"ottsche-Lehn  \cite{EGL}.
The following two results will be useful to us.
First, let $P$ be a polynomial in the Chern classes of the tangent
bundle $TS^{[n]}$ and the Chern classes of $F_1^{[n]}, \dots, F_l^{[n]}$.
Then there is a universal polynomial $\tilde{P}$,
depending only on $P$, in the Chern classes of $TS$,
the ranks
$r_1,\dots,r_l$ and the Chern classes of $F_1, \dots ,F_l$ such that
$$\int_{S^{[n]}} P = \int_S \tilde{P}.$$
Secondly,
let $\Psi: K \to H$ be a multiplicative function,
also let $\varphi(x) \in \bQ[[x]]$ be a formal power series and, for
a complex manifold $X$ of dimension $n$ put $\Phi(X) := \varphi(x1)\cdots \varphi(x_n)$
with $x-1,\dots,x_n$ the Chern roots of $TX$.
Then for each integer $r$ there are universal power series $A_i \in \bQ[[z]]$,
$i = 1,\dots, 5$, depending only on $\Psi$, $\Phi$ and $r$,
such that for each $x \in K(S)$ of rank $r$
one has
\be
\begin{split}
& \sum_{n \geq0} z^n \int_{S^{[n]}} \Psi(x^{[n]}) \Phi(S^{[n]}) \\
= & \exp(c^2_1(x)A_1 + c_2(x)A_2 + c_1(x)c_1(S)A_3 + c^2_1(S)A_4 + c_2(S)A_5).
\end{split}
\ee
As pointed out in \cite{EGL},
the series $A_i$($i = 1, \dots,5$) can be determined by doing the computations
for $(\bP^2, r\cdot 1)$, $(\bP^2, \cO(1)+(r-1)\cdot 1)$, $(\bP^2, \cO(2)+(r-1)\cdot 1)$,
$(\bP^2,2\cO(1)+(r-2)\cdot 1)$,
and $(\bP^1 \times \bP^1, r \cdot 1)$.
So it suffices to do computations
for sheaves on $\bP^2$ and $\bP^1 \times \bP^1$,
which are equivariant sheaves on toric surfaces.
The torus action on a toric surface $S$ induces a torus action
on $S^{[n]}$,
so after an application of  localization formula,
one can reduce to compute equivariant intersection numbers on $(\bC^2)^{[n]}$.
This turns out to be an effective method for computing intersection numbers.
This is the approach taken by the author and his collaborators
in a series of papers \cite{Liu-Yan-Zhou, Wang-Zhou1, Wang-Zhou2},
of which this paper is a natural continuation.

Working in the equivariant setting and applying localization formula
on Hilbert schemes dates back to the 1987 paper of Ellingsrud-Str{\o}mme  \cite{Ell-Str1}.
In the 1990s,
Haiman and his collaborators (see e.g. \cite{Haiman})
related equivariant K-theory of Hilbert schemes $(\bC^2)^{[n]}$
to Macdonald polynomials and found deep applications in combinatorics.
Nakajima \cite{Nak3} related equivariant cohomology of
Hilbert schemes of a line bundle over $\bP^1$ to
Jack polynomials.
Recall that Jack polynomials depend on some parameter $\alpha$.
In this case,
$\alpha$ is the negative of the degree of the line bundle.
Motivated by this result,
Vasserot \cite{Vas} identified the ring $H^{2n}_T((\bC^2)^{[n]})$ with the class algebra
of the symmetric group $S_n$ in 2002.
Here the $T=\bC^*$-action on $S =\bC^2$ is defined by
$t\cdot (x, y)=(tx, t^{-1}y)$.
In this setting,
the fixed points of the circle action
on $(\bC^2)^{[n]}$ are parameterized by partitions $\{\lambda\;|\;|\lambda|=n\}$ of $n$,
and their classes correspond to Schur functions $\{s_\lambda\;|\;|\lambda|=n\}$.
Once a connection with the theory of symmetric functions
is set up,
one can apply its connection to integrable hierarchies as developed by
the Kyoto school \cite{MJD}
and its connection to Gromov-Witten invariants of algebraic curves
developed by Okounkov-Pandharipande \cite{Oko-Pan}.
Furthermore, one can apply a connection to Bloch-Okounkov formula \cite{BO}
via boson-feermion correspondence.
This was done by  Li-Qin-Wang \cite{LQW6} in 2003.
In 2004,
these authors \cite{LQW7} established another link between equivariant cohomolgy
of Hilbert schemes and Jack polynomials.
They considered the following one-torus action on $\bC^2$:
$t \cdot (x, y) = (t^{\alpha}x, t^{-\beta}y)$.
In this case the parameter for the Jack polynomials is $\beta/\alpha$.
See Nakajima's lecture notes \cite{Nak-More}
for a review of some results on equivariant chomology of Hilbert schemes $(\bC^2)^{[n]}$.

The above results that we have reviewed so far only concern the conjecture
of Vafa-Witten \cite{Vafa-Witten} on the relation with 2-dimensional conformal field
theories.
In their discussion of $N=4$ supersymmetric gauge theory
in the end of \cite{Vafa-Witten},
Vafa and Witten have focused on moduli spaces on K3 surfaces and $T^4$.
They remarked that:``So it is
conceivable that $K3$ could be replaced here by another, perhaps non-compact, space."
As it turns out,
the  conjectured relation with string theory
has gradually been developed  into what is called ``geometric engineering"
of $N=2$ gauge theories  from noncompact toric Calabi-Yau 3-folds
first proposed in 1996 by Katz, Klemm and Vafa \cite{Katz-Klemm-Vafa}.
Furthermore,
while  the partition function of $N=4$ gauge theory is the generating series
of the Euler characteristics of the moduli spaces,
for $N=2$ gauge theory on $\bR^4$ one considers the Nekrasov partition function \cite{Nek}
which can be defined by localization on the moduli spaces
of framed torison sheaves on $\bP^2$ with respect to a torus action
(see Nakajima-Yoshioka \cite{Nak-Yos-Lect} for the mathematical backgrounds).
More precisely,
let $M(r, k)$ denote the framed moduli space
of torsion free sheaves on $\bP^2$ with rank $r$ and $c_2 = k$.
By framing  we mean a trivialization of the sheaf restricted to the line at infinity.
In particular for $r = 1$, $M(1,k) =(\bC^2)^{[k]}$.
On the gauge theory side we are interested in
the generating series of equivariant Riemann-Roch numbers
\be
\sum_{k\geq 0} Q^k\chi(M(r, k), K_{r,k}^{1/2} \otimes  (\det \bV)^m)(e_1,\dots, e_r, t_1, t_2)
\ee
with respect to some $T^r\times T^2$-action,
where $K_{r,k}$ is the canonical line bundle of $M(r, k)$,
and $\bV$ is some tautological bundle on $M(r,k)$.
The idea of geometric engineering is to identify this partition function
with generating series of local Gromov-Witten invariants
of some toric Calabi-Yau 3-fold which is a fibration over $\bP^1$ with ALE spaces
of type $A_{r-1}$ as fibers.
This was done by Li-Liu-Zhou \cite{Li-Liu-Zho} in 2004
using the theory of the topological vertex \cite{AKMV, LLLZ},
together with some combinatorial results in \cite{Zhou}.
When the rank is one,
the relevant moduli spaces are the Hilbert schemes $(\bC^2)^{[n]}$,
the tautological bundle is the tautological bundle $\xi_n$ induced
from the trivial bundle $\cO_{\bC^2}$.
The following result was proved in Yang-Zhou \cite{Yang-Zhou} in 2011:
\be
Z_{\cO_{\bP^1}(k) \oplus \cO_{\bP^1}(-k - 2)}(\lambda, t) = \sum_{n\geq 0}
((-1)^kt)^n\chi((\bC^2)^{[n]}, (\det \xi_n)^{k+1})(q^{-1}, q),
\ee
where $Z_{\cO_{\bP^1}(k) \oplus \cO_{\bP^1}(-k - 2)}(\lambda, t)$
is the generating series of local Gromov-Witten invariants
of $\cO_{\bP^1}(k) \oplus \cO_{\bP^1}(-k - 2)$.
Hence geometric engineering
directly relates curve counting to $K$-theoretical intersection
numbers on Hilbert schemes.
(For other applications of Hilbert schemes to curve counting problem,
see \cite{Guo-Zhou} and the references therein.)
The motivation behind \cite{Li-Liu-Zho} and \cite{Yang-Zhou}
is the computation of Gopakumar-Vafa invariants of the corresponding
toric Calabi-Yau 3-folds.
The idea is to find analogues of G\"ottsche's formula
for K-theoretical intersection numbers on Hilbert schemes.
Motivated by this, in 2006 the author made the following conjecture:
The following identity holds for $k \geq 2$:
\begin{equation} \label{eqn:Main}
\begin{split}
& \sum_{n \geq 0} Q^n \chi((\bC^k)^{[n]},
\Lambda_{-u} \xi^A_n \otimes \Lambda_{-v}(\xi^A_n)^*)(t_1, \dots, t_k) \\
=& \exp ( \sum_{n=1}^{\infty}
\frac{(1-u^nt^{nA})(1-v^nt^{-nA})Q^n}{n \prod_{i=1}^k (1-t_i^n)}),
\end{split}
\end{equation}
where
$$t^{nA} = t_1^{na_1} \cdots t_k^{na_k}.$$
Some special cases were proved then by relating the local contributions
from the fixed points to specializations of Macdonald polynomials.
In 2011,
Zhilan Wang realized that a slight modification can be used to prove the general case,
hence another connection to Macdonald polynomials was
established in \cite{Wang-Zhou1} in the setting of equivariant $K$-theory of
$(\bC^2)^{[n]}$.

Nekerasov \cite{Nek} made a conjecture relating his partition functions
to Seiberg-Witten prepotential.
Two different proofs appeared in 2003.
The proof by Nakajima and Yoshioka \cite{Nak-Yos1} derived an equation
for the partition function called the blowup equation
by considering the instanton moduli space on $\bC^2$ blown up at the origin.
They also proved the K-theoretic version using the same idea \cite{Nak-Yos2}.
See also \cite{Nak-Yos-Lect}.
In the paper by Nekrasov and Okounkov where they gave a different proof,
the following result result was proved in the the rank $r=1$ case:
\be
\sum_{\lambda \in \cP} q^{|\lambda|}
\prod_{\Box \in \lambda} \frac{h(\Box)^2- m^2}{h(\Box)^2}  \\
= \tr \cV_m(1)|_{\cF_0}
= \prod_{n > 0} (1-q^n)^{m^2-1},
\ee
where $\cF_0$ denotes the charge zero part of the fermionic Fock space,
and $\cV_m(z)$ is the vertex operator:
\be
\cV_m(z) : = \exp\biggl(-m \sum_{n>0} \alpha_{-n} \frac{z^n}{n}\biggr)
\exp \biggl( m \sum_{n>0} \alpha_n \frac{z^{-n}}{n} \biggr).
\ee
The geometric realization of this vertex operator on Hilbert schemes
has been discovered by Carlsson and Okounkov \cite{Car1, Car2, Car-Oko}.
With the help of this vertex operator,
it is possible to evaluate intersection numbers of the form:
\be
F(k_1, \dots, k_N):=\sum_{n\geq 0} q^n \int_{S^{[n]}} \ch_{k_1}(L_1^{[n]}) \cdots \ch_{k_N}(L_N^{[n]}) e(T_mS^{[n]})
\ee
where $e(T_mS^{[n]}) = \sum_j m^j c_{2n-j}(TS^{[n]})$ can be understood as the Chern polynomial
of $TS^{[n]}$.
In \cite{Car1, Car2, Car-Oko},
many examples of type $F(k_1, \dots, k_N)$ have been shown to lead to quasimodular forms.
More generally,
Okounkov \cite{Okounkov} made a conjecture connecting
cohomological intersection numbers on Hilbert schemes to q-zeta values.
This conjecture has been verified up to lower weight terms
by Qin and Yu \cite{Qin-Yu} based on the vertex operator of Carlsson and Okounkov
\cite{Car1, Car-Oko} and
the Chern character operators \cite{LQW4}.
See also the more recent work by Qin and Shen \cite{Qin-Shen}.

Inspired by the above work on the Okounkov conjecture,
the author \cite{Zhou2018} related the computations of intersection numbers of the form
\be
\sum_{n \geq 0} q^n
\int_{(\bC^2)^{[n]}_{S^1}} \prod_{j=1}^N
\sum_{k_j\geq0} z_j^{k_j}\ch_{k_j}(\xi_n^{A_j})_{S^1} \cdot e_{S^1}(T\bC^{[n]})
\ee
to Bloch-Okounkov formula \cite{BO}
where the circle acts on $\bC^2$ by $t\cdot (x, y) = (tx, t^{-1}y)$,
related the two-torus-equivariant version of it, i.e., 
\ben
&& \sum_{n \geq 0} q^n
\int_{(\bC^2)^{[n]}} \prod_{j=1}^N
\sum_{k_j\geq0} z_j^{k_j}\ch_{k_j}(\xi_n^{A_j})_{T} \cdot e_{T}(T\bC^{[n]}),
\een
to the generalized Bloch-Okounkov correlation functions introduced by Cheng and Wang \cite{Cheng-Wang}.

Now we can state the problem we want to solve in this paper:
We want to find a general algorithm to compute equivariant
cohomological intersection numbers of the form:
\ben
&& \sum_{n \geq 0} q^n
\int_{(\bC^2)^{[n]}} \prod_{j=1}^N \ch_{k_j}(\xi_n^{A_j})_{T},
\een
or their generating series:
\ben
&& \sum_{n \geq 0} q^n
\int_{(\bC^2)^{[n]}} \prod_{j=1}^N
\sum_{k_j\geq0} z_j^{k_j} \ch_{k_j}(\xi_n^{A_j})_{T}.
\een
We will use a strategy that has been widely used:
If the problem does not seem to have a good answer,
try a more complicated related problem.
For us we first change the problem to a problem in equivariant $K$-theory by taking
each $z_i$ to be a positive integer $m_i$,
then
\be
\sum_{k_i \geq 0} m_i^{k_i} \ch_{k_i}(\xi^{A_i}_n)
= \ch (\psi_{m_i}(\xi_n^{A_i})).
\ee
where $\psi_m$ are the Adams operations.
So we are interested in the problem of computing equivariant Riemann-Roch numbers:
\be
\chi((\bC^2)^{[n]}, \psi_{m_1}(\xi_n^{A_1}) \otimes \cdots \otimes
\psi_{m_N}(\xi_n^{A_N}))(t_1, t_2),
\ee
As it turns out,
we need to consider an even more complicated problem of
computing the equivariant Riemann-Roch numbers of the form:
\be
\chi((\bC^2)^{[n]}, \psi_{m_1}(\xi_n^{A_1}) \otimes \cdots \otimes \psi_{m_N}(\xi_n^{A_N})
\otimes \Lambda_{-u}\xi_n^A \otimes \Lambda_{-v} (\xi_n^A)^*)(t_1, t_2),
\ee
and after obtaining results for computing them, take $u=v=0$.
The inclusion of $\Lambda_{-u}\xi_n^A \otimes \Lambda_{-v} (\xi_n^A)^*$
is clearly inspired by Conjecture 1 and the results in \cite{Wang-Zhou1}.
This is the crucial step that enables us to combine the two connections to Macdonald
functions discovered by Haiman \cite{Haiman} and Wang-Zhou \cite{Wang-Zhou1}
in the same framework,
and apply some vertex operator realizations of the Macdonald operators.

Our goal is to combine and generalize all the ideas in the works mentioned in this Section together.
In the rest of this paper we will report on some progresses in this direction.

\section{Equivariant Intersection Numbers on $(\bC^2)^{[n]}$ by Localization Formulas}

In this Section
we review how to compute some equivariant intersection numbers
in equivariant K-theory and equivariant cohomology theory of $(\bC^2)^{[n]}$ by localizations\emph{\emph{}}.

\subsection{Preliminaries for localizations on Hilbert schemes of the affine plane}

See \cite{Nak-Lecture} for references on equivariant indices
on Hilbert schemes of points on $\bC^2$.
The torus action on $\bC^2$ given by
$$(t_1, t_2) \cdot (x, y) = (t_1^{-1} x,  t_2^{-1} y)$$
on linear coordinates induces an action on $(\bC^2)^{[n]}$.
The fixed points are isolated and parameterized by partitions
$\lambda=(\lambda_1, \dots, \lambda_l)$ of weight $n$.
The weight decomposition of the tangent bundle of $T(\bC^2)^{[n]}$ at a
fixed point $I_\lambda$ indexed by a partition $\lambda$ is given by \cite{Nak-Lecture}:
\begin{eqnarray} \label{eqn:Weights}
&& \sum_{s \in \lambda} (t_1^{l(s)} t_2^{-(a(s)+1)}
+ t_1^{-(l(s)+1)}t_2^{a(s)}).
\end{eqnarray}
It follows that the equivariant Euler class is given at $I_\lambda$ by:
\be
e_T(T(\bC^2)^{[n]})|_{I_\lambda}
= \prod_{s \in \lambda} (l(s)w_1-(a(s)+1)w_2) (-(l(s)+1)w_1 +a(s)w_2).
\ee

For a vector $A =(a, b) \in \bZ^2$,
denote by $\cO_{\bC^2}^A$ the $T^2$-equivariant line bundle on $\bC^2$
with weight $A$,
and let $\xi_n^A$ be its associated tautological bundles.
Then one has the following weight decomposition at the fixed points:
\begin{eqnarray*}
&& \xi_n^A|_{I_{\lambda}}
%%% =  \sum_{(i, j) \in \lambda} t_1^{i-1}t_2^{j-1}t_1^{a}t_2^{b}
= \sum_{s \in \lambda} t_1^{l'(s)}t_2^{a'(s)}t_1^{a}t_2^{b}.
\end{eqnarray*}
So their equivariant Chern characters at the fixed points are given by the following formula:
\be
\begin{split}
\ch(\xi_n^A)_T|_{I_{\lambda}} %% & =  \sum_{(i, j) \in \lambda} e^{(i-1+a)t_1+(j-1+b)t_2}  \\
& = \sum_{s \in \lambda} e^{(l'(s)+a)w_1+(a'(s)+b)w_2}.
\end{split}
\ee
In particular,
\be
\begin{split}
\ch_k(\xi_n^A)_T|_{I_{\lambda}}
%%%% & =  \frac{1}{k!} \sum_{(i, j) \in \lambda} ((i-1+a)t_1+(j-1+b)t_2)^k  \\
& = \frac{1}{k!} \sum_{s \in \lambda} ((l'(s)+a)w_1+(a'(s)+b)w_2)^k.
\end{split}
\ee
In this Subsection,
the following notations have been used:
\begin{align*}
a_{\mu}(s) & = \mu_i - j, & a_{\mu}'(s) & = j-1,\\
l_{\mu}(s) & = \mu^t_j - i, & l_{\mu}'(s) & = i -1.
\end{align*}

\subsection{Some equivariant K-theoretical intersection numbers}

For the equivariant K-theoretical case,
we are interested in the equivariant Riemann-Roch numbers:
\be
\chi((\bC^2)^{[n]}, \xi_n^{A_1} \otimes \cdots \otimes \xi_n^{A_N})(t_1, t_2),
\ee
where $A_j = (a_j, b_j) \in \bZ^2$.
By holomorphic Lefschetz formula \cite{Ati-Sin},
\begin{equation} \label{eqn:K-Int-1}
\begin{split}
& \sum_{n \geq 0} Q^n
\chi((\bC^2)^{[n]}, \bigotimes_{j=1}^N \xi_n^{A_j} )(t_1, t_2) \\
= & \sum_{\mu} Q^{|\mu|} \prod_{j=1}^N \sum_{s\in \mu} t_1^{l'(s)+a_j}t_2^{a'(s)+b_j}
 \\
& \cdot \prod_{s \in \mu}
\frac{1}{(1- t_1^{-l(s)} t_2^{a(s)+1})(1 - t_1^{l(s)+1}t_2^{-a(s)})}.
\end{split}
\end{equation}
Inspired by the proof of Theorem 1  in \cite{Wang-Zhou1},
we will instead first compute the more complicated equivariant Riemann-Roch numbers:
\be
\chi((\bC^2)^{[n]}, \xi_n^{A_1} \otimes \cdots \otimes \xi_n^{A_N}
\otimes \Lambda_{-u}\xi_n^A \otimes \Lambda_{-v} (\xi_n^A)^*)(t_1, t_2),
\ee
where $A =(a,b)$,
and then take $u=v=0$.
By holomorphic Lefschetz formula,
\begin{equation} \label{eqn:K-Int0}
\begin{split}
& \sum_{n \geq 0} Q^n
\chi((\bC^2)^{[n]}, \bigotimes_{j=1}^N \xi_n^{A_j} \otimes
\Lambda_{-u}\xi_n^A \otimes \Lambda_{-v} (\xi_n^A)^*)(t_1, t_2) \\
= & \sum_{\mu} Q^{|\mu|} \prod_{j=1}^N \sum_{s\in \mu} t_1^{l'(s)+a_j}t_2^{a'(s)+b_j}
 \\
& \cdot \prod_{s \in \mu}
\frac{(1-ut^At_1^{l'(s)}t_2^{a'(s)}) \cdot(1- vt^{-A}t_1^{-l'(s)}t_2^{-a'(s)})}
{(1- t_1^{-l(s)} t_2^{a(s)+1})(1 - t_1^{l(s)+1}t_2^{-a(s)})}.
\end{split}
\end{equation}
This will enable us to exploit the connection to Macdonald polynomials as in \cite{Wang-Zhou1}.

We make some digressions on symmetric functions before we discuss more K-theoretical intersection numbers.

\subsection{Basis of the space $\Lambda$ of symmetric functions}

As a vector space, $\Lambda$ has several natural additive basis.
For our purpose,
we will use the following three of  them:
$\{e_\lambda\}_{\lambda\in \cP}$, $\{h_\lambda\}_{\lambda\in\cP}$, $\{p_\lambda\}_{\lambda\in \cP}$.
Here for a partition $\lambda = (\lambda_1, \dots, \lambda_l)$,
$e_\lambda = e_{\lambda_1} \cdots e_{\lambda_l}$,
$h_\lambda = h_{\lambda_1} \cdots h_{\lambda_l}$,
and $p_\lambda = p_{\lambda_1} \cdots p_{\lambda_l}$,
where $e_n$, $h_n$ and $p_n$ are the {\em elementary}, {\em complete},
and {\em Newton} symmetric functions, respectively.
Their generating series are defined by:
\bea
&& E(t) = \sum_{n\geq 0} e_n t^n, \\
&& H(t) = \sum_{n\geq 0} h_n t^n, \\
&& P(t) = \sum_{n \geq 1} \frac{p_n}{n} t^n.
\eea
They are related to each other as follows:
\begin{align} \label{eqn:SymmetricFunct}
H(t) & = \frac{1}{E(-t)}, &
P(t) & = - \log E(-t) = \log H(t).
\end{align}

\subsection{Action of the space $\Lambda$  on $K$-theory}

There is an action of $\Lambda$ on $K(X)$  using the power operations on $K(X)$ (see e.g. Atiyah \cite{Atiyah-K}).
First, define a map $\Lambda \to \End(K(X))$  by
by $p_\mu \mapsto \prod_{i=1}^l \psi^{\mu_i}: K(X) \to K(X)$
where $\psi^n$ are the Adams operations.
Under this map $e_\mu \mapsto \lambda^\mu = \lambda^{\mu_1} \cdots \lambda^{\mu_l}$,
where $\lambda^i: K(X) \to K(X)$ is defined as follows:
\be
\sum_{i\geq 0} t^i \lambda^i(E-F) = \sum_{i\geq 0} t^i \Lambda^i(E) \sum_{j\geq 0} (-t)^j S^j(F).
\ee
Similarly,
$h_\mu \mapsto \sigma^\mu = \sigma^{\mu_1} \cdots \sigma^{m_l}$,
where $\sigma^i£º K(X) \to K(X)$ is defined by
\be
\sum_{i\geq 0} t^i \sigma^i(E-F) = \sum_{i\geq 0} t^i S^i(E) \sum_{j\geq 0} (-t)^j \Lambda^j(F).
\ee
Indeed,
\eqref{eqn:Symmetric} now becomes
\begin{align} \label{eqn:K-Oper}
\sum_{i \geq 0} t^i \sigma^i & = \frac{1}{\sum_{i\geq 0} (-t)^i \lambda^i}, &
\sum_{i\geq 0} \frac{t^i}{i} \psi^i  & = - \log \sum_{i\geq 0} (-t)^i \lambda^i
= \log \sum_{i \geq 0} t^i \sigma^i.
\end{align}
These operations can also be defined in equivariant $K$-theory.

\subsection{More K-theoretical intersection numbers}

Note one has
\be
\ch(\psi^m(\xi_n^A))_T|_{I_\lambda} = \sum_{s \in \lambda} e^{m[(l'(s)+a)w_1+(a'(s)+b)w_2]}.
\ee
In particular,
\be
\begin{split}
\ch_k(\psi^m(\xi_n^A))_T|_{I_{\lambda}}
& = \frac{m^k}{k!} \sum_{s \in \lambda} ((l'(s)+a)w_1+(a'(s)+b)w_2)^k \\
& = m^k \cdot \ch(\xi_n^A)_T|_{I_\lambda}.
\end{split}
\ee

We are interested in the equivariant Riemann-Roch numbers:
\be
\chi((\bC^2)^{[n]}, \psi^{m_1}(\xi_n^{A_1}) \otimes \cdots \otimes
\psi^{m_N}(\xi_n^{A_N}))(t_1, t_2),
\ee
where $A_j = (a_j, b_j) \in \bZ^2$.
Again we will instead first compute the equivariant Riemann-Roch numbers:
\be
\chi((\bC^2)^{[n]}, \psi^{m_1}(\xi_n^{A_1}) \otimes \cdots \otimes \psi^{m_N}(\xi_n^{A_N})
\otimes \Lambda_{-u}\xi_n^A \otimes \Lambda_{-v} (\xi_n^A)^*)(t_1, t_2).
\ee
By holomorphic Lefschetz formula,
\begin{equation} \label{eqn:K-Int}
\begin{split}
& \sum_{n \geq 0} Q^n
\chi((\bC^2)^{[n]}, \bigotimes_{j=1}^N \psi^{m_j}(\xi_n^{A_j})
\otimes
\Lambda_{-u}\xi_n^A \otimes \Lambda_{-v} (\xi_n^A)^*)(t_1, t_2) \\
= & \sum_{\mu} Q^{|\mu|} \prod_{j=1}^N \sum_{s\in \mu} e^{m_j[(l'(s)+a_j)w_1+(a'(s)+b_j)w_2]}
 \\
& \cdot \prod_{s \in \mu}
\frac{(1-ut^At_1^{l'(s)}t_2^{a'(s)}) \cdot(1- vt^{-A}t_1^{-l'(s)}t_2^{-a'(s)})}
{(1- t_1^{-l(s)} t_2^{a(s)+1})(1 - t_1^{l(s)+1}t_2^{-a(s)})}.
\end{split}
\end{equation}
We also consider other intersection numbers by changing $\psi^{m_j}$ to $\lambda^{m_j}$ or $\sigma^{m_j}$.
We will also establish a connection of all these to Macdonald polynomials
as in \cite{Wang-Zhou1}.

One can also consider the generating series of these series:
\ben
&& \sum_{m_1, \dots, m_N \geq 0} \frac{\prod_{j=1}^N a^{A_j}_{m_j}}{N!}  \sum_{n \geq 0} Q^n
\chi((\bC^2)^{[n]}, \bigotimes_{j=1}^N \psi^{m_j}(\xi_n^{A_j})
\otimes
\Lambda_{-u}\xi_n^A \otimes \Lambda_{-v} (\xi_n^A)^*)(t_1, t_2), \\
&& \sum_{m_1, \dots, m_N \geq 0} \frac{\prod_{j=1}^N b^{A_j}_{m_j}}{N!}  \sum_{n \geq 0} Q^n
\chi((\bC^2)^{[n]}, \bigotimes_{j=1}^N \lambda^{m_j}(\xi_n^{A_j})
\otimes
\Lambda_{-u}\xi_n^A \otimes \Lambda_{-v} (\xi_n^A)^*)(t_1, t_2), \\
&& \sum_{m_1, \dots, m_N \geq 0} \frac{\prod_{j=1}^N c^{A_j}_{m_j}}{N!}  \sum_{n \geq 0} Q^n
\chi((\bC^2)^{[n]}, \bigotimes_{j=1}^N \sigma^{m_j}(\xi_n^{A_j})
\otimes
\Lambda_{-u}\xi_n^A \otimes \Lambda_{-v} (\xi_n^A)^*)(t_1, t_2), \\
\een
where $\{a^A_m, b^A_m, c^A_m\}_{m \geq 0}$ are some formal variables,
indexed by $A\in \bZ^2$.
These generating series are related to each other if we require
\be
\sum_{\mu \in \cP} a^A_\mu \psi^\mu = \sum_{\mu\in \cP} b^A_\mu \lambda^\mu
= \sum_{\mu \in \cP} c^A_\mu \sigma^\mu
\ee
for each $A \in \bZ^2$.

\subsection{Some equivariant cohomological intersection numbers}

Similarly,
in the equivariant cohomological case,
we are interested in the equivariant intersection numbers:
\begin{equation}
\begin{split}
& \sum_{n \geq 0} q^n
\int_{(\bC^2)^{[n]}_T} \ch_{k_1}(\xi_n^{A_1})_T \cdots \ch_{k_N}(\xi_n^{A_n})_T.
\end{split}
\end{equation}
In \cite{Zhou2018} we have considered a different type of equivariant intersection numbers:
\be
\sum_{n \geq 0} q^n
\int_{(\bC^2)^{[n]}_T} \ch_{k_1}(\xi_n^{A_1})_T \cdots \ch_{k_N}(\xi_n^{A_n})_T \cdot e_T(T\bC^{[n]}).
\ee
The motivation there was Okounkov's Conjecture.
By localization formula we have:
\begin{equation} \label{eqn:Coh-Int}
\begin{split}
& \sum_{n \geq 0} q^n
\int_{(\bC^2)^{[n]}_T} \ch_{k_1}(\xi_n^{A_1})_T \cdots \ch_{k_N}(\xi_n^{A_n})_T  \\
= & \sum_{\lambda} q^{|\lambda|}
\prod_{j=1}^N \frac{1}{k_j!} \sum_{s \in \lambda} ((l'(s)+a_j)w_1+(a'(s)+b_j)w_2)^{k_j} \\
 & \cdot \prod_{s \in \lambda} \frac{1}{(l(s)w_1-(a(s)+1)w_2) (-(l(s)+1)w_1 +a(s)w_2)}
\end{split}
\end{equation}
and for the intersection numbers with the equivariant Euler class:
\begin{equation}
\begin{split}
& \sum_{n \geq 0} q^n
\int_{(\bC^2)^{[n]}_T} \ch_{k_1}(\xi_n^{A_1})_T \cdots \ch_{k_N}(\xi_n^{A_n})_T \cdot e_T(T\bC^{[n]})  \\
= & \sum_{\lambda} q^{|\lambda|}
\prod_{j=1}^N \frac{1}{k_j!} \sum_{s \in \lambda} ((l'(s)+a_j)w_1+(a'(s)+b_j)w_2)^{k_j}.
\end{split}
\end{equation}
It was explained in \cite{Zhou2018} that the latter can be computed by first considering the generating series:
\ben
&& \sum_{n \geq 0} q^n
\int_{(\bC^2)^{[n]}_{S^1}} \prod_{j=1}^N
\sum_{k_j\geq0} z_j^{k_j}\ch_{k_j}(\xi_n^{A_j})_{T} \cdot e_{T}(T\bC^{[n]}) \\
& = & \sum_{\lambda} q^{|\lambda|} \prod_{k=1}^N \sum_{(i, j) \in \lambda}
e^{z_k[(i-1+a_k)t_1+(j-1+b_k)t_2]} \\
& = & e^{\sum_{j=1}^N (a_jt_1+b_jt_2)z_j} \cdot
\sum_{\lambda} q^{|\lambda|} \prod_{k=1}^N \sum_{s \in \lambda}
e^{z_k[l'_\lambda(s)t_1+a'_\lambda(s)t_2]},
\een
then reducing to the deformed $n$-point function
\be
\corr{\mB_\lambda(e^{z_1t_2}, e^{z_1t_1}) \cdots \mB_\lambda(e^{z_Nt_2}, e^{z_Nt_1})}_q
\ee
introduced by Cheng and Wang \cite{Cheng-Wang}.
In order to compute \eqref{eqn:Coh-Int} using the same ideas,
we first compute
\begin{equation} \label{eqn:Coh-Int2}
\begin{split}
& \sum_{n \geq 0} q^n
\int_{(\bC^2)^{[n]}_T} \ch_{k_1}(\xi_n^{A_1})_T \cdots \ch_{k_N}(\xi_n^{A_n})_T
\cdot c_x(\xi_n^A)_T\cdot c_y((\xi_n^A)^*)_T   \\
& = \sum_{\lambda} q^{|\lambda|}
\prod_{j=1}^N \frac{1}{k_j!} \sum_{s \in \lambda} ((l'(s)+a_j)w_1+(a'(s)+b_j)w_2)^{k_j} \\
\cdot & \prod_{s \in \lambda} \frac{(x+(l'(s)+a)w_1+(a'(s)+b)w_2)
 \cdot(y-(l'(s)+a)w_1-(a'(s)+b)w_2)}{(l(s)w_1-(a(s)+1)w_2) (-(l(s)+1)w_1 +a(s)w_2)},
\end{split}
\end{equation}
then take the coefficients of $q^nx^ny^n$.
Here $c_x(E)$ denotes the Chern polynomial of $E$:
$c_x(E) = \sum_{j=0}^r x^{r-j} c_j(E)$.
One can understand \eqref{eqn:K-Int-2} as a quantization of \eqref{eqn:Coh-Int2}.

\section{Localization on Hilbert Schemes and Macdonald Functions}

We rewrite \eqref{eqn:K-Int0} as follows:
\begin{equation} \label{eqn:K-Int-2}
\begin{split}
& \sum_{n \geq 0} Q^n
\chi((\bC^2)^{[n]}, \bigotimes_{j=1}^N \xi_n^{A_j}
\otimes
\Lambda_{-u}\xi_n^A \otimes \Lambda_{-v} (\xi_n^A)^*)(t_1, t_2) \\
= & \prod_{j=1}^N t_1^{a_j}t_2^{b_j} \cdot
\sum_{\mu} (-ut^AQ)^{|\mu|} \prod_{j=1}^N \sum_{s\in \mu} t_1^{l'(s)}t_2^{a'(s)}
 \\
& \cdot \prod_{s \in \mu}
\frac{(t_2^{a'(s)}- vt^{-A}t_1^{-l'(s)})}{(1- t_1^{-l(s)} t_2^{a(s)+1})}
\frac{(t_1^{l'(s)}-(ut^A)^{-1}t_2^{-a'(s)})}
{(1 - t_2^{-a(s)}t_1^{l(s)+1})}.
\end{split}
\end{equation}
One can recognize the following three expressions on the right-hand side.
First the expression
\begin{align}
& \sum_{s\in \mu}  t_2^{a'(s)}\cdot t_1^{l'(s)}
\end{align}
is related to the eigenvalues of Macdonald operator $E$,
and this fact has been used by Haiman \cite{Haiman} to relate Macdonald polynomials to Hilbert schemes.
Secondly,
the expressions
\begin{align}
\prod_{s \in \mu}
\frac{(t_2^{a'(s)}- vt^{-A}t_1^{-l'(s)})}{(1- t_1^{-l(s)} t_2^{a(s)+1})},
&& \prod_{s\in \mu}
\frac{(t_1^{l'(s)}-(ut^A)^{-1}t_2^{-a'(s)})}
{(1 - t_2^{-a(s)}t_1^{l(s)+1})}
\end{align}
have been recognized by Wang and the author \cite{Wang-Zhou1}
as the specializations of the Macdonald polynomials $P_\mu(x;q,t)$ and $Q_\mu^t(x;t,q)$
respectively,
and hence once again a connection between Macdonald polynomials and Hilbert schemes
was observed.

In this Section we will first recall these connections,
then reformulate them in term of operators on the bosonic Fock space.
This provides a method to compute the right-hand side of \eqref{eqn:K-Int-2}
and its generalizations obtained by power operations in K-theory.

\subsection{Operator formulation of orthogonality of Macdonald polynomials}

Denote by $\Lambda^x$ the space of symmetric functions in $x = (x_1, \dots, x_n, \dots)$,
and let $\Lambda^x_{q,t} =  \Lambda^x \otimes \bC(q, t)$
the algebra of symmetric functions with coefficients in rational functions in $q$ and $t$.
The Macdonald polynomials $\{P_{\mu} = P_{\mu}(x; q, t)\}$
are suitable normalized symmetric functions in $\Lambda^x_{q,t}$
such that \cite[\S VI.4]{Macdonald}:
\begin{itemize}
\item [(i)] $P_\lambda = m_\lambda+\sum_{\mu < \lambda} u_{\lambda\mu}m_\mu$,
where $m_\mu$ are the monomial symmetric functions,
and the coefficients $u_{\lambda\mu}$
are rational functions of $q$ and $t$;
\item[(ii)] With respect to the scalar product
$(\cdot,\cdot)_{q,t}$ on $\Lambda^x_{q,t}$ defined
by
\be
(p_\lambda, p_\mu)_{q,t} = \delta_{\lambda\mu}z_\lambda
\prod_{i=1}^{l(\lambda)} \frac{1-q^{\lambda_i}}{1-t^{\lambda_i}},
\ee
$\{P_\lambda(x;q,t)\}$ is an orthogonal basis:
\bea
&& (P_\lambda, P_\mu)_{q,t} = \delta_{\lambda, \mu} b_\lambda^{-1}, \\
&& b_\lambda = \prod_{s\in \lambda}
\frac{1-q^{a(s)}t^{l(s)+1}}{1-q^{a(s)+1}t^{l(s)}}.
\eea
\end{itemize}
Define another basis $\{Q_\lambda(x;q,t)\}$ by
\be
Q_\lambda(x;q,t) = b_\lambda P_\lambda(x;q,t),
\ee
then $\{Q_\lambda(x;q,t)\}$ is the orthogonal
dual basis of $\{P_\lambda(x;q,t)\}$
with respect to the scalar product $(\cdot, \cdot)_{q,t}$.

One can also consider the orthogonal dual basis
of $\{P_\lambda(x;q,t)\}$
with respect to the scalar product $(\cdot, \cdot)$ defined by
\be
(p_\lambda, p_\mu) = \delta_{\lambda\mu}z_\lambda.
\ee
By \cite[\S VI.5]{Macdonald},
\be \label{eqn:Orth-1}
( \omega P_{\lambda^t}(x;q,t), P_\mu(x;t,q) ) ={\delta_\lambda\mu},
\ee
where  $\omega: \Lambda \to \Lambda$ is the involution on $\Lambda$ defined by:
\be
\omega(p_\lambda) = (-1)^{|\lambda|-l(\lambda)} p_\lambda.
\ee
Equivalently,
\be \label{eqn:Orth}
\sum_\lambda P_\lambda(x;q,t) P_{\lambda^t}(y; t, q)
= \prod_{i,j} (1+x_iy_i)
= \exp \sum_{n=1}^\infty \frac{(-1)^{n-1}}{n} p_n(x)p_n(y).
\ee
Now we reinterpret this identity by operators on $\Lambda_{q,t}$.
For $n> 0$,
let $\alpha_{-n}: \Lambda^x_{q,t} \to \Lambda^x_{q,t}$ be the operator
defined by multiplication by $p_n$,
and let $\alpha_n: \Lambda^x_{q,t} \to \Lambda^x_{q,t}$ be the operator
$n \frac{\pd}{\pd p_n}$.
For a partition $\mu$,
denote by $|\mu\rangle:= s_\mu(x) \in \Lambda^x_{q,t}$,
and by $|\mu;q,t\rangle: = P_\mu(x;q,t)\in \Lambda^x_{q,t}$.
Now \eqref{eqn:Orth} can be rewritten as
\be \label{eqn:Orth2}
\exp \biggl(\sum_{n=1}^\infty \frac{(-1)^{n-1}}{n} p_n(y)
\alpha_{-n} \biggr) \vac
= \sum_{\lambda} P_{\lambda^t}(y; t,q) \cdot |\lambda; q,t\rangle.
\ee
By interchanging $x$ and $y$ we also have:
\be \label{eqn:Orth3}
\exp \biggl(\sum_{n=1}^\infty \frac{(-1)^{n-1}}{n} p_n(y)
\alpha_{-n}\biggr) \vac
= \sum_{\lambda} P_{\lambda}(y; q,t) \cdot |\lambda^t; t,q\rangle.
\ee

\subsection{Specialization of Macdonald polynomials}
\label{sec:Specialization}

Denote by $\epsilon^y_{u,t}: \Lambda_{q,t}^y \to \bQ(q,t)$
the specialization homomorphism defined by
$$\epsilon^y_{u,t} p_n(y) = \frac{1-u^n}{1-t^n}$$
for each integer $n \geq 1$.
Then we have \cite[(6.16), (6.17)]{Macdonald}:
\begin{eqnarray}
&& \epsilon^y_{u,t} P_{\lambda}(y; q, t)
= \prod_{s \in \lambda} \frac{t^{l'(s)}-q^{a'(s)}u}{1- q^{a(s)}t^{l(s)+1}}. \label{eqn:Mac3}
\end{eqnarray}
Changing $\lambda$ to $\lambda^t$ and $u$ to $v$,
we also have
\begin{eqnarray}
&& \epsilon^y_{v,t} P_{\lambda^t}(y; q, t)
= \prod_{s \in \lambda} \frac{t^{a'(s)}-q^{l'(s)}u}{1- q^{l(s)}t^{a(s)+1}}. \label{eqn:Mac3t}
\end{eqnarray}
We further interchange $q$ and $t$ to get:
\begin{eqnarray}
&& \epsilon^y_{v,q} P_{\lambda^t}(y; t, q)
= \prod_{s \in \lambda} \frac{q^{a'(s)}-t^{l'(s)}v}{1- t^{l(s)}q^{a(s)+1}}. \label{eqn:Mac3t-2}
\end{eqnarray}

Now we apply $\epsilon^y_{u,t}$ on both sides of \eqref{eqn:Orth3} to get:
\be \label{eqn:Orth3-2}
\exp \biggl(\sum_{n=1}^\infty \frac{(-1)^{n-1}}{n} \frac{1-u^n}{1-t^n}
\alpha_{-n} \biggr) \vac
= \sum_{\lambda} \prod_{s \in \lambda}
 \frac{t^{l'(s)}-q^{a'(s)}u}{1- q^{a(s)}t^{l(s)+1}}
 \cdot |\lambda^t; t,q\rangle.
\ee
We now change $q$ to $q^{-1}$ and $t$ to $t^{-1}$ to get:
\be \label{eqn:Orth3-3}
\exp \biggl(\sum_{n=1}^\infty \frac{(-1)^{n-1}}{n} \frac{1-u^n}{1-t^{-n}}
\alpha_{-n} \biggr) \vac
= \sum_{\lambda} \prod_{s \in \lambda}
 \frac{t^{-l'(s)}-q^{-a'(s)}u}{1- q^{-a(s)}t^{-(l(s)+1)}}
 \cdot |\lambda^t; t,q\rangle.
\ee
In the above we have used the following property of $P_\lambda$
(c.f. \cite[p. 324]{Macdonald}):
\be
P_\lambda(x;q^{-1}, t^{-1}) = P_\lambda(x;q,t).
\ee
Similarly,
we apply $\epsilon^y_{v,q}$ on both sides of \eqref{eqn:Orth2} to get:
\be \label{eqn:Orth2-2}
\exp \biggl(\sum_{n=1}^\infty \frac{(-1)^{n-1}}{n}\frac{1-v^n}{1-q^n}
 \alpha_{-n} \biggr) \vac
= \sum_{\lambda}\prod_{s \in \lambda}
\frac{q^{a'(s)}-t^{l'(s)}v}{1- t^{l(s)}q^{a(s)+1}}
\cdot |\lambda; q,t\rangle.
\ee
Now we take $q=t_2$, $t=t_1^{-1}$,
and change $u$ to $(ut^A)^{-1}$, $v$ to $vt^{-A}$:
\be \label{eqn:Orth2-3}
\exp \biggl(\sum_{n=1}^\infty
\frac{(-1)^{n-1}}{n}\frac{1-v^nt^{-nA}}{1-t_2^n}
 \alpha_{-n} \biggr) \vac
= \sum_{\lambda}\prod_{s \in \lambda}
\frac{t_2^{a'(s)}-vt^{-A}t_1^{-l'(s)}}{1- t_1^{-l(s)}t_2^{a(s)+1}}
\cdot |\lambda; t_2,t_1^{-1}\rangle.
\ee
\be \label{eqn:Orth3-4}
\exp \biggl(\sum_{n=1}^\infty \frac{(-1)^{n-1}}{n} \frac{1-(ut^{A})^{-n}}{1-t_1^{n}}
\alpha_{-n} \biggr) \vac
= \sum_{\lambda} \prod_{s \in \lambda}
 \frac{t_1^{l'(s)}-t_2^{-a'(s)}(ut^A)^{-1}}{1- t_2^{-a(s)}t_1^{l(s)+1}}
|\lambda^t; t_1^{-1},t_2\rangle.
\ee
Let $K$ be the operator on $\Lambda$ such that $K|\mu\rangle:= (-ut^AQ)^{|\mu|}|\mu\rangle$.
Extend it naturally to $\Lambda_{q,t}$,
and it is easy to see that
one has $K|\mu;q,t\rangle = (-ut^A)^{|\mu|} |\mu; q, t\rangle$.
Apply operator $K \omega $ on both sides of \eqref{eqn:Orth2-3}
and take scalar product with both sides of \eqref{eqn:Orth3-4}:
\ben
&& \sum_{n \geq 0} Q^n
\chi((\bC^2)^{[n]}, \Lambda_{-u}\xi_n^A \otimes \Lambda_{-v} (\xi_n^A)^* )(t_1, t_2) \\
& = & \sum_{\mu} (-ut^AQ)^{|\mu|}
 \cdot \prod_{s \in \mu}
\frac{(t_2^{a'(s)}- vt^{-A}t_1^{-l'(s)})}{(1- t_1^{-l(s)} t_2^{a(s)+1})}
\frac{(t_1^{l'(s)}-(ut^A)^{-1}t_2^{-a'(s)})}
{(1 - t_2^{-a(s)}t_1^{l(s)+1})} \\
& = & \lvac \exp \biggl(\sum_{n=1}^\infty \frac{(-1)^{n-1}}{n} \frac{1-(ut^{A})^{-n}}{1-t_1^{n}}
\alpha_{n} \biggr) K \omega \exp \biggl(\sum_{n=1}^\infty
\frac{(-1)^{n-1}}{n}\frac{1-v^nt^{-nA}}{1-t_2^n}
 \alpha_{-n} \biggr) \vac \\
& = & \exp \biggl(\sum_{n=1}^\infty \frac{(-1)^{n-1}(-ut^AQ)^{n}}{n} \frac{1-(ut^{A})^{-n}}{1-t_1^{n}}
\cdot \frac{1-v^nt^{-nA}}{1-t_2^n} \biggr) \\
& = &
\exp \biggl(\sum_{n=1}^\infty \frac{Q^n}{n} \frac{1-(ut^{A})^{n}}{1-t_1^{n}}
\cdot \frac{1-(vt^{-A})^n}{1-t_2^n} \biggr).
\een
This concludes our operator reformulation of the observation in Wang-Zhou \cite{Wang-Zhou1}.

\subsection{The Macdonald operators $D_n^r$}

We recall some facts about Macdonald's operators \cite[Chapter VI]{Macdonald}
as summarized in \cite[Appendix]{Wang-Zhou1}.
Denote by $\Lambda_n$ the space of symmetric polynomials in $n$ variables $x_1, \dots, x_n$.
For every function $f$ in $x_1, \dots, x_n$,
$q \in \bC$, $1 \leq i \leq n$,
define the shift operator $T_{q, x_i}$ by
$$(T_{u, x_i} f)(x_1, \dots, x_n)
= f(x_1, \dots, ux_i, \dots, x_n).$$
Define the {\em Macdonald operators $D_n^r$} on $\Lambda_n$ by \cite[p.316]{Macdonald}:
$$D_n^r = \sum_{I=\{1 \leq i_1 < \cdots < i_r \leq n\}}  A_I(x_1, \dots, x_n;t)\prod_{i \in I} T_{q,x_i},$$
where
$$A_I(x_1, \dots x_n; t) = t^{r(r-1)/2} \prod_{i \in I, j \in I^c} \frac{tx_i-x_j}{x_i-x_j}.$$
Here $I^c = \{1, \dots, n\} - I$.
The Macdonald polynomials $\{P_{\mu}(x_1, \dots, x_n; q, t)\}$
are suitable normalized symmetric functions such that
(cf. \cite[Chapter VI, Section 4]{Macdonald}):
\begin{eqnarray} \label{eqn:MacDef}
&& D_n^r P_{\mu}(x_1, \dots, x_n;q,t)
= t^{nr} e_r(q^{\mu_1}t^{-1}, \dots, q^{\mu_n}t^{-n})
\cdot P_{\mu}(x_1,\dots, x_n;q,t).
\end{eqnarray}
It turns out that
\be
P_\lambda(x_1, \dots, x_n;q,t)
= m_\lambda(x_1, \dots, x_n)+\sum_{\mu < \lambda} u_{\lambda\mu}m_\mu(x_1, \dots, x_n),
\ee
where the coefficients $u_{\lambda\mu}$ are independent of $n$,
and since
\be
m_\mu(x_1, \dots, x_{n-1}, 0) = m_\mu(x_1, \dots, x_{n-1})
\ee
for $n > l(\mu)$,
hence one has
\be
P_\lambda(x_1, \dots, x_{n-1}, 0; q, t)
= P_\lambda(x_1, \dots, x_{n-1}; q, t)
\ee
for $n > l(\lambda)$.
Hence one can define $P_\lambda(x_1, \dots, x_n, \dots;q,t)$
as a symmetric function in infinitely many variables $x_1, \dots, x_n, \dots$.

\subsection{The modified Macdonald operator $E$}

The operators $D^r_n$ do not have stability property,
instead,
\be \label{eqn:D-res}
D^r_n|_{x_n =0} = t^rD_{n-1}^r + t^{r-1}D_{n-1}^{r-1},
\ee
where $D_n^0=1$,
so one needs to modify them to take $n \to \infty$.
For $r=1$ this was done by Macdonald \cite[Section VI.4]{Macdonald}.
Define an operator $E$ on $\Lambda$ whose restriction to $\Lambda_n$ is given by:
$$E_n = t^{-n} D_n^1(x) -  \sum_{i=1}^n t^{-i}.$$
It is easy to see that \cite[Lemma 1]{Cheng-Wang}:
\be \label{eqn:Arm-Leg}
\sum_{i=1}^{n} q^{\mu_i}t^{-i} - \sum_{i=1}^n t^{-i}
=\frac{q-1}{t} \sum_{(i,j) \in \mu} t^{-(i-1)}q^{j-1}
= \frac{q-1}{t} \sum_{s \in \mu} t^{-l'(s)}q^{a'(s)}.
\ee
Hence we have
\begin{eqnarray} \label{eqn:E}
&& EP_{\mu}(x; q, t) = \frac{q-1}{t} \sum_{s \in \mu} t^{-l'(s)}q^{a'(s)}
\cdot P_{\mu}(x; q, t).
\end{eqnarray}
It follows that
\begin{eqnarray} \label{eqn:E-2}
&& EP_{\mu}(x; t_2, t_1^{-1}) = t_1(t_2-1)
\sum_{s \in \mu} t_1^{l'(s)}t_2^{a'(s)}
\cdot P_{\mu}(x; t_2, t_1^{-1}).
\end{eqnarray}

\subsection{Vertex operator representation of the
modified Macdonald operator $E$}

The action of the operator $D_n^1$ on Newton polynomials
is given by \cite[(63)]{Gar-Hai}:
\begin{equation} \label{eqn:Dn1}
\begin{split}
& D_n^1p_{\mu}(x_1, \dots, x_n)
= \frac{1}{1-t}p_{\mu}(x_1, \dots, x_n) \\
& + \frac{t^n}{t-1} \left[\prod_{i=1}^n \frac{1-zx_i}{1-zx_it} \prod_{i=1}^{l(\mu)}
\big( p_{\mu_i}(x_1,\dots, x_n)+\frac{q^{\mu_i}-1}{(tz)^{\mu_i}}\big)\right]_0,
\end{split}
\end{equation}
where for a formal power series $f(z) = \sum_{n \in \bZ} b_n z^n$,
$f(z)_0 = b_0$.
This can be seen as follows.
Since
\be
D_n^1 = \sum_{i=1}^n \prod_{j \neq i}
\frac{tx_i-x_j}{x_i-x_j} T_{q,x_i},
\ee
we have
\ben
D_n^1p_\mu(x_1, \dots, x_n)
& = & \sum_{i=1}^n \prod_{j\neq i} \frac{tx_i-x_j}{x_i-x_j}
\prod_{k=1}^{l(\mu)}
(x_1^{\mu_k} + \cdots + x_n^{\mu_k} + (q^{\mu_k} - 1) x_i^{\mu_k}).
\een
The idea is to express each term on the right-hand side as the residue of a function
and apply Cauchy's residue formula.
We take
$$f(z) = \frac{1}{z}\prod_{j=1}^n \frac{1-x_jz}{1-tx_jz}
\prod_{k=1}^{l(\mu)}
\biggl(x_1^{\mu_k} + \cdots + x_n^{\mu_k} + \frac{q^{\mu_k} - 1}{(tz)^{\mu_k}} \biggr).
$$
This function has poles at $z=0$, $z=\frac{1}{tx_i}$, $i=1, \dots, n$.
For $R> 0$ large enough,
one has
\ben
&& \frac{1}{2\pi i} \int_{|z|=R} f(z) dz \\
& = & - \frac{1}{2\pi i} \int_{|w|=\frac{1}{R}}
w \cdot \prod_{j=1}^n \frac{w-x_j}{w-tx_j} \prod_{k=1}^{l(\mu)}
\biggl(x_1^{\mu_k}+\cdots+x_n^{\mu_k}+ (q^{\mu_k}-1)t^{-\mu_k} w^{\mu_k}\biggr)
\frac{-dw}{w^2} \\
& = & t^{-n} \prod_{k=1}^{l(\mu)}
\biggl(x_1^{\mu_k}+\cdots+x_n^{\mu_k}\biggr).
\een
On the other hand,
by Cauchy Residue formula,
\ben
&& \frac{1}{2\pi i} \int_{|z|=R} f(z) dz \\
& = & \biggl[\prod_{j=1}^n \frac{1-x_jz}{1-tx_jz}
\prod_{k=1}^{l(\mu)}
\biggl(x_1^{\mu_k} + \cdots + x_n^{\mu_k} + \frac{q^{\mu_k} - 1}{(tz)^{\mu_k}} \biggr)\biggr]_{z^0} \\
& + & \sum_{i=1}^n \frac{t-1}{t^n} \prod_{j \neq i} \frac{tx_i-x_j}{x_i-x_j}
\cdot \prod_{k=1}^{l(\mu)}
\biggl(x_1^{\mu_k} + \cdots + x_n^{\mu_k} + (q^{\mu_k} - 1)x_i^{\mu_k} \biggr).
\een
This completes the proof of \eqref{eqn:Dn1}.
It follows that the action of $E$ on $p_\mu$ is given by:
\begin{equation} \label{eqn:Epmu}
\begin{split}
& Ep_{\mu}(x_1, \dots, x_n, \dots)  \\
& = \frac{1}{t-1} \left[\prod_{i=1}^\infty \frac{1-zx_i}{1-zx_it}
\cdot \prod_{i=1}^{l(\mu)}
\left( p_{\mu_i}(x_1,\dots, x_n, \dots)+\frac{q^{\mu_i}-1}{(tz)^{\mu_i}}\right)
 \right]_{z^0} \\
& - \frac{1}{t-1} p_\mu(x_1, \dots, x_n, \dots).
\end{split}
\end{equation}
Since we are taking the coefficients of $z^0$,
the answer will not be changed if we change $z$ to $z/t$.
Now we are led to consider the operator on $\Lambda_{q,t}$:
\be
p_{\mu} \mapsto \prod_{i=1}^\infty \frac{1-t^{-1}x_iz}{1-x_iz} \cdot
\prod_{i=1}^{l(\mu)}
\left( p_{\mu_i}+\frac{q^{\mu_i}-1}{z^{\mu_i}}\right).
\ee
The right-hand side can be written in the following way
\be
\exp \biggl( \sum_{n=1}^\infty \frac{1-t^{-n}}{n} p_nz^n\biggr) \cdot
\exp \biggl( -\sum_{n=1}^\infty \frac{1-q^{n}}{n} \cdot n\frac{\pd}{\pd p_n}z^{-n}
\biggr) p_\mu,
\ee
so Macdnald's operator $E$ can be written in term of vertex operator as follows:
\be
E = \frac{1}{t-1}
\biggl[\exp \biggl( \sum_{n=1}^\infty \frac{1-t^{-n}}{n} p_nz^n\biggr) \cdot
\exp \biggl( -\sum_{n=1}^\infty \frac{1-q^{n}}{n} \cdot n\frac{\pd}{\pd p_n}z^{-n}
\biggr) -1\biggr]\biggl|_{z^0}.
\ee
This vertex operator appears in \cite[(32)]{AMOS}.
It also  appears in a different form in the study of Macdonald
polynomials by Garsia and Haiman \cite[(73)]{Gar-Hai},
as remarked by Cheng and Wang \cite[Remark 6]{Cheng-Wang}
who rephrased it in terms of vertex operator.

\subsection{Higher modified Macdonald operators}

From \eqref{eqn:MacDef} it is natural to consider $t^{-rn}D_n^r$,
then \eqref{eqn:D-res} becomes
\be \label{eqn:D-res-2}
t^{-rn}D^r_n|_{x_n =0} = t^{-r(n-1)}D_{n-1}^r + t^{-n} \cdot t^{-(r-1)(n-1)}D_{n-1}^{r-1},
\ee
So one is led to consider a combination of the form
\be
E^r_n:=\sum_{j=0}^{r} c_{j,n}(t) t^{-(r-j)n} D_n^{r-j},
\ee
where $c_{j,n}(t)$ is a function in $t$, parameterized by $j, n$,
and $c_{0,n}(t) = 1$.
If we require that
\be
E_n^r|_{x_n=0} = E^r_{n-1},
\ee
then from
\ben
E^r_n|_{x_n=0}
& = & (t^{-rn} D_n^r + \sum_{j=1}^{r} c_{j,n}(t) t^{-(r-j)n} D_n^{r-j})|_{x_n=0} \\
& = & \sum_{j=0}^{r-1} c_{j,n}(t)  (t^{-(r-j)(n-1)}D_{n-1}^{r-j}
+ t^{-n} \cdot t^{-(r-j-1)(n-1)}D_{n-1}^{r-j-1})
+ c_{r,n}(t) \\
& = & t^{-r(n-1)} D_{n-1}^r + \sum_{j=1}^{r} (c_{j,n}(t) + t^{-n}c_{j-1,n}(t))
t^{-(r-j)(n-1)}D_{n-1}^{r-j} \\
& = & D_{n-1}^r + \sum_{j=0}^{r-1} c_{j,n-1}(t) \cdot t^{-jr}D_{n-1}^j.
\een
One has the following recursion relation:
\be
c_{j,n}(t) = c_{j,n-1}(t) - t^{-n}c_{j-1, n}(t),
\ee
for $j=0, \dots, r-1$.
This can be solved with the initial values $c_{r,n}(t) = 1$
and $c_{j,0}(t) = 0$.
The following are some examples:
\ben
c_{1,n}(t) & = & -( t^{-1} + \cdots + t^{-n})= -t^{-1}e_1(1, t^{-1}, \dots, t^{-(n-1)}) , \\
c_{2,n}(t) & = & t^{-1} \cdot t^{-1} + t^{-2} \cdot (t^{-1} + t^{-2}) + \cdots
+ t^{-n}(t^{-1} + \cdots + t^{-n}) \\
& = & t^{-1} e_2(1, t^{-1}, \dots, t^{-n}), \\
c_{3,n}(t) & = & - t^{-1} \cdot t^{-1} e_2(1,t^{-1})
- t^{-2} \cdot t^{-1} e_2(1, t^{-1}, t^{-2}) \\
&& - \cdots -t^{-n} \cdot t^{-1} e_2(1,t^{-1}, \dots, t^{-n}) \\
& = & -e_3(1, t^{-1}, \dots, t^{-(n+1)}), \\
c_{4,n}(t) & = & t^{-1} \cdot e_3(1,t^{-1}, t^{-2})
+ t^{-2} \cdot e_3(1, t^{-1}, \dots, t^{-3}) + \cdots \\
&& + t^{-n} \cdot e_3(1, t^{-1}, \dots, t^{-(n+1)}) \\
& = & t^2 \cdot e_4(1, t^{-1}, \dots, t^{-(n+2)}), \\
c_{5,n}(t) & = & -t^{-1} \cdot t^2e_4(1,t^{-1},\cdot, t^{-3})
- t^{-2} \cdot t^2e_4(1, t^{-1}, \dots, t^{-4}) + \cdots \\
&& - t^{-n} \cdot t^2 e_4(1, t^{-1}, \dots, t^{-(n+2)}) \\
& = & -t^5 \cdot e_5(1, t^{-1}, \dots, t^{-(n+3)}).
\een
In general
\be
c_{j,n}(t) = (-1)^j t^{(j^2-3j)/2} e_j(1, t^{-1}, \dots, t^{-(n+j-2)}).
\ee
So we have proved the following:

\begin{prop}
Let the operator $E^r_n$ be defined by:
\be
E^r_n:= \sum_{j=0}^{r} (-1)^j t^{(j^2-3j)/2} e_j(1, t^{-1}, \dots, t^{-(n+j-2)})
\cdot  t^{-(r-j)n} D_n^{r-j}.
\ee
Then for $n$ large enough one has
\be
E^r_n|_{x_n=0} = E^r_{n-1}.
\ee
The Macdonald polynomials $P_\mu(x_1, \dots, x_n;q,t)$ are eigenvectors of $E_n^r$
with eigenvalues:
\be
e_n^r(\mu):=\sum_{j=0}^{r} (-1)^j t^{\frac{j^2-3j}{2}} e_j(1, t^{-1}, \dots, t^{-(n+j-2)})
\cdot e_{r-j}(q^{\mu_1}t^{-1}, \dots, q^{\mu_n}t^{-n}),
\ee
and the Macdonald functions $P_\mu(x_1, \dots, x_n, \dots;q,t)$ are eigenvectors of $E^r$
with eigenvalues:
\be
e^r(\mu):=\sum_{j=0}^{r} (-1)^j t^{\frac{j^2-3j}{2}} e_j(1, t^{-1}, \dots, t^{-n}, \dots)
\cdot e_{r-j}(q^{\mu_1}t^{-1}, \dots, q^{\mu_n}t^{-n}, \dots).
\ee
\end{prop}

We now understand the eigenvalues of $E_n^r$ and $E_n$ by
some basic results in basic hypergeometric series.
Recall Gauss $q$-binomial formula:
\bea
&& \prod_{j=0}^{n-1} (1+q^j z)
= \sum_{j=0}^n q^{j(j-1)/2} \binom{n}{j}_q z^j, \label{eqn:Gauss1} \\
&& \prod_{j=0}^{n-1} \frac{1}{1-q^jz} =
\sum_{j=0}^\infty \binom{n+j-1}{j}_q z^j,  \label{eqn:Gauss2}
\eea
where the coefficients on the right-hand side
are defined by:
\be
\binom{n}{j}_q
= \frac{(1-q^n)(1-q^{n-1}) \cdots (1-q^{n-j+1})}{(1-q) \cdots (1-q^j)},
\ee
one gets:
\be
e_j(1, q, \dots, q^{n-1})
= q^{j(j-1)/2} \cdot  \frac{(1-q^n)(1-q^{n-1}) \cdots (1-q^{n-j+1})}{(1-q) \cdots (1-q^j)}.
\ee
It follows that
\ben
&& (-1)^j t^{(j^2-3j)/2} e_j(1, t^{-1}, \dots, t^{-(n+j-2)}) \\
%% & = & (-1)^j t^{(j^2-3j)/2} \cdot t^{-j(j-1)/2}
%% \cdot \frac{(1-t^{-(n+j-2)})\cdots (1-t^{-(n-1)})}{(1-t^{-1}) \cdots (1-t^{-j})}
& = & (-1)^j t^{-j}
\cdot \frac{(1-t^{-(n+j-2)})\cdots (1-t^{-(n-1)})}{(1-t^{-1}) \cdots (1-t^{-j})} \\
& = & (-1)^j t^{-j} \binom{n+j-2}{j}_{t^{-1}}.
\een
It follows that
\ben
&& \sum_{r=0}^\infty e^r_n(\mu) z^r
= \sum_{r=0}^\infty z^r \sum_{j=0}^r (-1)^j t^{-j} \binom{n+j-2}{j}_{t^{-1}}
\cdot  e_{r-j}(q^{\mu_1}t^{-1}, \dots, q^{\mu_n}t^{-n}) \\
& = & \sum_{j=0}^\infty (-1)^j t^{-j} \binom{n+j-2}{j}_{t^{-1}} z^j
\cdot  \sum_{k=0}^n e_{k}(q^{\mu_1}t^{-1}, \dots, q^{\mu_n}t^{-n}) z^k \\
& = & \frac{\prod_{j=1}^n (1+q^{\mu_j}t^{-j}z)}{\prod_{j=1}^{n-1}(1+t^{-j}z)}.
\een
Taking $n \to \infty$:
\be \label{eqn:Eigen-E}
 \sum_{r=0}^\infty e^r(\mu) z^r
 = \prod_{j=1}^\infty  \frac{1+q^{\mu_j}t^{-j}z}{1+t^{-j}z}.
\ee

One can also require that the operators $E_n^r$ has the Macdonald polynomials
$P_\mu(x;q,t)$ as eigenvectors,
with eigenvalues $e_r(q^{\mu_1}t^{-1}, \dots, q^{\mu_n}t^{-n}, \dots)$.
I.e.,
\ben
\sum_{j=0}^r c_{j,n}(t) \cdot e_{r-j}(q^{\mu_1}t^{-1}, \dots, q^{\mu_n}t^{-n})
= e_r(q^{\mu_1}t^{-1}, \dots, q^{\mu_n}t^{-n}, \dots).
\een
This is very easy to achieve,
it suffices to take $c_{j,n}(t) = e_j(t^{-n-1}, t^{-n-2}, \dots)$.
In fact,
\ben
&& \sum_{r=0}^\infty z^r \sum_{j=0}^r e_j(t^{-n-1}, t^{-n-2}, \dots)
\cdot e_{r-j}(q^{\mu_1}t^{-1}, \dots, q^{\mu_n}t^{-n}) \\
& = & \sum_{j=0}^\infty z^j e_j(t^{-n-1}, t^{-n-2}, \dots)
\cdot \sum_{k=0}^\infty e_k(q^{\mu_1}t^{-1}, \dots, q^{\mu_n}t^{-n}) \\
& = & \prod_{j=n+1}^\infty (1+zt^{-j}) \cdot \prod_{k=1}^n (1+zq^{\mu_k}t^{-k}) \\
& = & \prod_{j=1}^\infty (1+zq^{\mu_j}t^{-j}) \\
& = & \sum_{r=0}^\infty z^r e_r(q^{\mu_1}t^{-1}, \dots, q^{\mu_n}t^{-n}, \dots).
\een
So we have recovered \cite[(C.5),(C.7)]{Awata-Kanno}:

\begin{prop}
Let the operator $\tilde{E}^r_n$ be defined by:
\be
\tilde{E}^r_n:= \sum_{j=0}^{r} e_j(t^{-n-1}, t^{-n-2}, \dots)
\cdot  t^{-(r-j)n} D_n^{r-j}.
\ee
Then for $n$ large enough,
the Macdonald polynomials $P_\mu(x;q,t)$ are eigenvectors of these operators,
with eigenvalues
$e_r(q^{\mu_1}t^{-1}, \dots, q^{\mu_n}t^{-n}, \dots)$.
\end{prop}

\subsection{Vertex operator realization of $\tilde{E}_n^r$}

In this subsection we will recall the vertex operator realization of $\tilde{E}_n^r$
first derived by Shiraishi \cite{Shiraishi}
and reformulated by Awata and Kanno \cite{Awata-Kanno}.
We will make a new derivation by combining some of their ideas.
Our derivation is completely elementary.

The generating series of the eigenvalues of $\tilde{E}_n^r$ is
\be
\sum_{r=0}^\infty z^r e_r(q^{\mu_1}t^{-1}, \dots, q^{\mu_n}t^{-n}, \dots)
= \prod_{j=1}^\infty (1+q^{\mu_j}t^{-j}z),
\ee
hence by comparing with \eqref{eqn:Eigen-E},
we have
\be
\sum_{r=0}^\infty z^r \tilde{E}^r
= \prod_{j=1}^\infty (1+t^{-j}z^j) \cdot \sum_{s=0}^\infty z^s E^s,
\ee
and
so
\be
\tilde{E}^r = \sum_{j=0}^r e_j(t^{-1}, t^{-2}, \dots) \cdot E^{r-j}.
\ee
By Euler's formula (see e.g. \cite{And}):
\be \label{eqn:Euler1}
\prod_{n=1}^\infty (1+q^nz) = 1+\sum_{n=1}^\infty \frac{q^{n(n+1)/2}}{\prod_{j=1}^n (1-q^j)}z^n,
\ee
one gets
\be
e_n(q,q^2, \dots) = \frac{q^{n(n+1)/2}}{\prod_{j=1}^n (1-q^j)},
\ee
so one has
\be
e_j(t^{-1}, t^{-2}, \dots) = \frac{t^{-j(j+1)/2}}{\prod_{j=1}^n (1-t^{-j})}
= \prod_{j=1}^n \frac{1}{t^j-1},
\ee
and so
\be
\tilde{E}^r = \sum_{j=0}^r\prod_{j=1}^n \frac{1}{t^j-1} \cdot E^{r-j}.
\ee
One also has
\ben
e_j(t^{-n-1}, t^{-n-2}, \dots)
&= & t^{-jn} e_j(t^{-1}, t^{-2}, \dots)
= t^{-jn} \cdot \frac{t^{-j(j+1)/2}}{\prod_{a=1}^j (1-t^{-a})} \\
& = & \frac{t^{-jn}}{\prod_{a=1}^j (t^a-1)}.
\een
It follows that
\ben
\tilde{E}^r_n & = & \frac{t^{-rn}}{\prod_{a=1}^r (t^a-1)} \cdot
\sum_{k=0}^r \prod_{i=0}^{k-1} (t^{r-i}-1) \cdot D^k_n \\
& = & t^{-rn} e_r(t^{-1}, t^{-2}, \dots) \cdot
\sum_{k=0}^r \prod_{i=0}^{k-1} (t^{r-i}-1) \cdot D^k_n.
\een
(See \cite[(78)]{Shiraishi} and  \cite[(D.9)]{Awata-Kanno}.)
From this one can now derive the vertex operator realization of $\tilde{E}^r$
in \cite{Shiraishi, Awata-Kanno}.
The idea is to consider the action of $t^{-rn}\cdot \sum_{k=0}^r
\prod_{i=0}^{k-1} (t^{r-i}-1) \cdot D^k_n$ on $p_\mu$
and rewrite it as the constant terms of some Laurent series.
From the definition of $D_n^r$,
we have
\be
\begin{split}
& t^{-rn} \sum_{k=0}^r \prod_{i=0}^{k-1} (t^{r-i}-1) \cdot D_n^k p_\mu(x_1, \dots, x_n) \\
 = & t^{-rn} \sum_{k=0}^r \prod_{i=0}^{k-1} (t^{r-i}-1) \cdot
\sum_{I=\{1 \leq i_1 < \cdots < i_k \leq n\}}
t^{k(k-1)/2} \prod_{i \in I, j \in I^c} \frac{tx_i-x_j}{x_i-x_j} \\
&  \cdot
\prod_{a=1}^{l(\mu)}
(x_1^{\mu_a}+ \cdots + x_n^{\mu_a}+ (q^{\mu_a}-1)
\sum_{j=1}^k x_{i_j}^{\mu_a}).
\end{split}
\ee
Consider the constant term of
\ben
f_r(z_1, \dots, z_r)
& = &
\prod_{\alpha=1}^r \prod_{j=1}^n
\frac{t - x_jz_\alpha}{1 - x_jz_\alpha}  \cdot
\prod_{1 \leq \alpha<\beta \leq r} \frac{1 - z_\alpha/z_\beta}{1 - z_\alpha/(z_\beta t)} \\
&& \cdot \prod_{a=1}^{l(\mu)} (\sum_{j=1}^n x_j^{\mu_a}
+ (q^{\mu_a}-1) \sum_{\alpha=1}^r z_{\alpha}^{-\mu_a} ).
\een
where $\prod_{1 \leq \alpha<\beta \leq r} \frac{1 - z_\alpha/z_\beta}{1 - z_\alpha/(z_\beta t)} $
is understood as the series
\be
\prod_{1\leq \alpha < \beta \leq r}
(1 + \sum_{n=1}^\infty t^{-n} (1-t) (z_\alpha/z_\beta)^n).
\ee
The constant term can be computed recursively as follows.
First,
\ben
&& f_r(z_1, \dots, z_r)|_{z_r^0} \\
& = & \Res_{z_r=\infty} (f_r(z_1, \dots, z_r)\frac{dz_r}{z_r} )
- \sum_{i=1}^n \Res_{z_r=1/x_i} (f_r(z_1, \dots, z_r)\frac{dz_r}{z_r} ).
\een
Since one has
\ben
f_r(z_1, \dots, z_r)
& = & \frac{1}{z_r}\prod_{j=1}^n
\frac{x_j- t/z_r}{x_j-1/z_r} \cdot
\prod_{1\leq \alpha < r} \frac{1 - z_\alpha/z_r}{1 - z_\alpha/(z_r t)} \\
&& \cdot\prod_{\alpha=1}^{r-1} \biggl(\frac{1}{z_\alpha}\prod_{j=1}^n
\frac{t - x_jz_\alpha}{1 - x_jz_\alpha} \biggr) \cdot
\prod_{\alpha<\beta<r} \frac{1 - z_\alpha/z_\beta}{1 - z_\alpha/(z_\beta t)} \\
&& \cdot \prod_{a=1}^{l(\mu)} (\sum_{j=1}^n x_j^{\mu_a}
+ (q^{\mu_a}-1) \sum_{\alpha=1}^{r-1} z_{\alpha}^{-\mu_a}
+ (q^{\mu_a}-1) z_r^{-\mu_a}),
\een
and so
\ben
\Res_{z_r=\infty} f(z_1, \dots, z_r) \frac{dz_r}{z_r}
& = & \prod_{\alpha=1}^{r-1} \biggl(\frac{1}{z_\alpha}\prod_{j=1}^n
\frac{t - x_jz_\alpha}{1 - x_jz_\alpha} \biggr) \cdot
\prod_{\alpha<\beta<r} \frac{1 - z_\alpha/z_\beta}{1 - z_\alpha/(z_\beta t)} \\
&& \cdot \prod_{a=1}^{l(\mu)} (\sum_{j=1}^n x_j^{\mu_a}
+ (q^{\mu_a}-1) \sum_{\alpha=1}^{r-1} z_{\alpha}^{-\mu_a}) \\
& = & f_{r-1}(z_1, \dots, z_{r-1}).
\een
We also have
\ben
&& \sum_{i=1}^n \Res_{z_r =1/x_i} f(z_1, \dots, z_r) \frac{dz_r}{z_r} \\
& = & (1-t) \sum_{i=1}^n
\prod_{j \neq i} \frac{tx_i-x_j}{x_i-x_j} \cdot
\prod_{\alpha=1}^{r-1} \frac{1-x_i z_\alpha}{1-x_iz_\alpha/t}
\cdot
\prod_{\alpha=1}^{r-1}  \prod_{j=1}^n
\frac{t - x_jz_\alpha}{1 - x_jz_\alpha}   \cdot
\prod_{\alpha<\beta<r} \frac{1 - z_\alpha/z_\beta}{1 - z_\alpha/(z_\beta t)} \\
&& \cdot \prod_{a=1}^{l(\mu)} (\sum_{j=1}^n x_j^{\mu_a}
+ (q^{\mu_a}-1) \sum_{\alpha=1}^{r-1} z_{\alpha}^{-\mu_a} + (q^{\mu_a}-1)x_i^{\mu_a}) \\
& = & t^{r-1}(1-t) \sum_{i=1}^n
\prod_{j \neq i} \frac{tx_i-x_j}{x_i-x_j} \cdot
\prod_{\alpha=1}^{r-1}  \prod_{j\neq i}
\frac{t - x_jz_\alpha}{1 - x_jz_\alpha}   \cdot
\prod_{\alpha<\beta<r} \frac{1 - z_\alpha/z_\beta}{1 - z_\alpha/(z_\beta t)} \\
&& \cdot \prod_{a=1}^{l(\mu)} (\sum_{j=1}^n x_j^{\mu_a}
+ (q^{\mu_a}-1) \sum_{\alpha=1}^{r-1} z_{\alpha}^{-\mu_a} + (q^{\mu_a}-1)x_i^{\mu_a}).
\een
Therefore,
one has the following recursion relation:
\ben
&& f_r(z_1, \dots, z_r)|_{z_r^0}
= f_{r-1}(z_1, \dots, z_{r-1}) \\
& + & t^{r-1}(t-1) \sum_{i=1}^n
\prod_{j \neq i} \frac{tx_i-x_j}{x_i-x_j} \cdot
\prod_{\alpha=1}^{r-1}  \prod_{j\neq i}
\frac{t - x_jz_\alpha}{1 - x_jz_\alpha}   \cdot
\prod_{1\leq \alpha<\beta\leq r-1} \frac{1 - z_\alpha/z_\beta}{1 - z_\alpha/(z_\beta t)} \\
&& \cdot \prod_{a=1}^{l(\mu)} (\sum_{j=1}^n x_j^{\mu_a}
+ (q^{\mu_a}-1) \sum_{\alpha=1}^{r-1} z_{\alpha}^{-\mu_a} + (q^{\mu_a}-1)x_i^{\mu_a}).
\een
Repeating the above computations for $z_{r-1}, \cdots, z_1$,
one gets:
\ben
&& f_r(z_1, \dots, z_r)|_{z_r^0}|_{z_{r-1}^0}
= f_{r-2}(z_1, \dots, z_{r-2}) \\
& + & t^{r-2}(t-1) \sum_{i=1}^n
\prod_{j \neq i} \frac{tx_i-x_j}{x_i-x_j} \cdot
\prod_{\alpha=1}^{r-2}  \prod_{j\neq i}
\frac{t - x_jz_\alpha}{1 - x_jz_\alpha}   \cdot
\prod_{1\leq \alpha<\beta\leq r-2} \frac{1 - z_\alpha/z_\beta}{1 - z_\alpha/(z_\beta t)} \\
&& \cdot \prod_{a=1}^{l(\mu)} (\sum_{j=1}^n x_j^{\mu_a}
+ (q^{\mu_a}-1) \sum_{\alpha=1}^{r-2} z_{\alpha}^{-\mu_a} + (q^{\mu_a}-1)x_i^{\mu_a}) \\
& + & t^{r-1}(t-1) \cdot \biggl[  \sum_{i=1}^n
\prod_{j \neq i} \frac{tx_i-x_j}{x_i-x_j} \cdot
\prod_{\alpha=1}^{r-2}  \prod_{j\neq i}\frac{t - x_jz_\alpha}{1 - x_jz_\alpha}
\cdot
\prod_{1\leq \alpha<\beta\leq r-2} \frac{1 - z_\alpha/z_\beta}{1 - z_\alpha/(z_\beta t)} \\
&&  \cdot
\prod_{a=1}^{l(\mu)} (\sum_{j=1}^n x_j^{\mu_a}
+ (q^{\mu_a}-1) \sum_{\alpha=1}^{r-2} z_{\alpha}^{-\mu_a} + (q^{\mu_a}-1)x_i^{\mu_a}) \\
& + & \sum_{i=1}^n
\prod_{j \neq i} \frac{tx_i-x_j}{x_i-x_j} \cdot \sum_{k\neq i} (t-1)
\prod_{j\neq i,k}  \frac{t-x_j/x_k}{1-x_j/x_k}
\prod_{\alpha=1}^{r-2} \prod_{j\neq i}  \frac{t - x_jz_\alpha}{1 - x_jz_\alpha}   \\
&& \cdot \prod_{1\leq \alpha\leq r-2} \frac{1 - x_k z_\alpha}{1 - x_k z_\alpha/t}
\cdot \prod_{1\leq \alpha<\beta\leq r-2}  \frac{1 - z_\alpha/z_\beta}{1 - z_\alpha/(z_\beta t)} \\
&& \cdot \prod_{a=1}^{l(\mu)} (\sum_{j=1}^n x_j^{\mu_a}
+ (q^{\mu_a}-1) \sum_{\alpha=1}^{r-2} z_{\alpha}^{-\mu_a}
+ (q^{\mu_a}-1)(x_k^{\mu_a}+x_i^{\mu_a})) \biggr].
\een
After some simplifications one gets:
\ben
&& f_r(z_1, \dots, z_r)|_{z_r^0}|_{z_{r-1}^0}
= f_{r-2}(z_1, \dots, z_{r-2}) \\
& + & (t^{r-1}+t^{r-2})(t-1) \sum_{i=1}^n
\prod_{j \neq i} \frac{tx_i-x_j}{x_i-x_j} \cdot
\prod_{\alpha=1}^{r-2}  \prod_{j\neq i}
\frac{t - x_jz_\alpha}{1 - x_jz_\alpha}   \cdot
\prod_{1\leq \alpha<\beta\leq r-2} \frac{1 - z_\alpha/z_\beta}{1 - z_\alpha/(z_\beta t)} \\
&& \cdot \prod_{a=1}^{l(\mu)} (\sum_{j=1}^n x_j^{\mu_a}
+ (q^{\mu_a}-1) \sum_{\alpha=1}^{r-2} z_{\alpha}^{-\mu_a} + (q^{\mu_a}-1)x_i^{\mu_a}) \\
& + & t^{r-1}t^{r-2}(t-1)^2 \cdot
\sum_{1\leq i \neq k \leq n} \frac{tx_i-x_k}{x_i-x_k} \cdot
\prod_{j \neq i, k} \frac{tx_i-x_j}{x_i-x_j} \frac{tx_k-x_j}{x_k-x_j} \\
&& \cdot
\prod_{\alpha=1}^{r-2} \prod_{j\neq i,k}  \frac{t - x_jz_\alpha}{1 - x_jz_\alpha}
\cdot
\prod_{1\leq \alpha<\beta\leq r-2}  \frac{1 - z_\alpha/z_\beta}{1 - z_\alpha/(z_\beta t)} \\
&& \cdot \prod_{a=1}^{l(\mu)} (\sum_{j=1}^n x_j^{\mu_a}
+ (q^{\mu_a}-1) \sum_{\alpha=1}^{r-2} z_{\alpha}^{-\mu_a}
+ (q^{\mu_a}-1)(x_k^{\mu_a}+x_i^{\mu_a})).
\een
For the summation $\sum_{1\leq i \neq k\leq n}$,
one considers two cases: When $i < k$,
set $i_1 = i$ and $i_2 = k$;
when $i> k$, set $i_1 = k$ and $i_2=i$.
Note
\be
\frac{tx_{i_1} - x_{i_2}}{x_{i_1} - x_{i_2}}
+ \frac{tx_{i_2} - x_{i_1}}{x_{i_2} - x_{i_1}}
= t+1.
\ee
So we get:
\ben
&& f_r(z_1, \dots, z_r)|_{z_r^0}|_{z_{r-1}^0}
= f_{r-2}(z_1, \dots, z_{r-2}) \\
& + & (t^{r-1}+t^{r-2})(t-1) \sum_{i=1}^n
\prod_{j \neq i} \frac{tx_i-x_j}{x_i-x_j} \cdot
\prod_{\alpha=1}^{r-2}  \prod_{j\neq i}
\frac{t - x_jz_\alpha}{1 - x_jz_\alpha}   \cdot
\prod_{1\leq \alpha<\beta\leq r-2} \frac{1 - z_\alpha/z_\beta}{1 - z_\alpha/(z_\beta t)} \\
&& \cdot \prod_{a=1}^{l(\mu)} (\sum_{j=1}^n x_j^{\mu_a}
+ (q^{\mu_a}-1) \sum_{\alpha=1}^{r-2} z_{\alpha}^{-\mu_a} + (q^{\mu_a}-1)x_i^{\mu_a}) \\
& + & t^{r-1}t^{r-2}(t-1)^2(t+1) \cdot
\sum_{1\leq i_1 < i_2  \leq n}
\prod_{j \neq i_1, i_2} \frac{tx_{i_1}-x_j}{x_{i_1}-x_j} \frac{tx_{i_2}-x_j}{x_{i_2}-x_j} \\
&& \cdot
\prod_{\alpha=1}^{r-2} \prod_{j\neq i_1,i_2}  \frac{t - x_jz_\alpha}{1 - x_jz_\alpha}
\cdot
\prod_{1\leq \alpha<\beta\leq r-2}  \frac{1 - z_\alpha/z_\beta}{1 - z_\alpha/(z_\beta t)} \\
&& \cdot \prod_{a=1}^{l(\mu)} (\sum_{j=1}^n x_j^{\mu_a}
+ (q^{\mu_a}-1) \sum_{\alpha=1}^{r-2} z_{\alpha}^{-\mu_a}
+ (q^{\mu_a}-1)(x_{i_1}^{\mu_a}+x_{i_2}^{\mu_a})).
\een
One can continue this procedure by induction.
For this purpose,
let us introduce some notations.
For a subset $I \subset \{1, 2, \dots, n\}$,
denote by $|I|$ the number of elements in $I$
and by $I^c = \{1, \dots, n\} - I$.
Define:
\ben
f_{r, I}(z_1, \dots, z_r):
& = & \prod_{j \in I^c} \biggl(\prod_{i\in I} \frac{tx_i-x_j}{x_i-x_j} \cdot
\prod_{\alpha=1}^{r} \frac{t-x_jz_\alpha}{1-x_jz_\alpha} \biggr)
\cdot \prod_{1 \leq \alpha < \beta \leq r}
\frac{1-z_\alpha/z_\beta}{1-z_\alpha/(z_\beta t)} \\
&& \cdot \prod_{a=1}^{l(\mu)}
(\sum_{j=1}^n x_j^{\mu_a}
+ (q^{\mu_a}-1) \sum_{\alpha=1}^{r} z_\alpha^{-\mu_a}
+ \sum_{i\in I} (q^{\mu_a} -1) x_i^{\mu_a}).
\een
When $I$ is the empty set $\emptyset$,
write $f_{r, \emptyset}(z_1, \dots, z_r) = f_r(z_1, \dots, z_r)$.
We also have
\ben
&& \Res_{z_r = \infty} f_{r, I}(z_1, \dots, z_r) \frac{dz_r}{z_r} \\
& = & \prod_{j \in I^c} \biggl(\prod_{i\in I} \frac{tx_i-x_j}{x_i-x_j} \cdot
\prod_{\alpha=1}^{r-1} \frac{t-x_jz_\alpha}{1-x_jz_\alpha} \biggr)
\cdot \prod_{1 \leq \alpha < \beta \leq r-1}
\frac{1-z_\alpha/z_\beta}{1-z_\alpha/(z_\beta t)} \\
&& \cdot \prod_{a=1}^{l(\mu)}
(\sum_{j=1}^n x_j^{\mu_a}
+ (q^{\mu_a}-1) \sum_{\alpha=1}^{r-1} z_\alpha^{-\mu_a}
+ \sum_{i\in I} (q^{\mu_a} -1) x_i^{\mu_a}), \\
& = & f_{r-1, I}(z_1, \dots, z_{r-1}).
\een
For $k \in I^c$,
\ben
&& \Res_{z_r=1/x_k} f_{r, I}(z_1, \dots, z_r) \frac{dz_r}{z_r} \\
& = & -(t-1) \prod_{j\in (I\cup\{k\})^c} \frac{t-x_j/x_k}{1-x_j/x_k} \cdot
\prod_{\alpha=1}^{r-1} \frac{1-x_kz_\alpha}{1-x_kz_\alpha/t} \\
&& \cdot \prod_{j \in I^c} \biggl(\prod_{i\in I} \frac{tx_i-x_j}{x_i-x_j} \cdot
\prod_{\alpha=1}^{r-1} \frac{t-x_jz_\alpha}{1-x_jz_\alpha} \biggr)
\cdot \prod_{1 \leq \alpha < \beta \leq r-1}
\frac{1-z_\alpha/z_\beta}{1-z_\alpha/(z_\beta t)} \\
&& \cdot \prod_{a=1}^{l(\mu)}
(\sum_{j=1}^n x_j^{\mu_a}
+ (q^{\mu_a}-1) \sum_{\alpha=1}^{r-1} z_\alpha^{-\mu_a}
+ \sum_{i\in I \cup\{k\}} (q^{\mu_a} -1) x_i^{\mu_a}) \\
& = & - t^{r-1}(t-1) \prod_{i\in I} \frac{tx_i-x_k}{x_i-x_k} \cdot
f_{r-1, I\cup\{k\}}(z_1, \dots, z_{r-1}).
\een
The constant term of $f_{r, I}(z_1, \dots, z_r)$ in $z_r$:
\ben
&& f_{r, I}(z_1, \dots, z_r)|_{z_r^0} \\
& = & \Res_{z_r=\infty} (f_{r,I}(z_1, \dots, z_r)\frac{dz_r}{z_r} )
- \sum_{k\in I^c} \Res_{z_r=1/x_i} (f_r(z_1, \dots, z_r)\frac{dz_r}{z_r} ) \\
& = & f_{r-1, I}(z_1, \dots, z_{r-1}) \\
& + & t^{r-1}(t-1) \sum_{k \in I^c}
\prod_{i\in I} \frac{tx_i-x_k}{tx_i-x_k} \cdot
f_{r-1, I\cup\{k\}}(z_1, \dots, z_{r-1}).
\een
In these notations,
we have proved above the following identities:
\ben
&& f_r(z_1, \dots, z_r)|_{z_r^0}
= f_{r-1}(z_1, \dots, z_{r-1})
+ t^{r-1}(t-1) \sum_{k=1}^n f_{r-1, \{i\}}(z_1, \dots, z_{r-1}), \\
&& f_r(z_1, \dots, z_r)|_{z_r^0}|_{z_{r-1}^0}
= f_{r-2}(z_1, \dots, z_{r-2}) \\
& + & (t^{r-1}+t^{r-2})(t-1) \sum_{|I|=1} f_{r-2, I}(z_1, \dots, z_{r-2}) \\
& + & t^{r-1}t^{r-2}(t-1)^2(t+1) \cdot \sum_{|I|=2}
f_{r-2, I}(z_1, \dots, z_{r-2}).
\een
Inductively,
one can show that
\ben
&& f_r(z_1, \dots, z_r)|_{z_r^0}|_{z_{r-1}^0} \cdots |_{z_{r-s}^0}
= \sum_{l=0}^s \alpha_{j}^{(s)} \sum_{|I|=l}  f_{r-s, I}(z_1, \dots, z_{r-s}).
\een
Taking the constant terms $z_{r-s-1}^0$ on both sides:
\ben
&& \sum_{l=0}^{s+1} \alpha_{l}^{(s+1)} \sum_{|I|=l}  f_{r-s-1, I}(z_1, \dots, z_{r-s-1}) \\
& = & f_r(z_1, \dots, z_r)|_{z_r^0}|_{z_{r-1}^0} \cdots |_{z_{r-s}^0}|_{z_{r-s-1}^0} \\
& = & \sum_{l=0}^s \alpha_{l}^{(s)} \sum_{|I|=l}  f_{r-s, I}(z_1, \dots, z_{r-s})|_{z_{r-s-1}^0} \\
& = & \sum_{l=0}^s \alpha_{l}^{(s)} \sum_{|I|=l}
\biggl(f_{r-s-1, I}(z_1, \dots, z_{r-s-1}) \\
& + & t^{r-s-1}(t-1)\sum_{k \in I^c} \prod_{i \in I}
\frac{tx_i-x_k}{x_i-x_k} \cdot
f_{r-s-1, I\cup \{k\}}(z_1, \dots, z_{r-s-1})\biggr)\\
& = & \sum_{l=0}^s \alpha_{l}^{(s)} \sum_{|I|=l}
f_{r-s-1, I}(z_1, \dots, z_{r-s-1}) \\
& + &  t^{r-s-1}(t-1) \sum_{l=0}^s \alpha_{l}^{(s)} \sum_{|I|=l}\sum_{k \in I^c}
\prod_{i \in I}
\frac{tx_i-x_k}{x_i-x_k} \cdot
f_{r-s-1, I\cup \{k\}}(z_1, \dots, z_{r-s-1}).
\een
Now note:
\ben
&& \sum_{|I|=l}\sum_{k \in I^c} \prod_{i \in I} \frac{tx_i-x_k}{x_i-x_k} \cdot
f_{r-s-1, I\cup \{k\}}(z_1, \dots, z_{r-s-1}) \\
& = & \sum_{|J|=l+1} \biggl( \sum_{k\in J}
\prod_{i \in J-\{k\}} \frac{tx_i-x_k}{x_i-x_k} \biggr)
\cdot f_{r-s-1, J}(z_1, \dots, z_{r-s-1}) \\
& = & \frac{1-t^{l+1}}{1-t}
\cdot \sum_{|J|=l+1}  f_{r-s-1, J}(z_1, \dots, z_{r-s-1}),
\een
where in the last equality we have used an identity:
\be
\sum_{i=1}^n \prod_{1\leq j \leq n, j \neq i}
\frac{tx_j -x_i}{x_j -x_i} = \frac{1-t^n}{1-t}.
\ee
So we have
\ben
&& \sum_{l=0}^{s+1} \alpha_{l}^{(s+1)} \sum_{|I|=l}  f_{r-s-1, I}(z_1, \dots, z_{r-s-1}) \\
& = & \sum_{l=0}^s \alpha_{l}^{(s)} \sum_{|I|=l}
f_{r-s-1, I}(z_1, \dots, z_{r-s-1}) \\
& + &  t^{r-s-1}\sum_{l=0}^s \alpha_{l}^{(s)} \sum_{|J|=l+1}
(t^{l+1}-1) \cdot
f_{r-s-1, J}(z_1, \dots, z_{r-s-1}).
\een
So we get a recursion relation:
\ben
&& \alpha_0^{(s+1)}= \alpha_0^{(s)}, \\
&& \alpha_l^{(s+1)} = \alpha_l^{(s)} + t^{r-s-1}(t^l-1) \alpha_{l-1}^{(s)}, \;\;\;
l=1, \dots, s+1,
\een
with the initial value $\alpha_0^{(0)} = 1$.
An easy solution is given by the elementary symmetric functions:
\be
\alpha^{(s)}_l = \prod_{j=1}^l (t^j-1) \cdot e_l(t^{r-1}, t^{r-2}, \dots, t^{r-s}).
\ee
Hence we get
\be
\alpha^{(r)}_l
= \prod_{j=1}^l (t^j-1) \cdot e_r(t^{r-1},t^{r-2}, \cdots, t, 1)
= t^{l(l-1)/2}
\cdot \prod_{j=1}^l (t^{r-l+1}-1),
\ee
where in the second equality we have used \eqref{eqn:Gauss1}.
To summarize,
we have shown that
\be
\begin{split}
& \tilde{E}^r_n p_\mu(x_1, \dots, x_n) \\
= & e_r(t^{-1}, t^{-2}, \dots) \cdot t^{-rn} \sum_{k=0}^r \prod_{i=0}^{k-1} (t^{r-i}-1) \cdot D_n^k p_\mu(x_1, \dots, x_n) \\
 = & e_r(t^{-1}, t^{-2}, \dots) \cdot
\biggl(\prod_{\alpha=1}^r \prod_{j=1}^n
\frac{1 - x_jz_\alpha/t}{1 - x_jz_\alpha}  \cdot
\prod_{1 \leq \alpha<\beta \leq r} \frac{1 - z_\alpha/z_\beta}{1 - z_\alpha/(z_\beta t)} \\
& \cdot \prod_{a=1}^{l(\mu)} (\sum_{j=1}^n x_j^{\mu_a}
+ (q^{\mu_a}-1) \sum_{\alpha=1}^r z_{\alpha}^{-\mu_a} )\biggr)\biggr|_{z_r^0z_{r-1}^0 \cdots z_1^0}.
\end{split}
\ee
It can be reformulated in terms of vertex operators:
\be \label{eqn:tildeE-vertex}
\begin{split}
\tilde{E}^r
 = & e_r(t^{-1}, t^{-2}, \dots) \cdot
\biggl[
\prod_{1 \leq \alpha<\beta \leq r} \frac{1 - z_\alpha/z_\beta}{1 - z_\alpha/(z_\beta t)}
 \cdot \exp\biggl(
\sum_{n=1}^\infty \frac{1-t^{-n}}{n} p_n \sum_{\alpha=1}^r z_\alpha^n \biggr) \\
&\cdot \exp \biggl(-\sum_{n=1}^\infty \frac{1-q^n}{n} \sum_{\alpha=1}^r z_\alpha^{-n}
\cdot n \frac{\pd}{\pd p_n} \biggr)
\biggr]\biggr|_{z_r^0z_{r-1}^0 \cdots z_1^0}.
\end{split}
\ee
This was obtained by  Awata-Kanno \cite{Awata-Kanno}.
Here we use the notation $|_{z_r^0z_{r-1}^0 \cdots z_1^0}$ to indicate the order
of taking the constant term: One first take the constant term in $z_r$,
then in $z_{r-1}$, and so on.
A symmetrized version was given ealier by Shiraishi \cite{Shiraishi}.
Our derivation follow their ideas with some differences in presentations.

\section{Vertex Realizable Operators and Their Correlations Functions}

In this Section we introduce a notion of vertex realizable operators on $\Lambda_{q,t}$
and define their correlation functions
by introducing $(u,v)$-bracket.
We also develop a method to compute the correlators of the modified Macdonald operators
$|tilde{E}^r$.

\subsection{Vertex realizable operators}

Now we introduce the following

\begin{definition}
An operator $\cA: \Lambda_{q, t} \to \Lambda_{q,t}$ with
$\{P_\mu(x; q,t)\}_{\mu \in \cP}$ as eigenvectors
with eigenvalues  $a_\mu(q,t)$ is said to be
{\em vertex realizable} of weight $r \geq 1$ if
and a Laurent series $K_\cA(z_r, \dots, z_1)$  with coefficients in $\bC(q,t)$ such that
\be
\begin{split}
\cA & = \biggl[
K_\cA(z_r, \dots, z_1)
 \cdot \exp\biggl(
\sum_{n=1}^\infty \frac{1-t^{-n}}{n} p_n \sum_{\alpha=1}^r z_\alpha^n \biggr) \\
&\cdot \exp \biggl(-\sum_{n=1}^\infty \frac{1-q^n}{n} \sum_{\alpha=1}^r z_\alpha^{-n}
\cdot n \frac{\pd}{\pd p_n} \biggr)
\biggr]\biggr|_{z_r^0z_{r-1}^0 \cdots z_1^0}.
\end{split}
\ee
We will refer to $K_\cA(z_r, \dots, z_1)$ is the {\em kernel} of $\cA$.
A linear combination of vertex realizable operators of different weights
will also be called vertex realizable.
\end{definition}

\begin{prop} \label{prop:Composition}
Suppose that $\cA$ and $\cB$ are two vertex realizable operators
of weights $r$ and $s$,
with kernels  $K_\cA(z_r, \dots, z_1)$ and $K_\cB(z_s, \dots, z_1)$,  respectively.
Then their composition $\cA\cB$ is a vertex realizable operator of weight $r+s$
with kernel
\be
\begin{split}
& \exp\biggl(
\sum_{n=1}^\infty \frac{(1-t^{-n})(1-q^n)}{n} \sum_{\alpha=s+1}^{r+s} z_\alpha^{-n}
\cdot \sum_{\beta=1}^{s} z_\alpha^{n} \biggr) \\
& \cdot K_\cA(z_{r+s}, \dots, z_{s+1}) K_\cB(z_{s}, \dots, z_{1}).
\end{split}
\ee
\end{prop}

\begin{proof}
This is a straightforward consequence of the identity:
\be \label{eqn:Weyl}
\begin{split}
& \exp \biggl( \sum_{n=1}^\infty \frac{a_n}{n} \cdot n \frac{\pd}{\pd p_n}\biggr)
\exp \biggl( \sum_{n=1}^\infty \frac{b_n}{n} p_n \biggr) \\
= & \exp \biggl( \sum_{n=1}^\infty \frac{a_nb_n}{n} \biggr) \cdot
\exp \biggl( \sum_{n=1}^\infty \frac{a_n}{n} \cdot n \frac{\pd}{\pd p_n}\biggr)
\exp \biggl( \sum_{n=1}^\infty \frac{b_n}{n} p_n \biggr).
\end{split}
\ee
\end{proof}

\subsection{The $(u,v)$-bracket and correlation functions}

For a vertex realizable operator $\cA$ as in the above Definition,
we define the $(u,v)$-bracket of $\cA$ by:
\be
\begin{split}
\corr{\cA}_{u,v}: & =  \sum_{\mu} (-u Q)^{|\mu|}\\
& \cdot \prod_{s \in \mu}
\frac{(q^{a'(s)}- vt^{l'(s)})}{(1- t^{l(s)} q^{a(s)+1})} \cdot a_\mu(q,t) \cdot
\prod_{s \in \mu} \frac{(t^{l'(s)}-u^{-1}q^{-a'(s)})}{(1 - q^{-a(s)}t^{-(l(s)+1)})}.
\end{split}
\ee
This is clearly motivated by \S \ref{sec:Specialization}:
If one changes $u$ to $ut^{-a}q^b$ and $v$ to $vt^{a}q^{-b}$,
then one gets the expressions as in \S \ref{sec:Specialization}.
We also define the normalized $(u,v)$-bracket of $\cA$ by:
\be
\corr{\cA}_{u,v}':=\frac{\corr{\cA}_{u,v}}{\corr{1}_{u,v}}.
\ee

For vertex realizable operators $\cA_1, \dots, \cA_n$,
define their $(u,v)$-correlation function by:
\be
\corr{\cA_1\cdots \cA_n}_{u,v}.
\ee
We also define
\be
\corr{\cA_1\cdots \cA_n}_{u,v}':=
\frac{\corr{\cA_1\cdots \cA_n}_{u,v}}{\corr{1}_{u,v}},
\ee
and the connected correlations functions:
\ben
 \corr{\cA_1\cA_2}^\circ_{u,v}&:= &\corr{\cA_1\cA_2}_{u,v}'
- \corr{\cA_1}_{u,v}' \cdot \corr{\cA_2}_{u,v}',\\
\corr{\cA_1\cA_2\cA_3}^\circ_{u,v} &:= & \corr{\cA_1\cA_2\cA_3}'_{u,v}
- \corr{\cA_1\cA_2}'_{u,v} \cdot \corr{\cA_3}'_{u,v}\\
&& - \corr{\cA_1\cA_3}'_{u,v} \cdot \corr{\cA_2}'_{u,v}
\corr{\cA_2\cA_3}'_{u,v} \cdot \corr{\cA_1}'_{u,v} \\
&& +2\corr{\cA_1}'_{u,v}\cdot \corr{\cA_2}'_{u,v} \cdot \corr{\cA_3}'_{u,v},
\een
etc. These notations borrowed from quantum field theory
have been used in mathematical literature,
see e.g. Okounkov \cite{Okounkov}.

\begin{prop} \label{prop:Correlator}
Let $\cA$ be a vertex realizable operator as in the above definition,
then we have:
\be \label{eqn:A'}
\begin{split}
\corr{\cA}_{u,v}'
= & \biggl[ K_\cA(z_r, \dots, z_1)
\exp \biggl(\sum_{n=1}^\infty \frac{(-1)^n}{n} (u^{n}-1)
\sum_{a=1}^r z_a^n Q^n \biggr) \\
& \cdot \exp \biggl(\sum_{n=1}^\infty \frac{(-1)^n}{n}
(1-v^n)\sum_{a=1}^r z_a^{-n} \biggr)
\biggr] \biggr|_{z_r^0z_{r-1}^0 \cdots z_1^0}.
\end{split}
\ee
\end{prop}

\begin{proof}
By slightly generalizing the computations in \S \ref{sec:Specialization},
we have:
\ben
\corr{\cA}_{u,vz}
& = & \lvac \exp \biggl(\sum_{n=1}^\infty \frac{(-1)^{n-1}}{n} \frac{1-u^{-n}}{1-t^{-n}}
\alpha_{n} \biggr) \\
&& K \omega \cA  \exp \biggl(\sum_{n=1}^\infty
\frac{(-1)^{n-1}}{n}\frac{1-v^n}{1-q^n}
 \alpha_{-n} \biggr) \vac \\
& = & \lvac \exp \biggl(\sum_{n=1}^\infty \frac{1}{n} \frac{1-u^{-n}}{1-t^{-n}}
(-uQ)^n\alpha_{n} \biggr) \\
&& \cdot  \biggl[
K_\cA(z_r, \dots, z_1)
 \cdot \exp\biggl(
\sum_{n=1}^\infty \frac{1-t^{-n}}{n}   \sum_{a=1}^r z_a^n \alpha_{-n} \biggr) \\
&& \cdot \exp \biggl(-\sum_{n=1}^\infty \frac{1-q^n}{n} \sum_{a=1}^r z_a^{-n}
\cdot \alpha_n \biggr) \biggr] \biggr|_{z_r^0z_{r-1}^0 \cdots z_1^0} \\
&& \cdot \exp \biggl(\sum_{n=1}^\infty
\frac{(-1)^{n-1}}{n}\frac{1-v^n}{1-q^n}
 \alpha_{-n} \biggr) \vac. \een
Then \eqref{eqn:A'} can be obtained by applying \eqref{eqn:Weyl}.
\end{proof}

\subsection{Explicit computations of correlation functions of $\tilde{E}^r$}

Using Proposition \ref{prop:Composition} and Proposition \ref{prop:Correlator},
it is possible to compute $\corr{\tilde{E}^{r_1} \cdots \tilde{E}^{r_n}}'_{u,v}$.
For example,
by \eqref{eqn:tildeE-vertex} and \eqref{eqn:A'} we get:
\ben
\corr{\tilde{E}^r}_{u,v}'
& = &  e_r(t^{-1}, t^{-2}, \dots) \cdot
\biggl[\prod_{1 \leq \alpha<\beta \leq r} \frac{1 - z_\alpha/z_\beta}{1 - z_\alpha/(z_\beta t)} \\
&& \cdot
\exp \biggl(\sum_{n=1}^\infty \frac{(-1)^n}{n} (u^{n}-1)
\sum_{a=1}^r z_a^n Q^n \biggr) \\
&& \cdot \exp \biggl(\sum_{n=1}^\infty \frac{(-1)^n}{n}
(1-v^n)\sum_{a=1}^r z_a^{-n} \biggr)
\biggr] \biggr|_{z_r^0z_{r-1}^0 \cdots z_1^0}.
\een
Note we have
\ben
&& \exp \biggl(\sum_{n=1}^\infty \frac{(-1)^n}{n} (u^{n}-1) z^n Q^n \biggr)
= \frac{1+zQ}{1+u zQ}
= 1 + \sum_{n=1}^n (-1)^n u^{n-1}(u - 1) Q^nz^n, \\
&& \exp \biggl(\sum_{n=1}^\infty \frac{(-1)^n}{n}(1-v^n) z^{-n} \biggr)
= \frac{1+vz^{-1}}{1+z^{-1}}
=  1 + \sum_{n=1}^\infty (-1)^n (1-v) z^{-n}.
\een
For the purpose of simplifying the notations in
the computations of $\corr{\tilde{E}^r}$,
we introduce the following notation:
Let $f$ be formal Laurent series in $z_1, \dots, z_r$,
define:
\be
\Corr{f}_r:=\biggl[f \cdot \prod_{1 \leq \alpha<\beta \leq r}
\frac{1 - z_\alpha/z_\beta}{1 - z_\alpha/(z_\beta t)} \cdot
\prod_{a=1}^r \frac{1+z_a Q}{1+u z_a Q}  \cdot \prod_{a=1}^r \frac{1+vz_a^{-1}}{1+z_a^{-1}}
\biggr] \biggr|_{z_r^0\cdots z_1^0}.
\ee
Expressions on the right-hand side of the above definition
are not understood as rational function,
but as series as follows:
\ben
&& \frac{1-z_\alpha/z_\beta}{1-z_\alpha/(z_\beta t)}
= 1+ (t^{-1} -1) \sum_{k=1}^\infty t^{-(k-1)} z_\alpha^k z_\beta^{-k},\\
&&   \frac{1+z_aQ}{1+uz_aQ}
=  1 + \sum_{n=1}^n (-1)^{n-1} u^{n-1}(1 - u) Q^nz_a^n, \\
&&  \frac{1+vz_a^{-1}}{1+z_a^{-1}}
=  1 + \sum_{n=1}^\infty (-1)^n (1-v) z_a^{-n}.
\een
It is also useful to define:
\be
\Cors{f}_r:= \biggl[f \cdot \frac{1+z_r Q}{1+u z_r Q}  \cdot \frac{1+vz_r^{-1}}{1+z_r^{-1}}
\biggr] \biggr|_{z_r^0}.
\ee
Note for a formal Laurent series $f$ in $z_1, \dots, z_r$,
\be
\Corr{f}_r = \COrs{\cdots \COrs{\COrs{f \cdot \prod_{1 \leq \alpha< r} \frac{1 - z_\alpha/z_r}{1 - z_\alpha/(z_r t)}}_r
\cdot \prod_{1 \leq \alpha< r-1} \frac{1 - z_\alpha/z_{r-1}}{1 - z_\alpha/(z_{r-1} t)} }_{r-1}  \cdots}_1.
\ee
In particular, for a formal Laurent series $f$ in $z_1$,
\be
\Corr{f}_1 = \Cors{f}_1,
\ee
and so we have
\be
\Cors{z_r^k}_r = \Cors{z_1^k}_1 = \Corr{z_1^k}_1.
\ee
So the computations of $\Corr{f}_r$ can be reduced to the computations of
$\Cors{z_1^k}_1 = \Corr{z_1^k}_1$.
This is our method for the computations of correlators of
$\tilde{E}^r$.

\subsubsection{Computations of $\Cors{z_1^k}_1$}

\begin{prop}
The following identities hold:
\bea
&& \Cors{1}_1 = 1 - Q \frac{(1-u)(1-v)}{1-uQ}, \\
&&  \Cors{z_1^k}_1  = (-1)^k \cdot \frac{(1-v)(1-Q)}{1-uQ}, \\
&& \Cors{ z_1^{-k} }_1  = (-uQ)^{k-1} \cdot Q\frac{(1-u)(1-uvQ)}{1-uQ},
\eea
where $k> 0$.
\end{prop}

\begin{proof}
These can also be obtained by direct computations.
\ben
\Cors{1}_1  & = & \biggl[
\frac{1+z_1 Q}{1+u z_1 Q}  \cdot \frac{1+vz_1^{-1}}{1+z_1^{-1}}
\biggr] \biggr|_{z_1^0} \\
& = & \biggl[ (1 + \sum_{n=1}^n (-1)^{n-1} u^{n-1}(1 - u) Q^nz_a^n)
\cdot ( 1 + \sum_{n=1}^\infty (-1)^n (1-v) z_a^{-n} )
\biggr] \biggr|_{z_2^0} \\
& = & 1 - \sum_{n=1}^\infty (1-u)(1-v) u^{n-1}Q^n
= 1 - Q \frac{(1-u)(1-v)}{1-uQ}.
\een

\ben
\Cors{z_1^k}_1
& = & \biggl[z_2^{k} \cdot  (1 + \sum_{n=1}^n (-1)^{n-1} u^{n-1}(1 - u) Q^nz_a^n)
\cdot ( 1 + \sum_{n=1}^\infty (-1)^n (1-v) z_a^{-n} )
\biggr] \biggr|_{z_2^0} \\
& = & (-1)^k (1-v)
+ \sum_{n=1}^\infty (-1)^{k-1} u^{n-1}(1-u)(1-v) Q^n  \\
& = & (-1)^k (1-v) \cdot  \biggl( 1 - \frac{1-u}{1-uQ} Q \biggr)
= (-1)^k \cdot \frac{(1-v)(1-Q)}{1-uQ}.
\een
\ben
\Cors{ z_2^{-k} }_2
%%%% & = & \biggl[z_2^{-k} \cdot \frac{1+z_2Q}{1+uz_2Q} \cdot
%%% \frac{1+vz_2^{-1}}{1+z_2^{-1}} \biggr] \biggr|_{z_2^0} \\
& = & \biggl[z_2^{-k} \cdot  (1 + \sum_{n=1}^n (-1)^{n-1} u^{n-1}(1 - u) Q^nz_a^n)
\cdot ( 1 + \sum_{n=1}^\infty (-1)^n (1-v) z_a^{-n} )
\biggr] \biggr|_{z_2^0} \\
& = & (-1)^{k-1} u^{k-1}(1-u)Q^k
+ \sum_{n=1}^\infty (-1)^{k-1} u^{k+n-1}(1-u)(1-v)Q^{k+n}  \\
& = & (-1)^{k-1} u^{k-1}(1-u)Q^k \biggl( 1
+ (1-v) \sum_{n=1}^\infty u^{n} Q^{n} \biggr) \\
& = & (-1)^{k-1} u^{k-1}(1-u)Q^k \biggl( 1
+ (1-v) \cdot \frac{uQ}{1-uQ}  \biggr) \\
& = & Q (-uQ)^{k-1} \cdot \frac{(1-u)(1-uvQ)}{1-uQ}.
\een
\end{proof}

As a corollary,
\be
\corr{\tilde{E}^1}_{u,v} = \frac{t^{-1}}{1-t^{-1}} \biggl(1 - Q \frac{(1-u)(1-v)}{1-uQ} \biggr).
\ee

\subsubsection{One-point function of $\tilde{E}^2$}

Our result for $\corr{\tilde{E}^2}'_{u,v}$ is
\be \label{eqn:E2}
\begin{split}
\corr{\tilde{E}^2}'_{u,v}
= & \frac{t^{-3}}{(1-t^{-1})(1-t^{-2})} \biggl[
\biggl(1- Q \frac{(1-u)(1-v)}{1-u Q}\biggr)^2 \\
+ & (1-t^{-1}) Q \frac{(1-Q)(1-u)(1-v)(1-uvQ)}{(1-t^{-1}uQ)(1-uQ)^2}
\biggr].
\end{split}
\ee
First we have:
\ben
\corr{\tilde{E}^2}_{u,v}'
& = &  e_2(t^{-1}, t^{-2}, \dots) \cdot
\biggl[ \frac{1 - z_1/z_2}{1 - z_1/(z_2 t)} \cdot
\prod_{a=1}^2 \frac{1+z_a Q}{1+u z_a Q}  \cdot \prod_{a=1}^2 \frac{1+vz_a^{-1}}{1+z_a^{-1}}
\biggr] \biggr|_{z_2^0z_1^0},
\een
and so
\ben
&& \corr{\tilde{E}^2}_{u,v}'
= \frac{t^{-3}}{(1-t^{-1})(1-t^{-2})} \cdot \Corr{1}_2.
\een
We use the expansions
\ben
&& \frac{1-z_1/z_2}{1-z_1/(z_2t)}
= 1+ (t^{-1} -1) \sum_{k=1}^\infty t^{-(k-1)} z_1^k z_2^{-k},
\een
to get:
\ben
\Corr{1}_2 & = & \COrs{ \COrs{\frac{1-z_1/z_2}{1-z_1/(z_2t)}}_2}_1
= \COrs{\Cors{1}_2+ (t^{-1} -1) \COrs{\sum_{k=1}^\infty t^{-(k-1)} z_1^k z_2^{-k}}_2}_1 \\
& = & \Corr{1}_1^2+ (t^{-1} -1) \cdot \Corr{ \sum_{k=1}^\infty t^{-(k-1)} z_1^k
\cdot (-uQ)^{k-1} \cdot Q \frac{(1-u)(1-uvQ)}{1-uQ}}_1 \\
& = & \Corr{1}_1^2+
+ (t^{-1}-1)Q \frac{(1-u)(1-uvQ)}{1-uQ} \cdot
\Corr{\sum_{k=1}^\infty (-t^{-1}uQ)^{k-1} z_1^k}_1.
\een
We rewrite it in the following form:
\be
\Corr{1}_2 =  \Corr{1}_1^2
+ (t^{-1}-1)Q \frac{(1-u)(1-uvQ)}{1-uQ} \cdot
\COrr{\frac{z_1}{1+t^{-1}uQ z_1}}_1.
\ee
Now note
\be
\COrr{\frac{z_1}{1+t^{-1}u Q z_1} }_1
= -\frac{1}{1-t^{-1}uQ} \cdot \frac{(1-v)(1-Q)}{1-uQ}.
\ee
Indeed,
\ben
\COrr{\frac{z_1}{1+t^{-1}uQ z_1}}_1
&  = &   \sum_{k=1}^\infty (-t^{-1}u)^{k-1}Q^{k-1} \cdot \Cors{z_1^k}_1 \\
& = & -\sum_{k=1}^\infty t^{-(k-1)} u^{k-1} Q^{k-1}  \cdot  \frac{(1-v)(1-Q)}{(1-uQ)} \\
& = & - \frac{1}{1-t^{-1}uQ} \cdot \frac{(1-v)(1-Q)}{1-uQ}.
\een
This finishes the computation for $\corr{\tilde{E}^2}'_{u,v}$.

\subsubsection{One-point function of $\tilde{E}^3$}

Now we come to compute $\corr{\tilde{E}^3}'_{u,v}$ in the same fashion.
Recall
\ben
\corr{\tilde{E}^3}_{u,v}'
& = &  e_3(t^{-1}, t^{-2}, \dots) \\
&& \cdot
\biggl[\prod_{1 \leq \alpha<\beta \leq 3} \frac{1 - z_\alpha/z_\beta}{1 - z_\alpha/(z_\beta t)}
\cdot
\prod_{a=1}^3 \frac{1+z_a Q}{1+u z_a Q}  \cdot \prod_{a=1}^3 \frac{1+vz_a^{-1}}{1+z_a^{-1}}
\biggr] \biggr|_{z_3^0z_2^0z_1^0}.
\een
We first get:
\ben
&& \COrs{\frac{1 - z_1/z_3}{1 - z_1/(z_3 t)} \cdot
\frac{1 - z_2/z_3}{1 - z_2/(z_3 t)}  }_3 \\
& = & \COrs{ (1+ (t^{-1} -1) \sum_{k=1}^\infty t^{-(k-1)} z_1^k z_3^{-k})
\cdot (1+ (t^{-1} -1) \sum_{l=1}^\infty t^{-(l-1)} z_2^l z_3^{-l} ) }_3 \\
& = & \Cors{1}_3  + (t^{-1} -1) \sum_{k=1}^\infty t^{-(k-1)} z_1^k \COrs{z_3^{-k}}_3
+ (t^{-1} -1) \sum_{l=1}^\infty t^{-(l-1)} z_2^l \COrs{z_3^{-l}}_3 \\
& + & (t^{-1} -1)^2 \sum_{k,l =1}^\infty t^{-(k-1)} z_1^k
\cdot t^{-(l-1)} z_2^l \COrs{z_3^{-k-l}}_3 \\
& = & \Corr{1}_1
+ (t^{-1} -1) \sum_{k=1}^\infty t^{-(k-1)} z_1^k (-uQ)^{k-1} \cdot Q \frac{(1-u)(1-uvQ)}{1-uQ} \\
& + & (t^{-1} -1) \sum_{l=1}^\infty t^{-(l-1)} z_2^l(-uQ)^{l-1} \cdot Q \frac{(1-u)(1-uvQ)}{1-uQ} \\
& + & (t^{-1} -1)^2 \sum_{k,l =1}^\infty t^{-(k-1)} z_1^k
\cdot t^{-(l-1)} z_2^l (-uQ)^{k+l-1} \cdot Q \frac{(1-u)(1-uvQ)}{1-uQ}.
\een
We rewrite it as follows:
\ben
\Corr{1}_3 & = & \Corr{1}_1 \cdot \Corr{1}_2 \\
& + & Q (t^{-1}-1) \frac{(1-u)(1-uvQ)}{1-uQ} \cdot \COrr{\frac{z_1}{1+t^{-1}uQ z_1} }_2 \\
& + &  Q (t^{-1}-1) \frac{(1-u)(1-uvQ)}{1-uQ} \cdot \COrr{\frac{z_2}{1+t^{-1}uQ z_2} }_2 \\
& - & uQ^2 (t^{-1}-1)^2 \frac{(1-u)(1-uvQ)}{1-uQ} \cdot
\COrr{\frac{z_1}{1+t^{-1}uQ z_1} \cdot \frac{z_2}{1+t^{-1}uQ z_2}}_2.
\een

\noindent Computations for $\COrr{\frac{z_1}{1+t^{-1}uQ z_1} }_2$.
From
\ben
&& \Cors{ z_1^l (1+ (t^{-1} -1) \sum_{k=1}^\infty t^{-(k-1)} z_1^k z_2^{-k}) }_2 \\
& = & z_1^l ( \Cors{1}_2+ (t^{-1} -1) \sum_{k=1}^\infty t^{-(k-1)} z_1^k \Cors{ z_2^{-k}}_2 ) \\
& = & z_1^l(\Corr{1}_1 + (t^{-1} -1) \sum_{k=1}^\infty t^{-(k-1)} z_1^k \Corr{ z_1^{-k}}_1 ) \\
& = & z_1^l\biggl(\Corr{1}_1 + (t^{-1} -1) \sum_{k=1}^\infty t^{-(k-1)} z_1^k
\cdot (-uQ)^{k-1} \cdot Q\frac{(1-u)(1-uvQ)}{1-uQ} \biggr) \\
& = & \Corr{1}_1 z_1^l
+ (t^{-1} -1) Q\frac{(1-u)(1-uvQ)}{1-uQ}\cdot z_1^l
\cdot \frac{z_1}{1+ut^{-1}Qz_1},
\een
we get:
\ben
\COrr{\frac{z_1}{1+t^{-1}uQ z_1} }_2
& = & \Corr{1}_1 \cdot
\COrr{\frac{z_1}{1+t^{-1}uQ z_1} }_1  \\
&+ & (t^{-1}-1) \cdot Q \frac{(1-u)(1-uvQ)}{1-uQ} \cdot
\COrr{\frac{z_1}{1+t^{-1}uQ z_1} \cdot \frac{z_1}{1+t^{-1}uQ z_1} }_1.
\een
We have
\ben
&& \COrr{\frac{z_1}{1+t^{-1}uQ z_1} \cdot \frac{z_1}{1+t^{-1}uQ z_1} }_1 \\
& = & \Corr{\sum_{l\geq 1} (-t^{-1}uQ)^{l-1} z_1^l \cdot
\sum_{k=1}^\infty (-t^{-1}uQ)^{k-1}z_1^k}_1 \\
& = & \sum_{k,l\geq 1} (-t^{-1}uQ)^{l-1}  \cdot
 (-t^{-1}uQ)^{k-1} \cdot \Corr{z_1^{k+l}}_1 \\
& = & \sum_{k,l\geq 1} (-t^{-1}uQ)^{l-1} \cdot
 (-t^{-1}uQ)^{k-1} \cdot (-1)^{k+l} \cdot \frac{(1-v)(1-Q)}{(1-uQ)} \\
& = & \frac{1}{(1-ut^{-1}Q)^2} \cdot \frac{(1-v)(1-Q)}{1-uQ}.
\een
In the same fashion one can prove the following identity:
\be
\COrr{\biggl(\frac{z_1}{1+t^{-1}u Q z_1} \biggr)^m}_1
= \frac{(-1)^m}{(1-t^{-1}uQ)^m} \cdot \frac{(1-v)(1-Q)}{(1-uQ)}.
\ee

\noindent Computations for $\COrr{\frac{z_2}{1+t^{-1}uQ z_2} }_2$.
We need to first compute
\ben
&& \sum_{l \geq 1} (-t^{-1}uQ)^{l-1} \COrs{z_2^l(1+ (t^{-1} -1) \sum_{k=1}^\infty t^{-(k-1)} z_1^k z_2^{-k})}_2 \\
& = & \sum_{l \geq 1} (-t^{-1}uQ)^{l-1} \Cors{z_2^l}_2
+ (t^{-1} -1) \sum_{k, l \geq 1} (-t^{-1}uQ)^{l-1}  t^{-(k-1)} z_1^k \Cors{z_2^{l-k}}_2.
\een
We have two parts to consider.
Part one:
\ben
&& \sum_{l \geq 1} (-t^{-1}uQ)^{l-1} \Cors{z_2^l}_2
= \sum_{l \geq 1} (-t^{-1}uQ)^{l-1} \cdot \Corr{z_1^l}_1 \\
& = & \COrr{\frac{z_1}{1+ut^{-1}Qz_1}}_1 = -\frac{1}{1-t^{-1}uQ} \cdot \frac{(1-v)(1-Q)}{(1-uQ)}.
\een
Part two without a factor of $(t^{-1} -1)$:
\ben
&& \sum_{k, l \geq 1} (-t^{-1}uQ)^{l-1}  t^{-(k-1)} z_1^k \Cors{z_2^{l-k}}_2
= \sum_{k, l \geq 1} (-t^{-1}uQ)^{l-1}  t^{-(k-1)} z_1^k \Corr{z_1^{l-k}}_1 \\
& = & \sum_{k \geq 1} (-t^{-1}uQ)^{k-1}  t^{-(k-1)} z_1^k \Corr{1}_1
+ \sum_{l> k \geq 1} (-t^{-1}uQ)^{l-1}  t^{-(k-1)} z_1^k \Corr{z_1^{l-k}}_1 \\
& + & \sum_{k > l \geq 1} (-t^{-1}uQ)^{l-1}  t^{-(k-1)} z_1^k \Corr{z_1^{l-k}}_1 \\
& = & \Corr{1}_1 \cdot  \frac{z_1}{1+ut^{-1}Qz_1}
+  \sum_{l> k \geq 1} (-t^{-1}uQ)^{l-1}  t^{-(k-1)} z_1^k
\cdot (-1)^{l-k} \frac{(1-v)(1-Q)}{(1-uQ)} \\
& + & \sum_{k > l \geq 1} (-t^{-1}uQ)^{l-1}  t^{-(k-1)} z_1^k
\cdot (-uQ)^{k-l-1} \cdot Q\frac{(1-u)(1-uvQ)}{1-uQ}.
\een
After simplification it becomes:
\ben
&& \Corr{1}_1 \cdot  \frac{z_1}{1+ut^{-1}Qz_1}
+  \sum_{l> k \geq 1} (-1)^{k-1}t^{-k-l+2}u^{l-1}Q^{l-1}  z_1^k \cdot
\frac{(1-v)(1-Q)}{(1-uQ)} \\
& + & \sum_{k > l \geq 1} (-1)^k t^{-k-l+2}u^{k-2}Q^{k-2}  z_1^k
 \cdot Q\frac{(1-u)(1-uvQ)}{1-uQ}.
\een
Take $\Cors{\cdot}_1$:
\ben
&& \Corr{1}_1 \cdot  \frac{1}{1-t^{-1}uQ} \cdot \frac{(1-v)(1-Q)}{(1-uQ)} \\
& + & \sum_{l> k \geq 1} (-1)^{k-1}t^{-k-l+2}u^{l-1}Q^{l-1}  \cdot
(-1)^k \cdot \frac{(1-v)(1-Q)}{(1-uQ)} \cdot
\frac{(1-v)(1-Q)}{(1-uQ)} \\
& + & \sum_{k > l \geq 1} (-1)^k t^{-k-l+2}u^{k-2}Q^{k-2} \cdot
(-1)^k \cdot \frac{(1-v)(1-Q)}{(1-uQ)}
 \cdot Q\frac{(1-u)(1-uvQ)}{1-uQ} \\
& = & \Corr{1}_1 \cdot  \frac{1}{1-t^{-1}uQ} \cdot \frac{(1-v)(1-Q)}{(1-uQ)} \\
& - & \frac{1}{1-ut^{-2}Q} \cdot \frac{u t^{-1}Q}{1- u t^{-1} Q}
\cdot \frac{(1-v)(1-Q)}{(1-uQ)} \cdot \frac{(1-v)(1-Q)}{(1-uQ)} \\
& + & \frac{t^{-1}}{1-ut^{-2}Q} \cdot \frac{1}{1- u t^{-1} Q}
\cdot \frac{(1-v)(1-Q)}{(1-uQ)} \cdot Q\frac{(1-u)(1-uvQ)}{1-uQ}.
\een

In the above we have used the following summations:
\ben
 \sum_{l> k \geq 1} t^{-k-l+2}u^{l-1}Q^{l-1}
& = &\sum_{j, k \geq 1} t^{-j-2k+2}u^{j+k-1}Q^{j+k-1} \\
&= &\frac{1}{1-ut^{-2}Q} \cdot \frac{u t^{-1}Q}{1- u t^{-1} Q}, \\
\sum_{k > l \geq 1} t^{-k-l+2}u^{k-2}Q^{k-2}
& = & \sum_{m, l \geq 1} t^{-m-2l+2}u^{l+m-2}Q^{l+m-2} \\
& = & \frac{t^{-1}}{1-ut^{-2}Q} \cdot \frac{1}{1- u t^{-1} Q}.
\een
The computation of $\COrr{\frac{z_1}{1+t^{-1}uQ z_1} \cdot \frac{z_2}{1+t^{-1}uQ z_2}}_2$
is similar and will be omitted.
From this example it is clear that the complexity the computation increase very rapidly.

\subsubsection{Computations of $\corr{\tilde{E}^1\tilde{E}^1}$}
Now we compute the two-point function of $\tilde{E}^1$:
\ben
\corr{\tilde{E}^1\tilde{E}^1}_{u,v}
& = & \biggl(\frac{t^{-1}}{1-t^{-1}}  \biggr)^2
\lvac \exp \biggl(\sum_{n=1}^\infty \frac{(-1)^{n-1}}{n} \frac{1-u^{-n}}{1-t^{-n}}
\alpha_{n} \biggr) K \omega \\
&& \cdot \biggl[ \exp\biggl(
\sum_{n=1}^\infty \frac{1-t^{-n}}{n} \alpha_{-n} z_1^n \biggr)
\cdot \exp \biggl(-\sum_{n=1}^\infty \frac{1-q^n}{n} z_1^{-n}
\cdot \alpha_n \biggr) \biggr]\biggr|_{z_1^0} \\
&& \cdot \biggl[ \exp\biggl(
\sum_{n=1}^\infty \frac{1-t^{-n}}{n} \alpha_{-n} z_2^n \biggr)
\cdot \exp \biggl(-\sum_{n=1}^\infty \frac{1-q^n}{n} z_2^{-n}
\cdot \alpha_n \biggr)\biggr]\biggr|_{z_2^0}  \\
&& \cdot \exp \biggl(\sum_{n=1}^\infty
\frac{(-1)^{n-1}}{n}\frac{1-v^n}{1-q^n}
 \alpha_{-n} \biggr) \vac \\
& = &  \biggl(\frac{t^{-1}}{1-t^{-1}}  \biggr)^2
\lvac \exp \biggl(\sum_{n=1}^\infty \frac{1}{n} \frac{1-u^{-n}}{1-t^{-n}}
(-uQ)^n\alpha_{n} \biggr) \\
&& \cdot \exp\biggl(
\sum_{n=1}^\infty \frac{1-t^{-n}}{n} \alpha_{-n} z_1^n \biggr)
\cdot \exp \biggl(-\sum_{n=1}^\infty \frac{1-q^n}{n} z_1^{-n}
\cdot \alpha_n \biggr)  \\
&& \cdot  \exp\biggl(
\sum_{n=1}^\infty \frac{1-t^{-n}}{n} \alpha_{-n} z_2^n \biggr)
\cdot \exp \biggl(-\sum_{n=1}^\infty \frac{1-q^n}{n} z_2^{-n}
\cdot \alpha_n \biggr)  \\
&& \cdot \exp \biggl(\sum_{n=1}^\infty
\frac{(-1)^{n-1}}{n}\frac{1-v^n}{1-q^n}
 \alpha_{-n} \biggr) \vac \biggr|_{z_2^0z_1^0}  . \een
It follows that
\ben
&& \corr{\tilde{E}^1\tilde{E}^1}_{u,v}
= \biggl(\frac{t^{-1}}{1-t^{-1}}  \biggr)^2 \\
&& \;\;\;\; \cdot \exp \biggl(\sum_{n=1}^\infty \frac{1}{n} (1-u^{-n})
(-uQ)^nz_2^n \biggr)
\cdot \exp \biggl(\sum_{n=1}^\infty \frac{1}{n} (1-u^{-n})
(-uQ)^nz_1^n \biggr) \\
&& \;\;\;\; \cdot  \exp \biggl(-\sum_{n=1}^\infty \frac{(1-q^n)(1-t^{-n})}{n} z_1^{-n}z_2^{n} \biggr)  \\
&& \;\;\;\; \cdot  \exp \biggl(\sum_{n=1}^\infty \frac{(-1)^n}{n} (1-v^n) z_1^{-n}  \biggr)
\cdot \exp \biggl(\sum_{n=1}^\infty
\frac{(-1)^n}{n}(1-v^n) z_2^{-n} \biggr)  \biggr|_{z_2^0z_1^0} .
\een
We rewrite as
\ben
&& \corr{\tilde{E}^1\tilde{E}^1}_{u,v}
= \biggl(\frac{t^{-1}}{1-t^{-1}}  \biggr)^2 \\
&& \;\;\;\; \cdot \prod_{a=1}^2 \frac{1 + z_aQ}{1+uz_aQ} \cdot
\prod_{a=1}^2 \frac{1+vz_a^{-1}}{1+z_a^{-1}} \cdot
\frac{(1-z_2z_1^{-1})(1-qt^{-1}z_2z_1^{-1})}{(1-qz_2z_1^{-1})(1-t^{-1}z_2z_1^{-1})}
 \biggr|_{z_2^0z_1^0}.
\een
We use the following expansions:
\ben
&& \exp \biggl(\sum_{n=1}^\infty \frac{(-1)^n}{n} (u^{n}-1) z_a^n Q^n \biggr)
=  \frac{1 + z_aQ}{1+uz_aQ}
= 1 - \sum_{n=1}^n (-1)^n u^{n-1}(1-u) Q^nz_a^n, \\
&& \exp \biggl(\sum_{n=1}^\infty \frac{(-1)^n}{n}(1-v^n) z_a^{-n} \biggr)
= \frac{1+vz_a^{-1}}{1+z_a^{-1}}
= 1 + \sum_{n=1}^\infty (-1)^n (1-v) z_a^{-n}, \\
&& \exp \biggl(-\sum_{n=1}^\infty \frac{(1-q^n)(1-t^{-n})}{n} z_2^nz_1^{-n} \biggr)
=  \frac{(1-z_2z_1^{-1})(1-qt^{-1}z_2z_1^{-1})}{(1-qz_2z_1^{-1})(1-t^{-1}z_2z_1^{-1})}\\
&& \;\;\;\;\;\; = 1 - \sum_{n=1}^\infty (1-q)(1-t^{-1}) \frac{q^n - t^{-n}}{q-t^{-1}} z_2^nz_1^{-n}.
\een
The last equality is an application of the following identity:
\be
\frac{(1-x)(1-t_1t_2x)}{(1-t_1x)(1-t_2x)} = 1 - \sum_{n=1}^\infty
(1-t_1)(1-t_2) \frac{t_1^n-t_2^n}{t_1-t_2} x^n.
\ee
We get:
\ben
&& \COrs{\frac{(1-z_2z_1^{-1})(1-qt^{-1}z_2z_1^{-1})}{(1-qz_2z_1^{-1})(1-t^{-1}z_2z_1^{-1})} }_2 \\
& = & \COrs{1 - \sum_{n=1}^\infty (1-q)(1-t^{-1})
\frac{q^n - t^{-n}}{q-t^{-1}} z_2^nz_1^{-n}}_2 \\
& = & \Cors{1}_2 - \sum_{n=1}^\infty (1-q)(1-t^{-1})
\frac{q^n - t^{-n}}{q-t^{-1}} z_1^{-n} \Cors{z_2^n }_2 \\
& = &  \Corr{1}_1 - \sum_{n=1}^\infty (1-q)(1-t^{-1})
\frac{q^n - t^{-n}}{q-t^{-1}} z_1^{-n} (-1)^n \cdot \frac{(1-v)(1-Q)}{(1-uQ)}.
\een
Take $\Corr{\cdot}_1$:
\ben
&&  \Corr{1}_1 \cdot \Cors{1}_1 - \sum_{n=1}^\infty (1-q)(1-t^{-1})
\frac{q^n - t^{-n}}{q-t^{-1}} \Cors{z_1^{-n}}_1 \cdot (-1)^n \cdot \frac{(1-v)(1-Q)}{(1-uQ)} \\
& = & \Corr{1}_1^2 - \sum_{n=1}^\infty (1-q)(1-t^{-1})
\frac{q^n - t^{-n}}{q-t^{-1}} \cdot (-uQ)^{n-1} \cdot Q\frac{(1-u)(1-uvQ)}{1-uQ}\\
&&  \cdot (-1)^n \cdot \frac{(1-v)(1-Q)}{(1-uQ)} \\
& = & \biggl( 1 - Q\frac{(1-u)(1-v)}{1-uQ} \biggr)^2 \\
& + & Q\frac{(1-Q)(1-uvQ)}{(1-uQ)^2} \cdot (1-u)(1-v) \cdot
\frac{(1-q)(1-t^{-1})}{(1-uqQ)(1-ut^{-1}Q)}.
\een
So we get
\be \label{eqn:E1E1}
\begin{split}
\corr{\tilde{E}^1\tilde{E}^1}_{u,v}
= & \biggl(\frac{t^{-1}}{1-t^{-1}}  \biggr)^2 \cdot
\biggl[\biggl( 1- Q\frac{(1-u)(1-v)}{1-uQ} \biggr)^2
 \\
+ & Q\frac{(1-Q)(1-uvQ)}{(1-uQ)^2} \cdot (1-u)(1-v) \cdot
\frac{(1-q)(1-t^{-1})}{(1-uqQ)(1-ut^{-1}Q)} \biggr].
\end{split}
\ee

\subsection{Generalizations of \eqref{eqn:Arm-Leg}}

In order to get more vertex realizable operators we will
first generalize \eqref{eqn:Arm-Leg}.
By changing $q$ to $q^m$ and $t$ to $t^m$ in \eqref{eqn:Arm-Leg},
we have
\be \label{eqn:Arm-Leg-m}
\begin{split}
\sum_{i=1}^{\infty} q^{m\mu_i}t^{-mi} - \sum_{i=1}^{\infty} t^{-mi}
= & \frac{q^m-1}{t^m} \sum_{(i,j) \in \mu} t^{-m(i-1)}q^{m(j-1)} \\
= & \frac{q^m-1}{t^m} \sum_{s \in \mu} t^{-ml'(s)}q^{ma'(s)}.
\end{split}
\ee
Therefore,
we have
\be
\begin{split}
\sum_{s \in \mu} t^{-ml'(s)}q^{ma'(s)}
= & \frac{t^m}{q^m-1} (\sum_{i=1}^{\infty} q^{m\mu_i}t^{-mi}
- \sum_{i=1}^{\infty} t^{-mi}) \\
= & \frac{t^m}{q^m-1} (p_m(q^{\mu_1}t^{-1}, q^{\mu_2}t^{-2}, \dots)
- p_m(t^{-1}, t^{-2}, \dots)).
\end{split}
\ee
Recall the following relationship between Newton symmetric functions
and elementary symmetric functions:
\be \label{eqn:Newton-Elementary}
\sum_{n=1}^\infty \frac{(-1)^{n-1}}{n} p_n z^{n}
= \log \sum_{n=0}^\infty e_n z^n.
\ee
Write the right-hand side as
\be \label{eqn:alpha-lambda}
\log \sum_{n=0}^\infty e_n z^n = \sum_{\lambda \in \cP} \alpha_\lambda \cdot e_\lambda z^{|\lambda|}
\ee
for some coefficients $\alpha_\lambda$.
The first few terms are:
\be
e_1z+(e_2-\frac{1}{2}e_1^2)z^2+(e_3-e_1e_2+\frac{1}{3}e_1^3)z^3
+(e_4-e_1e_3-\frac{1}{2}e_2^2+e_2e_1^2-\frac{1}{4}e_1^4)z^4
+ \cdots.
\ee
In other words,
\begin{align}
\alpha_{(1)} & = 1, & \alpha_{(2)} & = 1, & \alpha_{(1,1)} & = - \half, \\
\alpha_{(3)} & = 1, & \alpha_{(2,1)} & = -1, & \alpha_{(1,1,1)} & = \frac{1}{3}.
\end{align}
One can also rewrite \eqref{eqn:Newton-Elementary} as
\be
\sum_{n=0}^\infty e_n z^n = \exp \sum_{n=1}^\infty \frac{(-1)^{n-1}}{n} p_n z^{n}.
\ee
This expresses the elementary symmetric functions in terms of Newton functions.
From this one sees that that $e_k(\{t^{-l'(s)}q^{a'(s)}\}_{s\in \mu})$
and $p_m(\{t^{-l'(s)}q^{a'(s)}\}_{s\in \mu})$
can be expressed as a polynomial in $e_r(q^{\mu_1}t^{-1}, q^{\mu_2}t^{-2}, \dots)$:
\ben
&& \sum_{k\geq 0} e_k(\{t^{-l'(s)}q^{a'(s)}\}_{s\in \mu}) z^k
= \exp \sum_{m=1}^\infty \frac{(-1)^{m-1}}{m} p_m(\{t^{-l'(s)}q^{a'(s)}\}_{s\in \mu}) z^m \\
& = & \exp \sum_{m=1}^\infty \frac{(-1)^{m-1}}{m}
\biggl(\frac{t^m}{q^m-1} (p_m(q^{\mu_1}t^{-1}, q^{\mu_2}t^{-2}, \dots)
- p_m(t^{-1}, t^{-2}, \dots)) \biggr) z^m \\
& = & \exp \sum_{n=0}^\infty \sum_{m\geq 1} \frac{(-1)^{m}}{m}
 p_m(q^{\mu_1}t^{-1}, q^{\mu_2}t^{-2}, \dots) (tq^nz)^m \\
 && \cdot \exp \sum_{n=0}^\infty \sum_{m\geq 1} \frac{(-1)^{m-1}}{m}
 p_m(t^{-1}, t^{-2}, \dots) (tq^nz)^m \\
 & = & \prod_{n\geq 0} \biggl(
 \sum_{r\geq 0} e_r(q^{\mu_1}t^{-1}, q^{\mu_2}t^{-2}, \dots) (tq^nz)^r
 \biggr)^{-1}
\cdot \prod_{n\geq 0}   \sum_{r\geq 0} e_r(t^{-1}, t^{-2}, \dots) (tq^nz)^r   \\
& = & \prod_{n\geq 0} \biggl(
\sum_{r\geq 0} e_r(q^{\mu_1}t^{-1}, q^{\mu_2}t^{-2}, \dots) (tq^nz)^r  \biggr)^{-1}
\cdot \prod_{n\geq 0} \prod_{j=1}^\infty (1+t^{1-j}q^nz).
\een
To compute the series expansions,
we need to generalize Euler's formulas \eqref{eqn:Euler1} and \eqref{eqn:Euler2} below:
\be \label{eqn:Euler2}
\prod_{n=0}^\infty \frac{1}{1- q^nz} = 1 + \sum_{n=1}^\infty \frac{z^n}{(1-q) \cdots (1-q^n)}.
\ee

\begin{prop}
Write
\be \label{eqn:Product}
\prod_{n \geq 0} \sum_{r\geq 0} a_r q^{nr} z^r
= \sum_{m \geq 0} b_m z^m,
\ee
where $a_0=1$, $\{a_r\}_{r\geq 1}$ are formal variables,
then each $b_m$ is a weighted homogeneous polynomial of $\{a_r\}_{r\geq 1}$
if we assign $\deg a_r = r$.
Furthermore,
the sequence $\{b_m\}_{m \geq 0}$ satisfies the following recursion relations:
\bea
&& b_0 = 1, \\
&& b_m = \frac{1}{1-q^m} \sum_{r=1}^{m} q^{m-r}a_r b_{m-r}.
\eea
\end{prop}

\begin{proof}
To get the first statement,
change $z$ to $\lambda z$ on both sides of \eqref{eqn:Product} to get:
\ben
&& \prod_{n \geq 0} \sum_{r\geq 0} (\lambda^ra_r) q^{nr} z^r
= \sum_{m \geq 0} (\lambda^mb_m) z^m.
\een
This means when $\{a_r\}_{r\geq 1}$ is changed to $\{\lambda^ra_r\}_{r\geq 1}$,
$\{b_m\}_{m\geq 1}$ is changed to $\{\lambda^mb_m\}_{m \geq 1}$.
To get the second statement,
note
\ben
&&\sum_{m \geq 0} b_m z^m  = \prod_{n \geq 0} \sum_{r\geq 0} a_r q^{nr} z^r
= \sum_{r\geq 0} a_r q^{nr} z^r \cdot \prod_{n \geq 0} \sum_{r\geq 0} a_r q^{nr} (qz)^r \\
& = & \sum_{r\geq 0} a_r z^r \cdot \sum_{m \geq 0} q^mb_m z^m,
\een
and so
\be
b_m = \sum_{r=0}^m a_rq^{m-r}b_{m-r}.
\ee
\end{proof}

The following are the first few terms of $b_m$:
\ben
b_1 & = &\frac{1}{1-q} a_1, \\
b_2 & = & \frac{1}{1-q^2}a_2 + \frac{q}{(1-q)(1-q^2)} a_1^2, \\
b_3 & = & \frac{1}{1-q^3}a_3 + \frac{q+2q^2}{(1-q^2)(1-q^3)}a_2a_1
+ \frac{q^3}{(1-q)(1-q^2)(1-q^3)} a_1^3, \\
b_4 & = & \frac{1}{1-q^4}a_4 + \frac{q+q^2+2q^3}{(1-q^3)(1-q^4)}a_3a_1
+ \frac{q^2}{(1-q^2)(1-q^4)} a_2^2\\
& + & \frac{q^3+2q^4+3q^5}{(1-q^2)(1-q^3)(1-q^4)} a_2a_1^2
+ \frac{q^6}{(1-q)(1-q^2)(1-q^3)} a_1^4.
\een

\begin{prop}
Write
\be \label{eqn:Product2}
\prod_{n \geq 0} (\sum_{r\geq 0} a_r q^{nr} z^r)^{-1}
= \sum_{m \geq 0} c_m z^m,
\ee
where $a_0=1$, $\{a_r\}_{r\geq 1}$ are formal variables,
then each $c_m$ is a weighted homogeneous polynomial of $\{a_r\}_{r\geq 1}$
if we assign $\deg a_r = r$.
Furthermore,
the sequence $\{c_m\}_{m \geq 0}$ satisfies the following recursion relations:
\bea
&& c_0 = 1, \\
&& c_m = - \frac{1}{1-q^m} \sum_{r=1}^{m} a_r c_{m-r}.
\eea
\end{prop}

\begin{proof}
The first statement can be proved by the same argument as above.
To get the second statement,
note
\ben
&&\sum_{m \geq 0} c_m z^m  = \prod_{n \geq 0} (\sum_{r\geq 0} a_r q^{nr} z^r)^{-1}
= (\sum_{r\geq 0} a_r q^{nr} z^r)^{-1} \cdot \prod_{n \geq 0} (\sum_{r\geq 0} a_r q^{nr} (qz)^r)^{-1} \\
& = & (\sum_{r\geq 0} a_r z^r)^{-1} \cdot \sum_{m \geq 0} q^mc_m z^m,
\een
and so
\be
\sum_{m \geq 0} q^mc_m z^m
= \sum_{r\geq 0} a_r z^r \cdot \sum_{m \geq 0} c_m z^m,
\ee
therefore,
\be
q^m c_m = \sum_{r=0}^m a_r c_{m-r}.
\ee
\end{proof}

The following are the first few terms of $c_m$:
\ben
c_1 & = & -\frac{1}{1-q} a_1, \\
c_2 & = & -\frac{1}{1-q^2}a_2 - \frac{1}{(1-q)(1-q^2)} a_1^2, \\
c_3 & = & -\frac{1}{1-q^3}a_3 - \frac{q+2}{(1-q^2)(1-q^3)}a_2a_1
- \frac{1}{(1-q)(1-q^2)(1-q^3)} a_1^3, \\
c_4 & = & - \frac{1}{1-q^4}a_4 - \frac{q^2+q+2}{(1-q^3)(1-q^4)}a_3a_1
- \frac{1}{(1-q^2)(1-q^4)} a_2^2\\
& - & \frac{q^2+2q}{(1-q)(1-q^3)(1-q^4)} a_2a_1^2
- \frac{q^3}{(1-q)(1-q^2)(1-q^3)} a_1^4.
\een
We will write
\begin{align}
b_n & = \sum_{|\mu|=n} \beta_\mu a_\mu, &
c_n & = \sum_{|\mu|=n} \gamma_\mu a_\mu,
\end{align}
where for a partition $\mu=(\mu_1, \dots, \mu_l)$, $a_\mu:=a_{\mu_1} \cdots a_{\mu_l}$,
$\beta_\mu$ and $\gamma_\mu$ are rational functions in $q$,
with poles only at roots of unity.

With the above preparations,
we now know how to compute the series expansion of the right-hand side of the second of
the following equations:
\ben
&& \sum_{k\geq 0} e_k(\{t^{-l'(s)}q^{a'(s)}\}_{s\in \mu}) z^k
= \exp \sum_{m=1}^\infty \frac{(-1)^{m-1}}{m} p_m(\{t^{-l'(s)}q^{a'(s)}\}_{s\in \mu}) z^m \\
& = & \prod_{n\geq 0} \biggl(
 \sum_{r\geq 0} e_r(q^{\mu_1}t^{-1}, q^{\mu_2}t^{-2}, \dots) (tq^nz)^r
 \biggr)^{-1}
\cdot \prod_{n\geq 0}   \sum_{r\geq 0} e_r(t^{-1}, t^{-2}, \dots) (tq^nz)^r.
\een
We formulate the result in the following:

\begin{prop}
For a partition $\lambda$,
the following identity holds:
\be
e_k(\{t^{-l'(s)}q^{a'(s)}\}_{s\in \lambda})
= \sum_{\substack{\mu, \nu\in \cP\\|\mu|+|\nu|=k}}
\gamma_\mu \cdot e_\mu(\{q^{\lambda_i}t^{-i+1}\}_{i \geq 1})
\cdot \beta_\nu \cdot e_\nu(\{t^{-i+1}\}_{i\geq 1}).
\ee
\end{prop}

\begin{proof}
This is an easy consequence of the following expansions:
\ben
\prod_{n\geq 0} \biggl(
\sum_{r\geq 0} e_r(\{q^{\lambda_i}t^{-i+1}\}_{i \geq 1})q^{rn}z^r  \biggr)^{-1}
= \sum_{\mu\in \cP} \gamma_\mu \cdot e_\mu(\{q^{\lambda_i}t^{-i+1}\}_{i \geq 1})
\cdot z^{|\mu|},
\een

\ben
&& \prod_{n\geq 0}   \sum_{r\geq 0} e_r(\{t^{-i+1}\}_{i\geq 1}) q^{rn}z^r
= \sum_{\nu\in\cP} \beta_\nu \cdot e_\nu(\{t^{-i+1}\}_{i\geq 1}) \cdot z^{|\nu|}.
\een
\end{proof}

For example,
\ben
&& e_1(\{q^{\lambda_i}t^{-i+1}\}_{i \geq 1})q^{rn}
= \frac{1}{1-q} e_1(\{t^{-i+1}\}_{i\geq 1}) - \frac{1}{1-q} e_1(\{q^{\lambda_i}t^{-i+1}\}_{i\geq 1}), \\
&& e_2(\{q^{\lambda_i}t^{-i+1}\}_{i \geq 1})q^{rn}
= \frac{1}{1-q^2} e_2(\{t^{-i+1}\}_{i\geq 1})
+  \frac{q}{(1-q)(1-q^2)} e_1^2(\{t^{-i+1}\}_{i\geq 1}) \\
&& - \frac{1}{(1-q)^2} e_1(\{t^{-i+1}\}_{i\geq 1}) e_1(\{q^{\lambda_i}t^{-i+1}\}_{i\geq 1}) \\
&& - \frac{1}{1-q^2} e_2(\{q^{\lambda_i}t^{-i+1}\}_{i\geq 1})
-  \frac{q}{(1-q)(1-q^2)} e_1^2(\{q^{\lambda_i}t^{-i+1}\}_{i\geq 1}).
\een

One can also get:
\ben
&& \sum_{m=1}^\infty \frac{(-1)^{m}}{m} p_m(\{t^{-l'(s)}q^{a'(s)}\}_{s\in \mu}) z^m  \\
& = & \sum_{n\geq 0} \log \sum_{r\geq 0} e_r(q^{\mu_1}t^{-1}, q^{\mu_2}t^{-2}, \dots) (tq^nz)^r
- \sum_{n\geq 0} \log \sum_{r\geq 0} e_r(t^{-1}, t^{-2}, \dots) (tq^nz)^r \\
& = & \sum_{n \geq 0} \sum_{\lambda \in \cP} \alpha_\lambda
\biggl( e_\lambda(q^{\mu_1}t^{-1}, q^{\mu_2}t^{-2}, \dots)
- e_\lambda(t^{-1}, t^{-2}, \dots) \biggr) \cdot (tq^nz)^{|\lambda|} \\
& = & \sum_{\lambda \in \cP} \alpha_\lambda
\biggl( e_\lambda(q^{\mu_1}t^{-1}, q^{\mu_2}t^{-2}, \dots)
- e_\lambda(t^{-1}, t^{-2}, \dots) \biggr) \cdot \frac{(tz)^{|\lambda|}}{1-q^{|\lambda|}}.
\een
Note we have
\ben
&& \sum_{m \geq 0} \sum_{|\lambda|=m} \alpha_\lambda \cdot
e_\lambda(t^{-1}, t^{-2}, \dots) z^m
= \log \sum_{m=0}^\infty e_m(t^{-1}, t^{-2}, \dots) z^m \\
& = & \log \prod_{n \geq 1} (1+t^{-n} z)  = \sum_{n \geq 1}\log  (1+t^{-n} z) \\
& = & \sum_{m=1}^\infty \frac{(-1)^{m-1}}{m} \sum_{n=1}^\infty t^{-mn} z^m
= \sum_{m=1}^\infty \frac{(-1)^{m-1}}{m} \frac{t^{-m}}{1-t^{-m}} z^m.
\een
Therefore,
we get the following generalization of \eqref{eqn:Arm-Leg}:

\begin{prop}
For a partition $\mu$ and $m \geq 1$,
the following identity holds:
\be \label{eqn:pm-in-e}
\begin{split}
&   p_m(\{t^{-l'(s)}q^{a'(s)}\}_{s\in \mu})  \\
= & \frac{(-1)^{m}mt^{m}}{1-q^{m}} \cdot \sum_{|\lambda|=m} \alpha_\lambda \cdot
e_\lambda(q^{\mu_1}t^{-1}, q^{\mu_2}t^{-2}, \dots)
+ \frac{1}{(1-q^m)(1-t^{-m})},
\end{split}
\ee
where $\alpha_\lambda$ is defined in \eqref{eqn:alpha-lambda}.
\end{prop}

For example,
\be \label{eqn:p2-in-e}
\begin{split}
& p_2(\{t^{-l'(s)}q^{a'(s)}\}_{s\in \mu})  \\
= & \frac{2t^{2}}{1-q^{2}} \cdot \biggl(
e_2(\{q^{\mu_i}t^{-i}\}_{i\geq 1})
- \half e_1^2(\{q^{\mu_i}t^{-i}\}_{i\geq 1}) \biggr)
+ \frac{1}{(1-q^2)(1-t^{-2})}.
\end{split}
\ee

In the above
we have related $e_m(\{t^{-l'(s)}q^{a'(s)}\}_{s\in \mu})$
and $p_m(\{t^{-l'(s)}q^{a'(s)}\}_{s\in \mu})$  to the eigenvalues
of the operators $\tilde{E}^r$.
We can also relate them to the eigenvalues of the operators $E^r$.
Recall the generating series of the eigenvalues $e^r(\mu)$
of  the operator  $E^r$ is given by \eqref{eqn:Eigen-E}:
\be
 \sum_{r=0}^\infty e^r(\mu) z^r
 = \prod_{j=1}^\infty  \frac{1+q^{\mu_j}t^{-j}z}{1+t^{-j}z}.
\ee
From the above calculation we have:
\ben
&& \sum_{k\geq 0} e_k(\{t^{-l'(s)}q^{a'(s)}\}_{s\in \mu}) z^k
= \exp \sum_{m=1}^\infty \frac{(-1)^{m-1}}{m} p_m(\{t^{-l'(s)}q^{a'(s)}\}_{s\in \mu}) z^m \\
& = & \prod_{n\geq 0} \biggl(
\sum_{r\geq 0} e_r(q^{\mu_1}t^{-1}, q^{\mu_2}t^{-2}, \dots) (tq^nz)^r  \biggr)^{-1}
\cdot \prod_{n\geq 0} \prod_{j=1}^\infty (1+t^{1-j}q^nz) \\
& = & \prod_{n\geq 0} \biggl[
\prod_{j \geq 1} [(1+q^{\mu_j} t^{-j} \cdot  (tq^nz))^{-1}(1+t^{-j}\cdot (tq^nz)) ]
\biggr] \\
& = & \prod_{n \geq 0} (\sum_{r \geq 0} e^r(\mu) (tq^nz)^r)^{-1}
\cdot \prod_{n\geq 0} \prod_{j=1}^\infty (1+t^{1-j}q^nz).
\een
It is then natural to consider:
\be
\begin{split}
& \biggl( \sum_{k\geq 0} e_k(\{t^{-l'(s)}q^{a'(s)}\}_{s\in \mu}) z^k\biggr)^{-1}
=  \prod_{n\geq 0}
\prod_{j \geq 1} \frac{1+q^{\mu_j} t^{-j} \cdot  (tq^nz)}{1+t^{-j}\cdot (tq^nz)} \\
& = \prod_{n\geq 0}
 \sum_{r\geq 0} e_r(q^{\mu_1}t^{-1}, q^{\mu_2}t^{-2}, \dots) (tq^nz)^r
 \cdot \prod_{m, n\geq 0} \frac{1}{1+t^{-m} q^n z}.
\end{split}
\ee
Recall the following relationship between elementary symmetric functions $\{e_k\}$
and complete symmetric functions $\{h_k\}$:
\be
(\sum_{k\geq 0} e_k z^k)^{-1} = \sum_{k\geq 0} (-1)^k h_k z^k.
\ee

\begin{prop}
For a partition $\lambda$,
the following identity holds:
\be
h_k(\{t^{-l'(s)}q^{a'(s)}\}_{s\in \lambda})
= (-1)^k\sum_{\substack{\mu, \nu\in \cP\\|\mu|+|\nu|=k}}
\beta_\mu e_\mu(\{q^{\lambda_i}t^{-i+1}\}_{i \geq 1})
\gamma_\nu  e_\nu(\{t^{-i+1}\}_{i\geq 1}).
\ee
\end{prop}

\begin{proof}
This is an easy consequence of the following expansions:
\ben
\prod_{n\geq 0}
\sum_{r\geq 0} e_r(\{q^{\lambda_i}t^{-i+1}\}_{i \geq 1})q^{rn}z^r
= \sum_{\mu\in \cP} \beta_\mu \cdot e_\mu(\{q^{\lambda_i}t^{-i+1}\}_{i \geq 1})
\cdot z^{|\mu|},
\een

\ben
&& \prod_{n\geq 0}  \biggl( \sum_{r\geq 0} e_r(\{t^{-i+1}\}_{i\geq 1}) q^{rn}z^r \biggr)^{-1}
= \sum_{\nu\in\cP} \gamma_\nu \cdot e_\nu(\{t^{-i+1}\}_{i\geq 1}) \cdot z^{|\nu|}.
\een
\end{proof}

There is another way to compute
$\prod_{n\geq 0} \prod_{j=1}^\infty (1+t^{1-j}q^nz)$.
Consider the expansion:
\ben
&&  \prod_{m,n\geq 0} (1+q_1^{m}q_2^{n}z) = \sum_{n\geq 0} c_n z^n.
\een

\ben
&& \sum_{n\geq 0} c_n z^n = \prod_{m,n \geq 0} (1+q_1^{m}q_2^{n}z) \\
& = &  \frac{1}{1+z}\prod_{m\geq 0} (1+q_1^{m}z)
\cdot \prod_{n\geq 0} (1+q_2^{n}z)
\cdot \prod_{m, n \geq 0} (1+q_1^{m}q_2^{n}(q_1q_2 z)) \\
%%%% & = & \sum_{k \geq 0} e_k(\{q_1^m\}) z^k  \cdot
%%%% \sum_{l \geq 0} e_l(\{q_2^n\}) z^l \cdot \sum_{m \geq 0} c_m (q_1q_2)^m z^m \\
& = & \sum_{l \geq 0} e_l(\{q_1^{n+1}, q_2^n\}_{n \geq 0}) z^l
\cdot \sum_{m \geq 0} c_m \cdot (q_1q_2)^m z^m \\
& = & \sum_{l, m \geq 0} e_l (\{q_1^{n+1}, q_2^n\}_{n \geq 0}) c_m \cdot  (q_1q_2)^m z^{l+m}.
\een
It follows that
\be
c_n =\sum_{l+ m=n} e_l (\{q_1^{n+1}, q_2^n\}_{n \geq 0}) (q_1q_2)^m c_m
\ee
and therefore,
\be
c_n = \frac{1}{1-(q_1q_2)^n} \sum_{l=0}^{m-1} (q_1q_2)^{m-l} e_l(\{q_1^n, q_2^n\}).
\ee

\subsection{More examples of vertex realizable operators}

Define three operator $\Lambda^m(q,t), \Sigma^m(q,t), \Psi^m(q,t): \Lambda_{q,t} \to \Lambda_{q, t}$
such that
\bea
&& \Lambda^m(q,t) P_\mu(x; q,t) =  e_m(\{t^{-l'(s)}q^{a'(s)}\}_{s\in \mu})
P_\mu(x; q,t), \\
&& \Sigma^m(q,t) P_\mu(x; q,t) =  h_m(\{t^{-l'(s)}q^{a'(s)}\}_{s\in \mu})
P_\mu(x; q,t), \\
&& \Psi^m(q,t) P_\mu(x; q,t) =  p_m(\{t^{-l'(s)}q^{a'(s)}\}_{s\in \mu})
P_\mu(x; q,t).
\eea
By the discussion in last subsection,
one easily get the following:

\begin{prop}
The operators $\Lambda^m$, $\Sigma^m$ and $\Psi^m$
can all be expressed as weighted homogeneous polynomials in
operators $\tilde{E}^r$ which are of weight $r$,
hence they are all vertex realizable.
\end{prop}

Therefore,
one can use the results for correlation functions
for $\tilde{E}^r$ to get
correlations functions of these operators.
For example, by \eqref{eqn:pm-in-e} we have
\be
\Psi^m = \frac{(-1)^{m}mt^{m}}{1-q^{m}} \cdot \sum_{|\lambda|=m} \alpha_\lambda \cdot
\tilde{E}^\lambda
+ \frac{1}{(1-q^m)(1-t^{-m})},
\ee
where for $\lambda=(\lambda_1, \dots, \lambda_l)$,
$\tilde{E}^\lambda = \tilde{E}^{\lambda_1} \cdots \tilde{E}^{\lambda_l}$.
In particular,
\bea
&& \Psi^1 = \frac{-t}{1-q} \tilde{E}^1 + \frac{1}{(1-q)(1-t^{-1})}, \\
&& \Psi^2 = \frac{2t^{2}}{1-q^{2}}\biggl(\tilde{E}^2-\half (\tilde{E}^1)^2\biggr)
+ \frac{1}{(1-q^2)(1-t^{-2})}, \\
&& \Psi^3 = \frac{-3t^{3}}{1-q^{3}}\biggl(\tilde{E}^3- \tilde{E}^2\tilde{E}^1
+ \frac{1}{3}(\tilde{E}^1)^3\biggr)
+ \frac{1}{(1-q^3)(1-t^{-3})}.
\eea
It follows that we have:
\ben
\corr{\Psi^1}'_{u,v} & = & \frac{-t}{1-q} \corr{\tilde{E}^1}'_{u,v} + \frac{1}{(1-q)(1-t^{-1})}\\
& = & \frac{-t}{1-q} \cdot \frac{t^{-1}}{1-t^{-1}} \biggl( 1- Q \frac{(1-u)(1-v)}{1- u Q} \biggr)
+ \frac{1}{(1-q)(1-t^{-1})} \\
& = & Q \frac{(1-u)(1-v)}{(1-q)(1-t^{-1})(1- u Q)}.
\een

\ben
\corr{\Psi^2}'_{u,v}
& = & -\frac{2t^{2}}{1-q^{2}}\biggl(\half \corr{(\tilde{E}^1)^2}'_{u,v}
- \corr{\tilde{E}^2}'_{u,v}\biggr)
+ \frac{1}{(1-q^2)(1-t^{-2})} \\
& = & -\frac{t^{2}}{1-q^{2}}\biggl( \biggl(\frac{t^{-1}}{1-t^{-1}}  \biggr)^2 \cdot
\biggl[\biggl( 1- Q\frac{(1-u)(1-v)}{1-uQ} \biggr)^2 \\
& + & Q\frac{(1-Q)(1-uvQ)}{(1-uQ)^2} \cdot (1-u)(1-v) \cdot
\frac{(1-q)(1-t^{-1})}{(1-uqQ)(1-ut^{-1}Q)} \biggr] \\
& - & 2 \frac{t^{-3}}{(1-t^{-1})(1-t^{-2})} \biggl[
\biggl(1- Q \frac{(1-u)(1-v)}{1-u Q}\biggr)^2 \\
& + & (1-t^{-1}) Q \frac{(1-Q)(1-u)(1-v)(1-uvQ)}{(1-t^{-1}uQ)(1-uQ)^2}
\biggr] \biggr)
+ \frac{1}{(1-q^2)(1-t^{-2})} \\
& = & \frac{1}{(1-q^2)(1-t^{-2})} - \frac{1}{(1-q^2)(1-t^{-2})}
\cdot \biggl( 1- Q\frac{(1-u)(1-v)}{1-uQ} \biggr)^2 \\
& + & \frac{-1+q+t^{-1}+qt^{-1} - 2uqt^{-1}Q}{(1-q^2)(1-t^{-2})}
\cdot \frac{Q(1-Q)(1-u)(1-v)(1-uvQ)}{(1-uQ)^2(1-uqQ)(1-ut^{-1}Q)}.
\een
We also have
\ben
\corr{(\Psi^1)^2}'_{u,v} & = & \frac{t^2}{(1-q)^2} \corr{(\tilde{E}^1)^2}'_{u,v}
- \frac{2t}{(1-q)^2(1-t^{-1})} \corr{\tilde{E}^1}'_{u,v}
+ \frac{1}{(1-q)^2(1-t^{-1})^2} \\
& = & \frac{t^2}{(1-q)^2} \cdot \biggl(\frac{t^{-1}}{1-t^{-1}}  \biggr)^2 \cdot
\biggl[\biggl( 1- Q\frac{(1-u)(1-v)}{1-uQ} \biggr)^2 \\
&+ & Q\frac{(1-Q)(1-uvQ)}{(1-uQ)^2} \cdot (1-u)(1-v) \cdot
\frac{(1-q)(1-t^{-1})}{(1-uqQ)(1-ut^{-1}Q)} \biggr]\\
& - & \frac{2t}{(1-q)^2(1-t^{-1})} \cdot
\frac{t^{-1}}{1-t^{-1}} \biggl( 1- Q \frac{(1-u)(1-v)}{1- u Q} \biggr)
+ \frac{1}{(1-q)^2(1-t^{-1})^2} \\
& = & \biggl( Q \frac{(1-u)(1-v)}{(1-q)(1-t^{-1})(1- u Q)} \biggr)^2 \\
& + & \frac{(1-u)(1-v)}{(1-q)(1-t^{-1})}
\cdot Q\frac{(1-Q)(1-uvQ)}{(1-uQ)^2(1-uqQ)(1-ut^{-1}Q)}.
\een
Since we have
\be
\Lambda^2 = \half ((\Psi^1)^2-\Psi^2),
\ee
we get the following formula:
\ben
2 \corr{\Lambda^2}'_{u,v}
& = & \biggl( Q \frac{(1-u)(1-v)}{(1-q)(1-t^{-1})(1- u Q)} \biggr)^2 \\
& + & \frac{(1-u)(1-v)}{(1-q)(1-t^{-1})}
\cdot Q\frac{(1-Q)(1-uvQ)}{(1-uQ)^2(1-uqQ)(1-ut^{-1}Q)} \\
& - & \frac{1}{(1-q^2)(1-t^{-2})} + \frac{1}{(1-q^2)(1-t^{-2})}
\cdot \biggl( 1- Q\frac{(1-u)(1-v)}{1-uQ} \biggr)^2 \\
& - & \frac{-1+q+t^{-1}+qt^{-1} - 2uqt^{-1}Q}{(1-q^2)(1-t^{-2})}
\cdot \frac{Q(1-Q)(1-u)(1-v)(1-uvQ)}{(1-uQ)^2(1-uqQ)(1-ut^{-1}Q)} \\
& = & \biggl( Q \frac{(1-u)(1-v)}{(1-q)(1-t^{-1})(1- u Q)} \biggr)^2 \\
& - & \frac{1}{(1-q^2)(1-t^{-2})} + \frac{1}{(1-q^2)(1-t^{-2})}
\cdot \biggl( 1- Q\frac{(1-u)(1-v)}{1-uQ} \biggr)^2 \\
& + & \frac{2+ 2uqt^{-1}Q}{(1-q^2)(1-t^{-2})}
\cdot \frac{Q(1-Q)(1-u)(1-v)(1-uvQ)}{(1-uQ)^2(1-uqQ)(1-ut^{-1}Q)}.
\een

\section{Applications to K-Theoretical Intersection Numbers on Hilbert Schemes}

In this Section we return to the discussions in Section 4 and the beginning of Section 5.

\subsection{Formal quantum field theory associated to $K$-theory on Hilbert schemes}
\label{sec:K-as-QFT}

Let $X$ be a regular algebraic surface.
For $F,F_1, \dots, F_N\in K(X)$,
define the following normalized correlator by
\be
\begin{split}
& \corr{\Psi^{m_1}(F_1) \cdots \Psi^{m_N}(F_N)}'_{u,v,F}\\
: = & \frac{\sum_{n \geq 0} Q^n
\chi(X, \bigotimes_{j=1}^N \psi^{m_j}(F_j^{[n]})
\otimes
\lambda_{-u}F^{[n]}\otimes \lambda_{-v} (F^{[n]})^*)}{
\sum_{n \geq 0} Q^n
\chi(X, \lambda_{-u}F^{[n]}\otimes \lambda_{-v} (F^{[n]})^*)}.
\end{split}
\ee
We can also define other correlators: When $\Psi^m$ is replaced by $\Sigma^m$ or $\Lambda^m$,
$\psi^m$ is replaced by $\sigma^m$ or $\lambda^m$.
The connected normalized correlators will be denoted by
$\corr{\cdot}_{u, v, F, c}$.

When $X$ is toric with a $T$-action and
we  work with equivariant $K$-theory,
the normalized correlators will be denoted by $\corr{\cdot}'_{u, v, F; T}$.
The connected normalized correlators will be denoted by
$\corr{\cdot}_{u, v, F, c; T}$.

\subsection{Equivariant K-theory of Hilbert schemes of $\bC^2$ as a vertex operatorial formal quantum field theory}

We rewrite \eqref{eqn:K-Int} as follows:
\begin{equation} \label{eqn:K-Int-22}
\begin{split}
& \sum_{n \geq 0} Q^n
\chi((\bC^2)^{[n]}, \bigotimes_{j=1}^N \psi^{m_j}(\xi_n^{A_j})
\otimes
\Lambda_{-u}\xi_n^A \otimes \Lambda_{-v} (\xi_n^A)^*)(t_1, t_2) \\
= & \prod_{j=1}^N e^{m_j(a_jw_1+b_jw_2)} \cdot
 \sum_{\mu} (-ut^A) Q^{|\mu|} \prod_{j=1}^N \sum_{s\in \mu} e^{m_j[l'(s)w_1+a'(s)w_2]}
 \\
& \cdot \prod_{s \in \mu}
\frac{(t_1^{l'(s)} -(ut^A)^{-1} t_2^{-a'(s)}) \cdot(t_2^{a'(s)}- (vt^{-A})t_1^{-l'(s)})}
{(1- t_1^{-l(s)} t_2^{a(s)+1})(1 - t_1^{l(s)+1}t_2^{-a(s)})}.
\end{split}
\end{equation}
After taking
\be
t_2=q,  \;\;\;\;\;t_1 =t^{-1},
\ee
we find
\begin{equation} \label{eqn:K-Int2}
\begin{split}
& \sum_{n \geq 0} Q^n
\chi((\bC^2)^{[n]}, \bigotimes_{j=1}^N \psi^{m_j}(\xi_n^{A_j})
\otimes
\Lambda_{-u}\xi_n^A \otimes \Lambda_{-v} (\xi_n^A)^*)(t_1, t_2) \\
= & \prod_{j=1}^N e^{m_j(a_jw_1+b_jw_2)} \cdot
\corr{\Psi^{m_1} \cdots \Psi^{m_N} }_{ut^A, vt^{-A}}.
\end{split}
\ee
So we have reached our main result:

\begin{theorem}
The equivariant K-theory on Hilbert schemes on $\bC^2$
can be reformulated as a vertex operatorial formal quantum field theory
and can be computed using the vertex operator realizations
of the Macdonald operators.
\end{theorem}

\subsection{Generalization to toric surfaces}

In this Subsection we extend the above result to Hilbert schemes
on a projective or quasi-projective
surface $X$ which admits a torus action with isolated fixed points.
Suppose that $T^2$ acts on $X$ with isolated fixed points
$p_1, \dots, p_m$,
such that the weights of $T_{p_i}S =t_{1,i}^{-1}+t_{2,i}^{-1}$.
The $T^2$-action on $S$ induces a natural $T^2$-action on $S^{[n]}$.
The fixed points on $X^{[n]}$ are parameterized by
$m$-tuples of partitions $(\mu^1, \dots, \mu^m)$,
where $\mu^i$ can be an empty partition,
such that
$$|\mu^1| + \cdots + |\mu^m| = n.$$
Furthermore,
the weight decomposition of the cotangent space at the fixed point is given by:
$$\sum_{i=1}^m \sum_{s^i \in \mu^i} (t_{1,i}^{l(s^i)} t_{2,i}^{-(a(s^i)+1)}
+ t_{1,i}^{-(l(s^i)+1)}t_{2,i}^{a(s^i)}),$$
Suppose that $L$ is an equivariant holomorphic line bundle on $S$ such that
\begin{equation}
L|_{p_i} = t^{A^i} = t_1^{a^i_1}t_2^{a^i_2}.
\end{equation}
The weights of $L^{[n]}$ are:
$$\sum_{i=1}^m t^{A^i} \sum_{s^i \in \mu^i} t_{1,i}^{l'(s^i)}t_{2,i}^{a'(s^i)}.$$
By holomorphic Lefschetz formula one then gets:

\begin{theorem}
Let $X$ be a surface as above and let $L, L_1, \dots, L_N$ be holomorphic line bundles such that
\begin{align}
L|_{p_i} & = t^{A^i} = t_1^{a^i_1}t_2^{a^i_2}, &
L_j|_{p_i} & = t^{A_j^i} = t_1^{a^i_{1,j}}t_2^{a^i_{j,2}}.
\end{align}
Then we have
\begin{equation} \label{eqn:Exterior}
\begin{split}
& \sum_{n \geq 0} Q^n
\chi(X^{[n]}, \bigotimes\limits_{j=1}^N \Lambda_{x_j} L_j^{[n]}
\otimes
\Lambda_{-u}L^{[n]} \otimes \Lambda_{-v}L^{[n]*})(t_1, t_2)  \\
= & \prod_{i=1}^m \sum_{\mu^i} Q^{|\mu^i|}
\prod_{s^i \in \mu^i} \prod_{j=1}^N (1 + x_j t^{A_j^i}  t_{1,i}^{l'(s^i)}t_{2,i}^{a'(s^i)}) \\
& \cdot \prod_{s^i \in \mu^i}
\frac{(1-u t^{A^i} t_{1,i}^{l'(s^i)}t_{2,i}^{a'(s^i)})(1-v t^{-A^i} t_{1,i}^{-l'(s^i)}t_{2,i}^{-a'(s^i)})}
{(1- t_{1,i}^{-l(s^i)} t_{2,i}^{a(s^i)+1})(1 - t_{1,i}^{l(s^i)+1}t_{2,i}^{-a(s^i)})}, \\
\end{split} \end{equation}
and similarly for symmetric powers,
\begin{equation} \label{eqn:Symmetric}
\begin{split}
& \sum_{n \geq 0} Q^n
\chi(X^{[n]}, \bigotimes\limits_{j=1}^N S_{x_j} L_j^{[n]}
\otimes
\Lambda_{-u}L^{[n]} \otimes \Lambda_{-v}L^{[n]*})(t_1, t_2)  \\
= & \prod_{i=1}^m \sum_{\mu^i} Q^{|\mu^i|}
\prod_{s^i \in \mu^i} \prod_{j=1}^N \frac{1}{1 - x_j t^{A_j^i}  t_{1,i}^{l'(s^i)}t_{2,i}^{a'(s^i)}} \\
& \cdot \prod_{s^i \in \mu^i}
\frac{(1-u t^{A^i} t_{1,i}^{l'(s^i)}t_{2,i}^{a'(s^i)})(1-v t^{-A^i} t_{1,i}^{-l'(s^i)}t_{2,i}^{-a'(s^i)})}
{(1- t_{1,i}^{-l(s^i)} t_{2,i}^{a(s^i)+1})(1 - t_{1,i}^{l(s^i)+1}t_{2,i}^{-a(s^i)})} .
\end{split} \end{equation}
\end{theorem}

\subsection{Reduction of $K$-theory on Hilbert schemes to vertex operatorial
formal quantum field theory}
By this Theorem,
we can reduce the $K$-theoretical intersection numbers on Hilbert schemes
to the special formal quantum field theory calculations we develop in last Section.
The strategy, as in Wang-Zhou \cite{Wang-Zhou1},
is to apply the results of Ellingsrud-G\"ottsche-Lehn \cite{EGL} to reduce
first to toric surfaces then to $\bC^2$.
We illustrate the idea by some examples in this Subsection.

If we take the coefficient of $x_1$ on both sides of \eqref{eqn:Exterior},
we get the following equations for toric surface $X$ and equivariant line bundle $L, L_1$:
\begin{equation*}
\begin{split}
& \frac{ \sum_{n \geq 0} Q^n
\chi(X^{[n]},   L_1^{[n]}
\otimes
\Lambda_{-u}L^{[n]} \otimes \Lambda_{-v}L^{[n]*})(t_1, t_2)}
{\sum_{n \geq 0} Q^n
\chi(X^{[n]},
\Lambda_{-u}L^{[n]} \otimes \Lambda_{-v}L^{[n]*})(t_1, t_2)}  \\
= & \sum_{i=1}^m \frac{  \sum\limits_{\mu^i} Q^{|\mu^i|}
\sum\limits_{s^i \in \mu^i} t^{A_1^i}  t_{1,i}^{l'(s^i)}t_{2,i}^{a'(s^i)}
\cdot
 \prod\limits_{s^i \in \mu^i}
\frac{(1-u t^{A^i} t_{1,i}^{l'(s^i)}t_{2,i}^{a'(s^i)})(1-v t^{-A^i} t_{1,i}^{-l'(s^i)}t_{2,i}^{-a'(s^i)})}
{(1- t_{1,i}^{-l(s^i)} t_{2,i}^{a(s^i)+1})(1 - t_{1,i}^{l(s^i)+1}t_{2,i}^{-a(s^i)})} }
{\sum_{\mu^i} Q^{|\mu^i|}
 \prod_{s^i \in \mu^i}
\frac{(1-u t^{A^i} t_{1,i}^{l'(s^i)}t_{2,i}^{a'(s^i)})(1-v t^{-A^i} t_{1,i}^{-l'(s^i)}t_{2,i}^{-a'(s^i)})}
{(1- t_{1,i}^{-l(s^i)} t_{2,i}^{a(s^i)+1})(1 - t_{1,i}^{l(s^i)+1}t_{2,i}^{-a(s^i)})} } \\
= & \sum_{i=1}^m t^{A_1^i} \corr{\Lambda^1}'_{ut^{A^i}, vt^{-A^i};t_{1,i}, t_{2,i}} \\
= & Q \sum_{i=1}^m t^{A_1^i}   \frac{(1-u)(1-v)}{(1-uQ)(1- t_{1,i})(1- t_{2,i})} \\
= & \frac{(1-u)(1-v) Q}{1-uQ} \cdot \chi(S, L_1)(t_1, t_2).
\end{split}
\end{equation*}
Here $\corr{\cdot}'_{u,v; t_{1,i}, t_{2,i}}$ means the following specialization of
the parameters $q$ and $t$:
\be
q= t_{2,i},  \;\;\;\;\; t= t_{1,i}^{-1},
\ee
After taking the nonequivariant limit,
\begin{equation*}
\begin{split}
& \frac{ \sum_{n \geq 0} Q^n
\chi(X^{[n]},   L_1^{[n]} \otimes
\Lambda_{-u}L^{[n]} \otimes \Lambda_{-v}L^{[n]*})}
{\sum_{n \geq 0} Q^n
\chi(X^{[n]},
\Lambda_{-u}L^{[n]} \otimes \Lambda_{-v}L^{[n]*})}
= \frac{(1-u)(1-v) Q}{1-uQ} \cdot \chi(X, L_1).
\end{split}
\end{equation*}
Combined with the following formula proved in Wang-Zhou \cite{Wang-Zhou1}:
\begin{equation}
 \sum_{n \geq 0} Q^n
\chi(X^{[n]}, \Lambda_{-u}L^{[n]} \otimes \Lambda_{-v} L^{[n]*})
=  \exp(\sum_{n=1}^{\infty} \frac{Q^n}{n}
\chi(X, \Lambda_{-u^n}L\otimes \Lambda_{-v^n}L^*),
\end{equation}
we then get:
\begin{equation}
\begin{split}
 \sum_{n \geq 0} Q^n
& \chi(X^{[n]},  L_1^{[n]} \otimes \Lambda_{-u}L^{[n]} \otimes \Lambda_{-v} L^{[n]*}) \\
= & \frac{(1-u)(1-v) Q}{1-uQ} \cdot \chi(X, L_1)
\cdot \exp(\sum_{n=1}^{\infty} \frac{Q^n}{n}
\chi(X, \Lambda_{-u^n}L\otimes \Lambda_{-v^n}L^*).
\end{split}
\end{equation}
By the results of Ellingsrud-G\"ottsche-Lehn \cite{EGL},
this formula holds for all projective surfaces.
In the notation of \S \ref{sec:K-as-QFT},
\be
\corr{\lambda^1(L_1)}'_{u,v,L} = \frac{(1-u)(1-v) Q}{1-uQ} \cdot \chi(X, L_1).
\ee

If we take the coefficient of $x_1x_2$ on both sides of \eqref{eqn:Exterior},
we get the following equations for toric surface $X$ and equivariant line bundle $L, L_1, L_2$:
\begin{equation*}
\begin{split}
& \frac{ \sum_{n \geq 0} Q^n
\chi(X^{[n]},   L_1^{[n]}
\otimes L_2^{[n]} \otimes
\Lambda_{-u}L^{[n]} \otimes \Lambda_{-v}L^{[n]*})(t_1, t_2)}
{\sum_{n \geq 0} Q^n
\chi(X^{[n]},
\Lambda_{-u}L^{[n]} \otimes \Lambda_{-v}L^{[n]*})(t_1, t_2)}  \\
= & \sum_{i=1}^m t^{A_1^i}t^{A^i_2} \corr{(\Lambda^1)^2 }'_{ut^{A^i}, vt^{-A^i};t_{1,i}, t_{2,i}}
+ \sum_{1 \leq i \neq j \leq m} t^{A_1^i} \corr{\Lambda^1}'_{ut^{A^i}, vt^{-A^i}}
\cdot t^{A^j_2}\corr{\Lambda^1}'_{ut^{A^j}, vt^{-A^j}; t_{1, i}, t_{2,i}} \\
= &  \sum_{i=1}^n t^{A_1^i}t^{A^i_2} \biggl[ \biggl( Q \frac{(1-u)(1-v)}{(1-t_{1,i})(1-t_{2,i})(1- u Q)} \biggr)^2 \\
+ & \frac{(1-u)(1-v)}{(1-t_{1,i})(1-t_{2,i})}
\cdot Q\frac{(1-Q)(1-uvQ)}{(1-uQ)^2(1-ut_{1,i}Q)(1-ut_{2,i}Q)} \biggr] \\
+ & \sum_{1 \leq i \neq j \leq m}
  t^{A_1^i} t^{A_2^j} Q\frac{(1-u)(1-v)}{(1- t_{1,i})(1- t_{2,i})(1-uQ)}
\cdot Q \frac{(1-u)(1-v)}{(1- t_{1,j})(1- t_{2,j})(1-uQ)} \\
= & \sum_{i=1}^m  t^{A_1^i} Q \frac{(1-u)(1-v)}{(1- t_{1,i})(1- t_{2,i})(1- u Q)}
\cdot  \sum_{j=1}^m t^{A_2^j} Q \frac{(1-u)(1-v)}{(1- t_{1,j})(1- t_{2,j})(1- u Q)}  \\
+ & \sum_{i=1}^m  \frac{(1-u)(1-v)}{(1- t_{1,j})(1- t_{2,j})}
\cdot Q\frac{(1-Q)(1-uvQ)t^{A_1^i}t^{A^i_2} }{(1-uQ)^2(1-ut_{1,j}Q)(1-ut_{2,j}Q)} \\
= & \prod_{j=1}^2 \biggl( \frac{Q(1-u)(1-v)}{1-uQ} \chi(X,L_j)(t_1, t_2) \biggr) \\
+ & (1-u)(1-v)
\cdot Q\frac{(1-Q)(1-uvQ)}{(1-uQ)^2} \chi(X, L_1\otimes L_2 \otimes S_{uQ}T^*X)(t_1, t_2).
\end{split}
\end{equation*}
After taking the nonequivariant limit:
\begin{equation*}
\begin{split}
& \frac{ \sum_{n \geq 0} Q^n
\chi(X^{[n]},   L_1^{[n]}
\otimes L_2^{[n]} \otimes
\Lambda_{-u}L^{[n]} \otimes \Lambda_{-v}L^{[n]*})}
{\sum_{n \geq 0} Q^n
\chi(X^{[n]},
\Lambda_{-u}L^{[n]} \otimes \Lambda_{-v}L^{[n]*})}  \\
= & \prod_{j=1}^2 \biggl( \frac{Q(1-u)(1-v)}{1-uQ} \chi(X,L_j)\biggr) \\
+ & (1-u)(1-v)
\cdot Q\frac{(1-Q)(1-uvQ)}{(1-uQ)^2} \chi(X, L_1\otimes L_2 \otimes S_{uQ}T^*X).
\end{split}
\end{equation*}
In the notation of \S \ref{sec:K-as-QFT},
\be \begin{split}
& \corr{\lambda^1(L_1)\lambda^1(L_2)}'_{u,v,L}
= \prod_{j=1}^2 \biggl( \frac{Q(1-u)(1-v)}{1-uQ} \chi(X,L_j) \biggr) \\
& \;\;\;\;\;\; +  (1-u)(1-v)
\cdot Q\frac{(1-Q)(1-uvQ)}{(1-uQ)^2} \chi(X, L_1\otimes L_2 \otimes S_{uQ}T^*X) \\
& \;\;\;\;\;\; = \corr{\lambda^1(L_1)}_{u,v,L}' \cdot \corr{\lambda^1(L_2)}'_{u,v,L} )  \\
& \;\;\;\;\;\;+(1-u)(1-v)
\cdot Q\frac{(1-Q)(1-uvQ)}{(1-uQ)^2} \chi(X, L_1\otimes L_2 \otimes S_{uQ}T^*X).
\end{split} \ee
Therefore,
the following formula for connected corrlator holds:
\be \begin{split}
& \corr{\lambda^1(L_1)\lambda^1(L_2)}'_{u,v,L,c}  \\
& \;\;\;\;\;\;=(1-u)(1-v)
\cdot Q\frac{(1-Q)(1-uvQ)}{(1-uQ)^2} \cdot \chi(X, L_1\otimes L_2 \otimes S_{uQ}T^*X).
\end{split}
\ee

If we take the coefficient of $x_1^2$ on both sides of the above formulas,
we get:
\begin{equation*}
\begin{split}
& \frac{ \sum_{n \geq 0} Q^n
\chi(S^{[n]},   \Lambda^2L_1^{[n]} \otimes
\Lambda_{-u}L^{[n]} \otimes \Lambda_{-v}L^{[n]*})(t_1, t_2)}
{\sum_{n \geq 0} Q^n
\chi(S^{[n]},
\Lambda_{-u}L^{[n]} \otimes \Lambda_{-v}L^{[n]*})(t_1, t_2)}  \\
= & \sum_{i=1}^n (t^{A_1^i})^2 \corr{\Lambda^2 }'_{ut^{A^i}, vt^{-A^i}}
+ \sum_{1 \leq i < j \leq n} t^{A_1^i} \corr{\Lambda^1}'_{ut^{A^i}, vt^{-A^i}}
\cdot t^{A_1^j} \corr{\Lambda^1}'_{ut^{A^j}, vt^{-A^j}} \\
= &  \half \sum_{i=1}^n (t^{A_1^i})^2 \biggl[
\biggl( Q \frac{(1-u)(1-v)}{(1-t_{1,i})(1-t_{2,i})(1- u Q)} \biggr)^2 \\
- & \frac{1}{(1-t_{1,i}^2)(1-t_{2,i}^2)} + \frac{1}{(1-t_{1,i}^2)(1-t_{2,i}^2)}
\cdot \biggl( 1- Q\frac{(1-u)(1-v)}{1-uQ} \biggr)^2 \\
+ & \frac{2+ 2uqt^{-1}Q}{(1-t_{1,i}^2)(1-t_{2,i}^2)}
\cdot \frac{Q(1-Q)(1-u)(1-v)(1-uvQ)}{(1-uQ)^2(1-ut_{1,i}Q)(1-ut_{2,i}Q)} \\
+ & \frac{1}{(1-t_{1,i})(1-t_{2,i})}
\cdot  \frac{Q(1-Q)(1-u)(1-v)(1-uvQ)}{(1-uQ)^2(1-ut_{1,i}Q)(1-ut_{2,i}Q)} \biggr] \\
+ & \sum_{1 \leq i < j \leq n}
  t^{A_1^i} t^{A_1^j} \frac{Q(1-u)(1-v)}{(1- t_{1,i} )(1-t_{2,i} )(1-uQ)}
\cdot \frac{Q(1-u)(1-v)}{(1- t_{1,j} )(1-t_{2,j} )(1-uQ)} \\
= & \half \biggl( \sum_{i=1}^n Q t^{A_1^i} \frac{(1-u)(1-v)}{(1-t_{1,i})(1-t_{2,i})(1- u Q)} \biggr)^2 \\
+ & \half \sum_{i=1}^n (t^{A_1^i})^2 \biggl[
\biggl( -\frac{1}{(1-t_{1,i}^2)(1-t_{2,i}^2)} + \frac{1}{(1-t_{1,i}^2)(1-t_{2,i}^2)}
\cdot \biggl( 1- Q\frac{(1-u)(1-v)}{1-uQ} \biggr)^2 \\
+ & \frac{2+ 2ut_{i,1}t_{i,2}Q}{(1-t_{1,i}^2)(1-t_{2,i}^2)}
\cdot \frac{Q(1-Q)(1-u)(1-v)(1-uvQ)}{(1-uQ)^2(1-ut_{1,i}Q)(1-ut_{2,i}Q)} \\
+ & \frac{1}{(1-t_{1,i})(1-t_{2,i})}
\cdot  \frac{Q(1-Q)(1-u)(1-v)(1-uvQ)}{(1-uQ)^2(1-ut_{1,i}Q)(1-ut_{2,i}Q)} \biggr].
\end{split}
\end{equation*}
We rewrite it in the following form:
\ben
&& \frac{ \sum_{n \geq 0} Q^n
\chi(S^{[n]},   \Lambda^2L_1^{[n]} \otimes
\Lambda_{-u}L^{[n]} \otimes \Lambda_{-v}L^{[n]*})(t_1, t_2)}
{\sum_{n \geq 0} Q^n
\chi(S^{[n]},
\Lambda_{-u}L^{[n]} \otimes \Lambda_{-v}L^{[n]*})(t_1, t_2)}  \\
&= & \half \biggl( Q  \frac{(1-u)(1-v)}{(1- u Q)} \cdot \chi(X, L_1)(t_1, t_2)\biggr)^2 \\
& - & \half \chi(X, L_1)(t_1^2, t_2^2)
 + \half \biggl( 1- Q\frac{(1-u)(1-v)}{1-uQ} \biggr)^2
\chi(X, L_1)(t_1^2, t_2^2) \\
& + & \frac{Q(1-Q)(1-u)(1-v)(1-uvQ)}{(1-uQ)^2} \\
&&  \cdot \chi(X,  L_1^2\otimes (\cO_X + uQ K_X)
\otimes S_\epsilon(T^*X)\otimes S_{uQ}(T^*X))(t_1, t_2) |_{\epsilon=-1} \\
& + & \frac{1}{2}
\frac{Q(1-Q)(1-u)(1-v)(1-uvQ)}{(1-uQ)^2} \chi(X, L_1^2 \otimes S_{uQ}(T^*X))(t_1, t_2).
\een
By taking nonequivariant limit one then gets:
\ben
&& \frac{ \sum_{n \geq 0} Q^n
\chi(S^{[n]},   \Lambda^2L_1^{[n]} \otimes
\Lambda_{-u}L^{[n]} \otimes \Lambda_{-v}L^{[n]*})}
{\sum_{n \geq 0} Q^n
\chi(S^{[n]},
\Lambda_{-u}L^{[n]} \otimes \Lambda_{-v}L^{[n]*})}  \\
&= & \half \biggl( Q  \frac{(1-u)(1-v)}{(1- u Q)} \cdot \chi(X, L_1)\biggr)^2 \\
& - & \half \chi(X, L_1)
 + \half \biggl( 1- Q\frac{(1-u)(1-v)}{1-uQ} \biggr)^2
\chi(X, L_1) \\
& + & \frac{Q(1-Q)(1-u)(1-v)(1-uvQ)}{(1-uQ)^2} \\
&&  \cdot \chi(X,  L_1^2\otimes (\cO_X + uQ K_X)
\otimes S_\epsilon(T^*X)\otimes S_{uQ}(T^*X))|_{\epsilon=-1} \\
& + & \frac{1}{2}
\frac{Q(1-Q)(1-u)(1-v)(1-uvQ)}{(1-uQ)^2} \chi(X, L_1^2 \otimes S_{uQ}(T^*X)).
\een
This gives a formula to compute $\corr{\lambda^2(L_1)}'_{u,v,L}$.

\section{Concluding Remarks}

In this paper we have made a first step towards reformulating
the K-theoretical intersection theory on Hilbert schemes of points
as a quantum field theory.
We have not attempted to construct a quantum field theory
as physicists do,
but instead reformulate it as a formal quantum field theory.
This not only facilitates the computations of these intersection number,
but also it suffices as a portal to further investigations
as a physicist normally will do to an authentic quantum field theory.
One can at least consider the following four types of problems.

Problem 1. Computations and properties of the correlators.
There are two directions to pursue for this question.
Symmetries intrinsic to a quantum field theory lead to symmetries
of the correlators, in the form of what physicists called Ward identities.
Such symmetries can also manifest themselves in the form of an integrable
hierarchy of which the partition function is a tau-function.
Instead of compute the correlators for the formal quantum field theory associated with K-theory
directly as in this paper,
one can search for the constraints or the integrable hierarchies satisfied by them.

Problem 2. Finding other field theories dual to our theory in this paper.
At least in the case of Hilbert schemes of $\bC^2$,
our theory is expected to be dual to a refined topological string theory on the
resolved conifold $\cO_{\bP^1}(-1) \oplus \cO_{\bP^1}(-1)$.
It is interesting to extend to the theories on toric surfaces to see whether they
are related to some refined topological string theory on some toric Calabi-Yau 3-folds.

Problem 3. Study of relationship between Problem 1 and Problem 2.
In Gromov-Witten theory
toric Calabi-Yau 3-folds are conjecturally mirror to some plane curves on which after quantization
one can define constrains and integrable hierarchies \cite{ADKMV}.
We expect this can be generalized to our theory defined using $K$-theory on Hilbert schemes.

Problem 4. Study the relationship to Cherednik's double affine Hecke algebras.
It will be interesting to relate tautological sheaves 
to the work of Gordon and Stafford \cite{Gor-Sta1, Gor-Sta2}. 
We expect a quantized version of the work by Costello-Gronjowski \cite{Cos-Gro}
describes the equivariant $K$-theory of the Hilbert schemes of points in $\bC^2$.
It should be related to but different from the quantum differential operators
discovered by Okounkov and Pandhariapnde \cite{Oko-Pan2}.

We expect that works in the field of noncommutative algebraic geometry
will be very helpful in solving these problems.

\vspace{.2in}

{\em Acknowledgements}.
The work in this paper was started  during the author's attendance of a conference
on quantum K-theory and the stay after that at Sun Yat-Sen University in January 2018.
The author thanks Professor Jianxun Hu and Professor Changzheng Li and other colleagues there
for the invitation and the hospitality. He also thanks
Professor Weiping Li, Professor Siqi Liu and Professor Yongbin Ruan for
stimulating my old interest in K-theory and Hilbert schemes.
In particular some discussions with Professor Li on Chern characters of tautological sheaves
got this work started.
The author is partly supported by NSFC grant 11661131005.

\end{document}